\documentclass[10pt, reqno]{amsart}
\usepackage{hyperref}
\usepackage{amsmath}
\usepackage{amscd}
\usepackage{amsthm}
\usepackage{amssymb}
\usepackage{latexsym}
\usepackage{eufrak}
\usepackage{euscript}
\usepackage{epsfig}
\usepackage{graphics}
\usepackage{array}
\usepackage{enumerate}
\usepackage{arydshln}

\usepackage{color}

\usepackage{amsfonts}
\usepackage{mathrsfs}
\usepackage{indentfirst}

\newcommand{\R}{{\mathbb R}}

\newtheorem{thm}{Theorem}[section]
\newtheorem{coro}[thm]{Corollary}
\newtheorem{lemma}[thm]{Lemma}
\newtheorem{pro}[thm]{Proposition}

\theoremstyle{definition}
\newtheorem{definition}[thm]{Definition}
\newtheorem{example}[thm]{Example}
\newtheorem{remark}[thm]{Remark}

\numberwithin{equation}{section}

\newcommand{\Hmm}[1]{\leavevmode{\marginpar{\tiny%
$\hbox to 0mm{\hspace*{-0.5mm}$\leftarrow$\hss}%
\vcenter{\vrule depth 0.1mm height 0.1mm width \the\marginparwidth}%
\hbox to 0mm{\hss$\rightarrow$\hspace*{-0.5mm}}$\\\relax\raggedright
#1}}}

\begin{document}
\title[Higher order Buser inequalities]{Curvature and higher order Buser inequalities for the graph connection Laplacian}

\author{Shiping Liu}
\address{Department of Mathematical Sciences, Durham University, DH1 3LE Durham, United Kingdom}
\email{shiping.liu@durham.ac.uk}

\author{Florentin M\"unch}
\address{Institut f\"{u}r Mathematik,
Universit\"{a}t Potsdam,
14476 Potsdam, Germany}
\email{florentin.muench@uni-jena.de}

\author{Norbert Peyerimhoff}
\address{Department of Mathematical Sciences, Durham University, DH1 3LE Durham, United Kingdom}
\email{norbert.peyerimhoff@durham.ac.uk}

\begin{abstract}
We study the eigenvalues of the connection Laplacian on a graph with an orthogonal group or unitary group signature. We establish higher order Buser type inequalities, i.e., we provide upper bounds for eigenvalues in terms of Cheeger constants in the case of nonnegative Ricci curvature. In this process, we discuss the concepts of Cheeger type constants and a discrete Ricci curvature for connection Laplacians and study their properties systematically. The Cheeger constants are defined as mixtures of the expansion rate of the underlying graph and the frustration index of the signature. The discrete curvature, which can be computed efficiently via solving semidefinite programming problems, has a characterization by the heat semigroup for functions combined with a heat semigroup for vector fields on the graph.

\smallskip
\noindent \keywordsname:  connection Laplacian; Cheeger constants; discrete curvature; Buser inequality; semidefinite programming; Cartesian product.
\end{abstract}

\maketitle
\setcounter{tocdepth}{1}
\tableofcontents
\section{Introduction}

A graph structure with its Laplacian matrix provides a mathematical tool to analyze the similarities between data points: those points with large enough similarities are connected by an edge. One can also assign edge weights to quantify such similarities. In many applications, it is noticed that the representation of the data set can be vastly improved by endowing the edges of the graph additionally with linear transformations \cite{Harary53,SingerWu,BSS13}. For example, when the graph is representing a social network, we hope to attach to each edge an element from the one dimensional orthogonal group $O(1)=\{\pm 1\}$ to indicate two kinds of opposite relationships between members of the network (vertices). When the graph is representing higher dimensional data set, e.g., $2$-dimensional photos of a $3$-dimensional object from different views, one would like to assign to each edge an element of the orthogonal group $O(2)$ which optimally rotationally aligns photos when comparing their similarity (see, e.g., \cite{SingerWu,BSS13}). In theoretical research, assigning linear transformations to the edges of a graph also provides mathematical structures that have been found very useful in various topics, e.g., the study of Heawood map-coloring problem \cite{Gross74,GrossTucker74}, the construction of Ramanujan graphs \cite{BL06,MSS}, and the study of a discrete analogue of magnetic operators \cite{Sunada93,Shubin94}. The corresponding Laplacian of a graph with such an additional structure is called \emph{the connection Laplacian}, defined by Singer and Wu \cite{SingerWu}.

In fact, the connection Laplacian of a graph yields a very elegant and general mathematical framework for the analysis of massive data sets, which includes several extensively studied graph operators as particular cases, e.g., the classical Laplacian, the signless Laplacian \cite{DS09}, the Laplacian for Harary's signed graphs \cite{ZaslavskyMatrices,AtayLiu14}, and the discrete magnetic Laplacian \cite{Sunada93,Shubin94,LLPP15}.

In this paper, we study the spectra of the graph connection Laplacian, which are closely related to the geometric structure of the underlying graph with those transformations attached to its edges. We describe this geometric structure by introducing two types of quantities, Cheeger type constants and a discrete Ricci curvature. Our main theorem is concerned with higher order Buser type inequalities, showing the close relations between eigenvalues of the connection Laplacian and the Cheeger constants assuming nonnegativity of the discrete Ricci curvature. We also obtain a lower bound estimate of the first nonzero eigenvalue of the connection Laplacian by the lower Ricci curvature bound, i.e. we show a Lichnerowicz type eigenvalue estimate. In this process, the properties of the Cheeger constants and discrete Ricci curvature are explored systematically. In particular, our eigenvalues estimates help us to deepen the understanding of these two geometric quantities.

\subsection{Higher order Buser inequalities}
We now aim to state our main theorem (Theorem \ref{thm:introMain} below) more explicitly. We first introduce relevant notation. Let $G=(V,E)$ be an undirected simple finite graph with vertex set $V$ and edge set $E$. For simplicity, we restrict ourselves to unweighted $D$-regular graphs in this Introduction. Let $H$ be a group. For each edge $\{x,y\}\in E$, we assign an element $\sigma_{xy}\in H$ to it, such that
\begin{equation}\label{intro:signature}
\sigma_{yx}=\sigma_{xy}^{-1}.
\end{equation}
Actually, we are defining a map $\sigma:E^{or}\to H$, where $E^{or}:=\{(x,y), (y,x)\mid \{x,y\}\in E\}$ is the set of all oriented edges. We call $\sigma$ a \emph{signature} of the graph $G$. In this paper, we restrict the group $H$ to be the $d$ dimensional orthogonal group $O(d)$ or unitary group $U(d)$.

Then the \emph{(normalized) connection Laplacian} $\Delta^\sigma$, as a matrix, is given by
\begin{equation}\label{eq:introConnction Lap}
 \Delta^\sigma:=\frac{1}{D}A^\sigma-\mathrm{I}_{Nd},
\end{equation}
where $D$ is the (constant) vertex degree and $\mathrm{I}_{Nd}$ is a $(Nd)\times (Nd)$-identity matrix, $N$ the size of vertex set $V$, and $A^\sigma$ is the $(Nd)\times (Nd)$-matrix, blockwisely defined as
\begin{equation}
 (A^\sigma)_{xy}=\left\{
                   \begin{array}{ll}
                     0, & \hbox{if $\{x,y\}\not\in E$;} \\
                     \sigma_{xy}, & \hbox{if $(x,y)\in E^{or}$.}
                   \end{array}
                 \right.
\end{equation}
Due to (\ref{intro:signature}), $\Delta^\sigma$ is Hermitian. Hence all eigenvalues of the matrix $\Delta^{\sigma}$ are real. Note that the connection Laplacian $\Delta^\sigma$ in (\ref{eq:introConnction Lap}) is defined as a negative semidefinite matrix for our later purpose of defining the discrete curvature, due to a convention originating from Riemannian geometry. However, we still want to deal with nonnegative eigenvalues. Hence, when we speak of eigenvalues of the connection Laplacian $\Delta^{\sigma}$, we mean the eigenvalue of the matrix $-\Delta^\sigma$. They can be listed (counting multiplicity) as
\begin{equation}
0\leq \lambda_1^{\sigma}\leq \lambda_2^{\sigma}\leq  \cdots\lambda^{\sigma}_{d}\leq \cdots\leq\lambda^{\sigma}_{(N-1)d+1}\leq\lambda^{\sigma}_{(N-1)d+2}\leq \cdots\leq \lambda^{\sigma}_{Nd}\leq 2.
\end{equation}

Observe that two different signatures do not necessarily lead to different spectra. Given a function $\tau:V\to H$ and a signature $\sigma$, we consider the new signature $\sigma^\tau$ defined by
\begin{equation}
 \sigma_{xy}^\tau:=\tau(x)^{-1}\sigma_{xy}\tau(y),\,\,\forall\,(x,y)\in E^{or}.
\end{equation}
Then the corresponding connection Laplacians $\Delta^\sigma$ and $\Delta^{\sigma^\tau}$ are unitarily equivalent and hence share the same spectra. Indeed, it is easy to check that
\begin{equation}\label{intro:unitary}
 \Delta^{\sigma^\tau}=(M_\tau)^{-1}\Delta^\sigma M_\tau,
\end{equation}
where $M_\tau$ stands for the matrix given blockwisely by $(M_\tau)_{xx}:=\tau(x)$. We call the function $\tau$ a \emph{switching function}. Two signatures $\sigma$ and $\sigma'$ are said to be \emph{switching equivalent}, if there exists a switching function $\tau$ such that $\sigma'=\sigma^\tau$. It follows from (\ref{intro:unitary}), the eigenvalues of the connection Laplacian $\Delta^\sigma$ are switching invariant.

The Cheeger type constants $\{h_k^\sigma, k=1,2,\ldots, N\}$ and the discrete Ricci curvature $K_\infty(\sigma)$ that we are going to introduce are also switching invariant. A signature $\sigma$ is said to be \emph{balanced} if it is switching equivalent to the trivial signature $\sigma_{\mathrm{triv}}:E^{or}\to id\in H$. In fact, the constants $\{h_k^\sigma, k=1,2,\ldots, N\}$ are quantifying the connectivity of the graph and the unbalancedness of the signature $\sigma$. The latter is described by the \emph{frustration index} $\iota^\sigma(S)$ of the signature $\sigma$ restricted to the induced subgraph of $S\subseteq V$, with the property that
$$\iota^\sigma(S)=0\,\,\Leftrightarrow\,\,\sigma \,\,\text{restricted on $S$ is balanced}.$$
By abuse of notation, we will also use $S$ to denote its induced subgraph. Denote by $|E(S, V\setminus S)|$ the number of edges connecting $S$ and its complement $V\setminus S$. We then define
\begin{equation}
\phi^\sigma(S):=\frac{\iota^\sigma(S)+|E(S,V\setminus S)|}{D\cdot |S|},
\end{equation}
where $|S|$ is the cardinality of the set $S$. Then the Cheeger constants $h_k^\sigma$ is defined as
\begin{equation}\label{intro:Cheeger}
 h_k^\sigma:=\min_{\{S_i\}_{i=1}^k}\max_{1\leq i\leq k}\phi^\sigma(S_i),
\end{equation}
where the minimum is taken over all nonempty, pairwise disjoint subsets $\{S_i\}_{i=1}^k$ of the vertex set $V$. The above definition of Cheeger constants is a natural extension of the constants in \cite{AtayLiu14} and \cite{LLPP15}, where $H=O(1)$ and $U(1)$, respectively, and is closely related to the $O(d)$ frustration $\ell^1$ constant in \cite{BSS13} (see Remark \ref{remark:BSS} for a detailed explanation).

The nonnegativity of the discrete Ricci curvature $K_\infty(\sigma)$, or the \emph{curvature dimension inequality with a signature}, $CD^\sigma(0,\infty)$, is an extension of the classical curvature dimension inequality \`{a} la Bakry and \'{E}mery \cite{BaEm,Bakry} on graphs, which has been studied extensively in recent years, see, e.g., \cite{Schmuckenschlager98,LinYau,JostLiu14,ChungLinYau14,KKRT15,LiuPeyerimhoff14,HuaLin15}. For related notions of curvature dimension inequalities and their strong implications in establishing various Li-Yau inequalities for heat semigroups on graphs, we refer to \cite{BHLLMY13, Horn14,FM14,FM15}. The definition of $CD^\sigma(0,\infty)$ uses both the connection Laplacian $\Delta^\sigma$ and the graph Laplacian $\Delta$, capturing the structure of the graph (especially, its cycles) and the signature (especially, the signature of cycles) locally around each vertex (see Proposition \ref{pro:gamma2Matrix}). The curvature condition $CD^\sigma(0,\infty)$ can be characterized by properties of the classical heat semigroup $P_t:=e^{t\Delta}$ for functions and the heat semigroup $P_t^\sigma:=e^{t\Delta^\sigma}$ for vector fields (vector valued functions) of the underlying graph (see Theorem \ref{thm:curvature-characterization}). Another appealing feature of this curvature notion is that it can be calculated very efficiently. Indeed, calculating this curvature is equivalent to solving semidefinite programming problems.

Our main theorem is the following result.
\begin{thm}[Higher order Buser inequalities]\label{thm:introMain}
Assume that a graph $G$ with a signature $\sigma$ satisfies $CD^\sigma(0,\infty)$. Then for each natural number $1\leq k\leq N$, we have
\begin{equation}
 \lambda_{kd}^\sigma\leq 16D(kd)^2\log(2kd)(h_k^\sigma)^2.
\end{equation}
\end{thm}
Note that $\lambda_{kd}^\sigma$ should be considered as the maximal value of the group of eigenvalues $\{\lambda_{(k-1)d+1}^\sigma, \ldots, \lambda_{kd}^\sigma\}$. There are $N$ different groups of eigenvalues and $N$ Cheeger constants, correspondingly.

In 1982, Peter Buser \cite{Buser82} showed that the first nonzero eigenvalue of the Laplace-Beltrami operator on a closed Riemannian manifold is bounded from above by the square of the Cheeger constant, up to a constant involving Ricci curvature. The authors of \cite{KKRT15} extend an argument of Ledoux \cite{Ledoux04} to establish analogous Buser type estimates on graphs satisfying the classical curvature dimension inequality $CD(0,\infty)$. In fact, Theorem \ref{thm:introMain} reduces to their result (see (\ref{intro:KK}) below) up to a constant, when $k=2$, $d=1$, and the signature $\sigma$
is balanced.

Higher order Buser inequalities were first proved by Funano \cite{Funano2013} on Riemannian manifolds, and later improved in \cite{Liu14} on manifolds, and in \cite{LiuPeyerimhoff14} for graph Laplacians, via showing an eigenvalue ratio estimate. However, the method in \cite{Liu14,LiuPeyerimhoff14} does not extend to the connection Laplacian for a general signature $\sigma: E^{or}\to H=O(d)$ or $U(d)$, except for the very special case $O(1)$ (see Example \ref{example:section 7}). We discuss extensions of the methods in \cite{Liu14,LiuPeyerimhoff14} for $H=O(1)$ signatures in Section \ref{section:O(1)}. For general signatures, our proof neatly extends Ledoux's \cite{Ledoux04} argument for Buser's inequality and provides new ideas for establishing higher order Buser inequalities.

In the following sections, we explain the ingredients of Theorem \ref{thm:introMain} in more details.

\subsection{Motivation and a dual Buser inequality}

In this subsection, we briefly recall some known results about Cheeger and dual Cheeger constants of a graph $G$ and the eigenvalues of the graph Laplacian $\Delta$, which can be listed as below,
$$0=\lambda_1\leq \lambda_2\leq\cdots\leq \lambda_N\leq 2.$$
This will explain one motivation of Theorem \ref{thm:introMain} from the spectral theory of the graph Laplacian $\Delta$. Recall that $\Delta:=\frac{1}{D}A-\mathrm{I}_N$, where $A$ is the adjacency matrix of $G$, i.e.,  $\Delta$ can be viewed as the connection Laplacian with the trivial signature $\sigma_{\mathrm{triv}}:E^{or} \to 1\in O(1)$. The above $\lambda_i$'s are eigenvalue of the matrix $-\Delta$.

By the results in \cite{AtayLiu14}, we know that if we assign to $G$ the trivial $O(1)$ signature $\sigma_{\mathrm{triv}}:E^{or}\to 1\in O(1)$, then the constant $h_2^{\sigma_{\mathrm{triv}}}$ coincides with the classical Cheeger constant of $G$. If, instead, we assign to $G$ the signature $-\sigma_{\mathrm{triv}}: E^{or}\to -1\in O(1)$, then the constant $h_1^{-\sigma_{\mathrm{triv}}}$ reduces to the bipartiteness ratio of Trevisan \cite{Trevisan2012}, or to one minus the dual Cheeger constant of Bauer and Jost \cite{BJ}. For details, we refer to \cite{AtayLiu14}. In fact, we have the following relations between eigenvalues, Cheeger constants and structural properties of the underlying graph:
\begin{align*}
 \lambda_2=0\,\,&\Leftrightarrow\,\,h_2^{\sigma_{\mathrm{triv}}}=0\,\,\Leftrightarrow\,\, G\,\,\text{has at least two connected components};\\
2-\lambda_N=0\,\,&\Leftrightarrow\,\,h_1^{-\sigma_{\mathrm{triv}}}=0\,\,\Leftrightarrow\,\,G\,\,\text{has a bipartite connected component}.
\end{align*}
The Cheeger \cite{Dodziuk1984,AM1985,Alon1986} and dual Cheeger inequalities \cite{Trevisan2012,BJ} asserts that
\begin{equation*}
 \frac{(h_2^{\sigma_{\mathrm{triv}}})^2}{2}\leq \lambda_2\leq 2h_2^{\sigma_{\mathrm{triv}}} \,\,\,\,\,\text{and}\,\,\,\,\,\frac{(h_1^{-\sigma_{\mathrm{triv}}})^2}{2}\leq 2-\lambda_N\leq 2h_1^{-\sigma_{\mathrm{triv}}}.
\end{equation*}
For many purposes, it is very useful to have further relations between $\lambda_2$ ($2-\lambda_N$, resp.) and $h_2^{\sigma_{\mathrm{triv}}}$ ($h_1^{-\sigma_{\mathrm{triv}}}$, resp.). The authors of \cite{KKRT15} prove the following \emph{Buser inequality}: If $G$ satisfies the curvature dimension inequality $CD(0,\infty)$, then
\begin{equation}\label{intro:KK}
 \lambda_2\leq 16D (h_2^{\sigma_{\mathrm{triv}}})^2.
\end{equation}
The inequality $CD(0,\infty)$ is defined solely by the graph Laplacian $\Delta$: For any two functions $f,g:V\to \mathbb{R}$, we define two operators $\Gamma$ and $\Gamma_2$ as follows:
\begin{align}
&2\Gamma(f,g):=\Delta(fg)-f\Delta g-(\Delta f)g,\label{intr:Gamma}\\
&2\Gamma_2(f,g):=\Delta(\Gamma(f,g))-\Gamma(f,\Delta g)-\Gamma(\Delta f,g).\label{intr:Gamma2}
\end{align}
The graph $G$ \emph{satisfies $CD(0,\infty)$} if we have for any function $f:V\to \mathbb{R}$,
\begin{equation}
 \Gamma_2(f,f)\geq 0.
\end{equation}

In particular, every cycle graph $\mathcal{C}_N$ with $N$ vertices satisfies $CD(0,\infty)$. Moreover, we have for the graph $\mathcal{C}_N$ (see, e.g., \cite[Proposition 7.4]{Liu13}),
\begin{equation}
  (h_2^{\sigma_{\mathrm{triv}}})^2\leq \lambda_2(\mathcal{C}_N)\leq 5(h_2^{\sigma_{\mathrm{triv}}})^2,
\end{equation}
which is in line with the Cheeger inequality and Buser inequality, and also
\begin{equation}\label{intr:cycleDual}
0.3(h_1^{-\sigma_{\mathrm{triv}}})^2\leq2-\lambda_N(\mathcal{C}_N)\leq 5(h_1^{-\sigma_{\mathrm{triv}}})^2.
\end{equation}

A natural question then arises: Is there any similar generalization of the right hand side of (\ref{intr:cycleDual})? That is, we are asking for a possible \emph{dual Buser inequality} for the graph Laplacian $\Delta$.

Observe that the first eigenvalue of the connection Laplacian $\Delta^{-\sigma_{\mathrm{triv}}}$, also known as the signless Laplacian \cite{DS09},
is equal to $2-\lambda_N$. Indeed, one check that
$$-\Delta^{-\sigma_{\mathrm{triv}}}=2I_N-(-\Delta)=I_N+\frac{1}{D}A.$$
Therefore, Theorem \ref{thm:introMain} implies the following dual Buser inequality for $\Delta$.
\begin{coro}[Dual Buser inequality]
 Assume that $G$ satisfies $CD^{-\sigma_{\mathrm{triv}}}(0,\infty)$. Then we have
\begin{equation}
2-\lambda_N\leq 16(\log 2) D (h_1^{-\sigma_{\mathrm{triv}}})^2.
\end{equation}
\end{coro}
This provides a "dual" version of the Buser inequality in (\ref{intro:KK}). We like to mention that every cycle graph $\mathcal{C}_N$ also fulfills the inequality $CD^{-\sigma_{\mathrm{triv}}}(0,\infty)$.

One may guess that the inequality $CD^{-\sigma_{\mathrm{triv}}}(0,\infty)$ is defined by replacing the Laplacian $\Delta$ in (\ref{intr:Gamma}) and (\ref{intr:Gamma2}) by $\Delta^{-\sigma_{\mathrm{triv}}}$. However, this does not work. The reason is that the corresponding heat semigroup $P_t^{-\sigma_{\mathrm{triv}}}:=e^{t\Delta^{-\sigma_{\mathrm{triv}}}}$ does not possess a probability kernel (the operator $P_t^{-\sigma_{\mathrm{triv}}}$ is not even nonnegative), a property which is essential for the proofs in \cite{Ledoux04,KKRT15}. In fact, our definition of $CD^{-\sigma_{\mathrm{triv}}}(0,\infty)$ involves both matrices $\Delta$ and $\Delta^{-\sigma_{\mathrm{triv}}}$, which will be explained in the next section.

\subsection{Curvature dimension inequalities with signatures}

It actually looks more natural to define the curvature dimension inequality with a signature using both matrices $\Delta$ and $\Delta^\sigma$  when we come back to the general setting: For a $d$-dimensional signature $\sigma$, the connection Laplacian $\Delta^\sigma$, as an operator, acts on vector fields, i.e. functions $f: V\to \mathbb{K}^d$, where $\mathbb{K}=\mathbb{R}$ or $\mathbb{C}$.

\begin{definition}\label{def:introGamma}
  For any two functions $f,g: V\rightarrow \mathbb{K}^d$, we define
  \begin{equation}\label{eq:introGamma}
    2\Gamma^\sigma(f,g):=\Delta(f^T\overline{g})-f^T(\overline{\Delta^\sigma g}) - (\Delta^\sigma f)^T\overline{g},
  \end{equation}
  and
  \begin{equation}\label{eq:introGammatwo}
    2\Gamma^{\sigma}_2(f,g):=\Delta\Gamma^\sigma(f,g)-\Gamma^\sigma(f, \Delta^\sigma g)-
    \Gamma^\sigma(\Delta^\sigma f, g).
  \end{equation}
\end{definition}

Note that $\Gamma^\sigma(f,g)$ and $\Gamma^{\sigma}_2(f,g)$ are $\mathbb{K}$-valued functions on $V$.
We also write $\Gamma^{\sigma}(f):=\Gamma^{\sigma}(f,f)$ and
$\Gamma^\sigma_2(f):=\Gamma^\sigma_2(f,f)$, for short.
In (\ref{eq:introGamma}) and (\ref{eq:introGammatwo}), we use the graph Laplacian whenever we deal with a $\mathbb{K}$-valued function, and we use the graph connection Laplacian whenever we deal with a $\mathbb{K}$-vector valued function.

\begin{definition}[$CD^\sigma(K,\infty)$ inequality]\label{def:introCDineq}
Let $K\in \mathbb{R}$. We say the graph $G$ with a signature $\sigma$ satisfies the curvature dimension inequality $CD^\sigma(K,\infty)$ if we have for any vector field $f:V\to \mathbb{K}^d$ and any vertex $x\in V$,
\begin{equation}\label{eq:introCD}
 \Gamma_2^\sigma(f)(x)\geq K\Gamma^\sigma(f)(x).
\end{equation}
The \emph{precise $\infty$-dimensional Ricci curvature lower bound} $K_\infty(\sigma)$ is defined as the largest constant $K$ such that (\ref{eq:introCD}) holds.
\end{definition}

In Section \ref{section:HeatChar}, we show that the above curvature condition $CD^\sigma(0,\infty)$ can be characterized in terms of the corresponding heat semigroups $P_t:=e^{t\Delta}$ and $P_t^\sigma=e^{t\Delta^\sigma}$ as follows:
\begin{equation*}
 CD^\sigma(K,\infty)\,\,\Leftrightarrow\,\,\Gamma^\sigma(P_t^\sigma f)\leq e^{-2Kt}P_t(\Gamma^\sigma(f)), \,\,\forall\,f:V\to\mathbb{K}^d,\,\,\forall\,t\geq 0.
\end{equation*}
This is very useful for the proof of Theorem \ref{thm:introMain}.

It turns out that every graph $G$ with a signature $\sigma$ satisfies $CD^\sigma(\frac{2}{D}-1,\infty)$ (see Corollary \ref{cor:lower curvature bound}). This is shown by considering the switching invariance of $CD^\sigma(K,\infty)$ and $CD^\sigma$ inequalities of covering graphs (see Sections \ref{section:switching inv} and \ref{section:covering}). In particular, every (unweighted) cycle graph with a signature $\sigma: E^{or}\to O(d)$~or~$U(d)$ satisfies $CD^\sigma(0,\infty)$.

Given a graph $G$ and a signature $\sigma$, the curvature $K_\infty(\sigma)$ can be computed very efficiently by reformulating the $CD^\sigma(K,\infty)$ inequality as linear matrix inequalities at local neighborhoods of all vertices (see Section \ref{section:LMI}). Computing the precise Ricci curvature lower bound $K_\infty(\sigma)$ is then equivalent to solving semidefinite programming problems, for which efficient solvers exist. In particular, we derive the precise formula of $K_\infty(\sigma)$ for a triangle ($3$-cycle) graph with $\sigma: E^{or}\to U(1)$ in Section \ref{subsection:triangle}.

Moreover, the class of graphs with signatures satisfying $CD^\sigma(0, \infty)$ inequalities is rich since this curvature property is preserved by taking Cartesian products: Given two graphs $G_i=(V_i, E_i)$, $i=1,2$ with signatures $\sigma_i:E_i^{or}\to H_i=O(d_i)$ or $U(d_i)$, $i=1,2$. Denote their Cartesian product graph by $G_1\times G_2=(V_1\times V_2, E_{12})$. Let us assign a signature $\widehat{\sigma}_{12}:E_{12}^{or}\to H_1\otimes H_2$ to $G_1\times G_2$ as follows:
\begin{align}
\widehat{\sigma}_{12, (x_1,y)(x_2,y)}&:=\sigma_{1,x_1x_2}\otimes \mathrm{I}_{d_2}, \,\,\text{ for any }\,\,(x_1,x_2)\in E_1^{or}, y\in V_2;\notag\\
\widehat{\sigma}_{12, (x,y_1)(x,y_2)}&:=\mathrm{I}_{d_1}\otimes\sigma_{2,y_1y_2}, \,\,\text{ for any }\,\,(y_1,y_2)\in E_2^{or}, x\in V_1.\notag
\end{align}
Then we have the following theorem (see Theorem \ref{thm:CartesianII} and Remark \ref{remark:vertex measure}).
\begin{thm}
Let $G_i,i=1,2$ with signatures $\sigma_i,i=1,2$ satisfy $CD^{\sigma_1}(K_1,\infty)$ and $CD^{\sigma_2}(K_2,\infty)$, respectively. Then the Cartesian product graph $G_1\times G_2$ with the signature $\widehat{\sigma}_{12}$ satisfies $CD^{\widehat{\sigma}_{12}}(\frac{1}{2}\min\{K_1, K_2\},\infty).$
\end{thm}
In Appendix \ref{section:appendixCurCheeCar}, we discuss similar behavior of the curvature dimension inequality on the Cartesian product $G_1\times G_2$ when we assign to it various choices of edge weights, vertex measures and signatures.

\subsection{Frustration index via spanning trees}

The frustration index $\iota^\sigma(S)$, measuring the discrepancy of the signature $\sigma$ from being balanced when restricted to $S$, is an important ingredient for the definition of Cheeger type constants (\ref{intro:Cheeger}). In particular, for an $U(1)$ signature $\sigma:E^{or}\to U(1)$, it is defined as
\begin{equation}
 \iota^\sigma(S):=\min_{\tau:S\to U(1)}\sum_{\{x,y\}\in E_S}|\sigma_{xy}\tau(y)-\tau(x)|,
\end{equation}
where $E_S$ stands for the edges of the induced subgraph of $S$, and the minimum is taken over all switching functions on $S$. For higher dimensional signatures, we need to choose a matrix norm to define $\iota^\sigma(S)$, see Section \ref{section:CheegerSig}.

For $U(1)$ signatures, we show that there is an easier way to calculate its frustration index: We can calculate $\iota^\sigma(S)$ by taking the minimum only over a finite set of switching functions.  Let the induced subgraph of $S$ be connected. Given a spanning tree $T$ of $S$, we pick a switching function $\tau_T$ that switches the signature $\sigma$, restricted to $T$, to the trivial signature. Then we have
\begin{equation}\label{eq:introfrust}
 \iota^\sigma(S):=\min_{T\in \mathbb{T}_S}\sum_{\{x,y\}\in E_S}|\sigma_{xy}\tau_T(y)-\tau_T(x)|,
\end{equation}
where $\mathbb{T}_S$ is the set of all spanning trees of $S$, which is a finite set (see Section \ref{subsection:spanning tree}). This provides combinatorial expressions for $\iota^\sigma(S)$ and hence the Cheeger constants.

Surprisingly, such a simplified expression (\ref{eq:introfrust}) of $\iota^\sigma(S)$ becomes false for signatures with dimension $\geq 2$. We present a counterexample in Appendix \ref{section:appendixSpanningTree}.

Frustration indices and hence Cheeger constants behave well under taking Cartesian products as in the case of curvature dimension inequalities. This is discussed in Appendix \ref{section:appendixCurCheeCar}.

\subsection{Lichnerowicz inequality and the jump of the curvature}
Let $\lambda^\sigma$ be the first nonzero eigenvalue of the connection Laplacian $\Delta^\sigma$. Suppose the graph $G$ is connected. We observed that when $\sigma$ is unbalanced, $\lambda_1^\sigma\neq 0$, and hence $\lambda^\sigma=\lambda^\sigma_1$. Moreover, when $\iota^\sigma(V)$ becomes very small, i.e., when $\sigma$ is very close to be balanced, $\lambda^\sigma=\lambda^\sigma_1$ becomes very close to $0$. Once $\sigma$ becomes balanced, $\lambda_1^\sigma=0$, and $\lambda^\sigma=\lambda_2^\sigma>0$. We say that the quantity $\lambda^\sigma$ \emph{jumps} when $\sigma$ becomes balanced.

We show the following Lichnerowicz type eigenvalue estimate in Section \ref{section:Lichnerowicz}.
\begin{thm}[Lichnerowicz inequality]\label{thm:introLic} Assume that the graph $G$ with a signature $\sigma$ satisfies $CD^\sigma(K,\infty)$. Then the first nonzero eigenvalue $\lambda^\sigma$ satisfies
\begin{equation}
 \lambda^\sigma\geq K.
\end{equation}
\end{thm}
For another Lichnerowicz type eigenvalue estimate for the eigenvalues $\lambda_2$ and $2-\lambda_N$ of the graph Laplacian $\Delta$ in terms of the coarse Ricci curvature bound due to Ollivier \cite{Ollivier09}, we refer to \cite{BJL12}.
An interesting application of Theorem \ref{thm:introLic} is the following: The jump phenomenon of the quantity $\lambda^\sigma$ imposes a similar jump phenomenon on the curvature.

Figure \ref{FtriangleEigenJump} illustrates the jumps of the first nonzero eigenvalue $\lambda^\sigma$ and the curvature $K_{\infty}(\sigma)$ of the particular example of a triangle graph $\mathcal{C}_3$ with $\sigma: E^{or}\to U(1)$, when $\sigma$ becomes balanced. In Figure \ref{FtriangleEigenJump}, the complex variable $s=Sgn(\mathcal{C}_3)\in U(1)$ is the signature of the triangle (see (\ref{eq:sigC}) for the definition). The signature $\sigma$ is balanced if and only if $\mathrm{Re(s)}=1$. See Section \ref{subsection:triangle} for details.

\begin{figure}[h]
\centering
\includegraphics[width=0.5\textwidth]{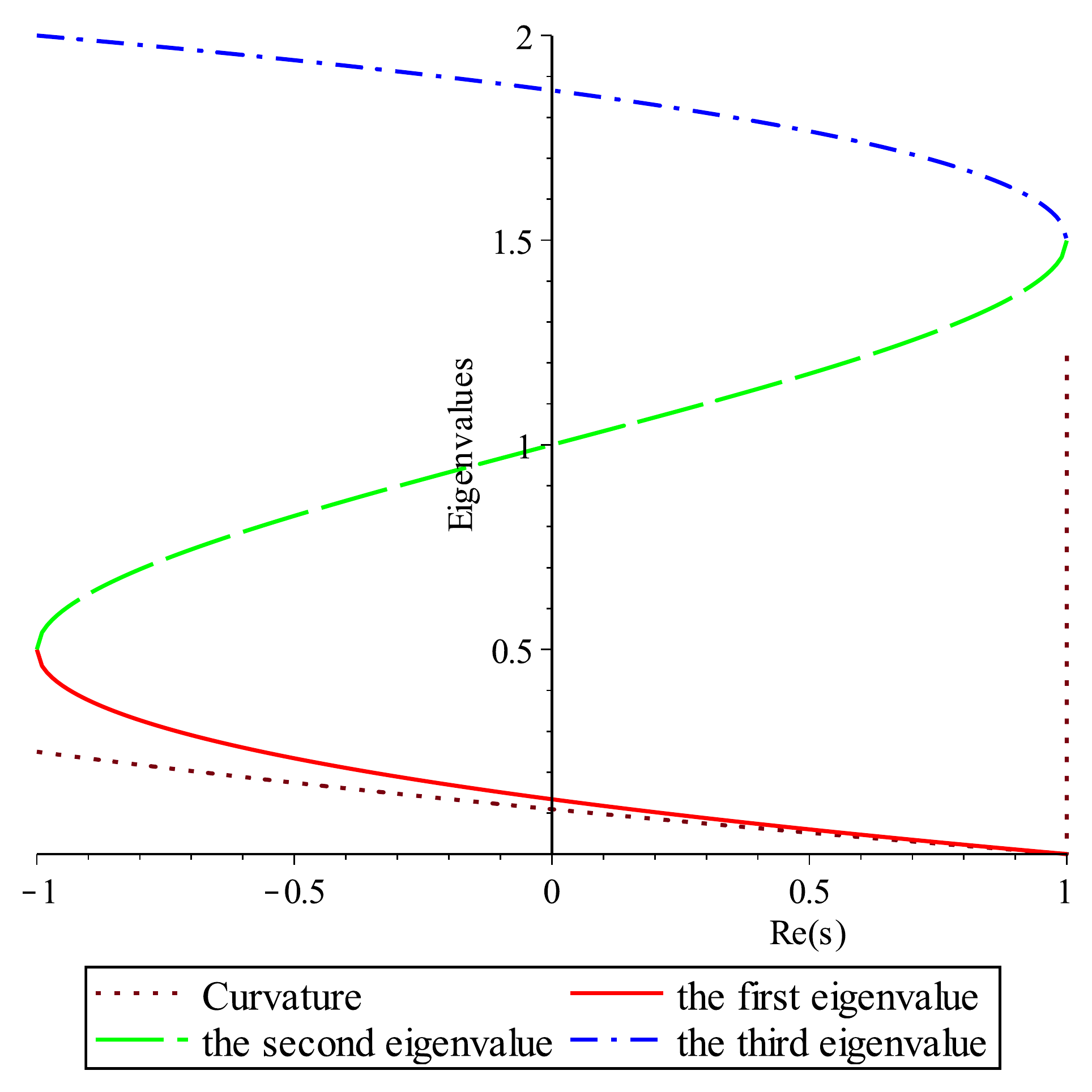}
\caption{Curvature and eigenvalues of a signed triangle\label{FtriangleEigenJump}}
\end{figure}

Moreover, Theorem \ref{thm:introLic} also establishes direct relations between Cheeger constants and the discrete Ricci curvature, see Section \ref{section:Lichnerowicz}.

\subsection{Organization of the paper}
In Section \ref{section:ConnLap}, we set up our general setting of a graph with edge weights, a general vertex measure and a signature, and discuss the associated connection Laplacian. In Section \ref{section:CurDimIneSig}, we discuss various basic properties of the curvature dimension inequalities with signatures and also their equivalent definitions. In Section \ref{section:MultCheeSig}, we introduce multi-way Cheeger constants with signatures and discuss some of the fundamental properties. Section \ref{section:BusIne} is devoted to the proof of our main result, that is, higher order Buser inequalities. In Section \ref{section:Lichnerowicz}, we prove a Lichnerowicz type eigenvalue estimate and discuss its applications. The special case of graphs with $O(1)$ signatures is treated in Section \ref{section:O(1)}, where an eigenvalue ratio estimate is obtained. In Appendix \ref{section:appendixCurCheeCar}, we provide a detailed discussion about the behavior on Cartesian product graphs of the two concepts, curvature dimension inequalities and Cheeger constants with signatures. Appendix \ref{section:appendixSpanningTree} contains a counterexample showing that a combinatorial expression of the frustration index via spanning trees, which we established in Section \ref{section:MultCheeSig} for graphs with $U(1)$ signatures, does no longer hold for $U(d)$ signatures with $d>1$.

\section{The Connection Laplacian}\label{section:ConnLap}

In this section, we introduce the basic setting of a graph with edge weights, a vertex measure and a signature, and the corresponding connection Laplacian.
\subsection{Basic setting} Throughout the paper, $G=(V, E, w)$ denotes an undirected weighted simple finite graph with vertex set $V$ and edge set $E$. If two vertices $x,y\in V$
are connected by an edge, we write $x\sim y$ and denote this edge by $\{x,y\}$. To each edge $\{x,y\}\in E$, we associate a positive symmetric weight $w_{xy}=w_{yx}$.  Let $$d_x:=\sum_{y,y\sim x}w_{xy}$$ be the (weighted) vertex degree of $x\in V$.


For the vertex set $V$, we assign a finite positive measure $\mu: V\to \mathbb{R}_{>0}$. The following two
quantities $D_G^{non}$ and $D_G^{nor}$ will appear naturally in our
arguments:
\begin{equation}\label{eq:degrees}
  D_G^{non}:=\max_{x\in V}\frac{d_x}{\mu(x)}, \,\,\text{ and }\,\, D_G^{nor}:=\max_{x\in V}\max_{y,y\sim x}\frac{\mu(x)}{w_{xy}}.
\end{equation}
Typically, one chooses $\mu(x)=1$ for all $x\in V$ ($\mu=\mathbf{1}_V$ for short), or $\mu(x)=d_x$ for all $x\in V$ ($\mu=\mathbf{d}_V$ for short). The superscripts in (\ref{eq:degrees}) are abbreviations for "nonnormalized" and "normalized", respectively. Observe that, $D_G^{non}=\max_{x\in V}d_x$ for the measure $\mu=\mathbf{1}_V$, while $D_G^{nor}=\max_{x\in V}d_x$ for the measure $\mu=\mathbf{d}_V$ and $w_{xy} = 1$ for all $\{x,y\} \in E$.

We write $(G, \mu, \sigma)$ to denote a graph $G=(V,E,w)$ with the vertex measure $\mu$ and the signature $\sigma:E^{or}\to H$ where $H$ is a group (recall (\ref{intro:signature})).

Recall from the Introduction that $\sigma$ is balanced if it is switching equivalent to the trivial signature $\sigma_{\mathrm{triv}}: E^{or}\to id\in H$. Actually, the original definition of balancedness of a signature by Harary \cite{Harary53} is defined via the signature of cycles of the underlying graph. Let $\mathcal{C}$ be a cycle of $G$, i.e., a subgraph composed of a sequence $(x_1,x_2), (x_2,x_3), \cdots, (x_{\ell-1},x_\ell),(x_\ell,x_1)$ of distinct edges. Then the signature $Sgn(\mathcal{C})$ of $\mathcal{C}$ is defined as the conjugacy class of the element
\begin{equation}\label{eq:sigC}
\sigma_{x_1x_2}\sigma_{x_2x_3}\cdots\sigma_{x_{\ell-1}x_\ell}\sigma_{x_\ell x_1}\in H.
\end{equation}
Note that, the signature of any cycle is switching invariant.
Harary \cite{Harary53} (see also \cite{Zaslavsky82}) defines a signature $\sigma:E^{or}\to H$ to be balanced if the
  signature of every cycle of $G$ is (the conjugacy class of the) identity
  element $id\in H$.
%
In fact, the above two definitions of balancedness of a signature are equivalent, see \cite[Corollary 3.3 and Section 9]{Zaslavsky82}.


For more historical background about signatures of graphs, we refer the reader to~\cite[Section 3]{LPV14}.

\subsection{Connection Laplacian}\label{subsection:Connection Laplacian}
Let $\mathbb{K}=\mathbb{R}$ or $\mathbb{C}$.
Throughout this paper, we restrict the group $H$ to be the orthogonal group $O(d)$ or the unitary group $U(d)$, of dimension $d$, $d\in \mathbb{Z}_{>0}$, when $\mathbb{K}=\mathbb{R}$ or $\mathbb{C}$, respectivly. For every edge $(x,y)\in E^{or}$, $\sigma_{xy}$ is a $(d\times d)$-orthogonal or unitary matrix and we have $\sigma_{yx}=\sigma_{xy}^{-1}=\overline{\sigma_{xy}^T}$.

For any vector-valued functions $f:V\rightarrow \mathbb{K}^d$ and any vertex $x\in V$,
the graph connection Laplacian $\Delta^\sigma$ is defined via
\begin{equation*}
\Delta^\sigma f(x):=\frac{1}{\mu(x)}\sum_{y,y\sim x}w_{xy}(\sigma_{xy}f(y)-f(x))\in \mathbb{K}^d.
\end{equation*}
Note that a function $f:V\rightarrow \mathbb{K}^d$ can also be considered as an $(Nd)$-dimensional column vector, which we denote by $\overrightarrow{f}\in \mathbb{K}^{Nd}$. This vector is well defined once we enumerate the vertices in $V$. The Laplacian can then be written as
\begin{equation}
 \Delta^\sigma=(\mathrm{diag}_\mu)^{-1}(A^\sigma-\mathrm{diag}_D),
\end{equation}
where $\mathrm{diag}_\mu$ and $\mathrm{diag}_D$ are $(Nd)\times (Nd)$-diagonal matrices with the diagonal blocks $(\mathrm{diag}_\mu)_{xx}=\mu(x)I_d$ and $(\mathrm{diag}_D)_{xx}=d_x\mathrm{I}_d$ for $x\in V$, respectively. Here we use $\mathrm{I}_d$ for a $(d\times d)$-identity matrix. The matrix $A^\sigma$ is defined blockwise as follows. For $x,y\in V$, the $(d\times d)$-block of it is given by
\begin{equation}
 (A^\sigma)_{xy}=\left\{
                   \begin{array}{ll}
                     0, & \hbox{if $\{x,y\}\not\in E$;} \\
                     w_{xy}\sigma_{xy}, & \hbox{$(x,y)\in E$.}
                   \end{array}
                 \right.
\end{equation}
Then we have $\overrightarrow{\Delta^\sigma f}=(\mathrm{diag}_\mu)^{-1}(A^\sigma-\mathrm{diag}_D)\overrightarrow{f}$.

If every edge has the trivial signature $1\in O(1)$, $\Delta^\sigma$ reduces to the graph Laplacian $\Delta$.
When $H=U(1)$, $\Delta^\sigma$ coincides with the discrete magnetic Laplacian \cite{Sunada93,Shubin94,LLPP15}.

Given two functions $f,g: V\to \mathbb{K}^d$, locally at a vertex $x$ the Hermitian inner product of $f(x)$ and $g(x)$ is given by $f(x)^T\overline{g(x)}$. The corresponding norm of $f(x)$ is denoted by $|f(x)|:=\sqrt{f^T(x)\overline{f(x)}}$. Globally, we have the following inner product between $f$ and $g$:
\begin{equation}\label{eq:inner product}
\langle f, g\rangle_{\mu}:=\sum_{x\in V}\mu(x)f(x)^T\overline{g(x)}.
\end{equation}
We denote by $\ell^2(V, \mathbb{K}^d; \mu)$ the corresponding Hilbert space of functions. The $\ell^2$ norm corresponding to (\ref{eq:inner product}) is denoted by $\Vert\cdot\Vert_{2,\mu}$. Note that $\Delta^{\sigma}$ is a self-adjoint operator on
$\ell^2(V, \mathbb{K}^d; \mu)$, i.e.,
\begin{equation}\label{eq:self-adjoint}
\langle \Delta^{\sigma}f, g\rangle_{\mu}=\langle f, \Delta^{\sigma}g\rangle_{\mu}.
\end{equation}

We call $\lambda^{\sigma}\in \mathbb{R}$ an eigenvalue of $\Delta^{\sigma}$ if there exists a non-zero function $f: V\to \mathbb{K}^d$ such that $\Delta^{\sigma}f=-\lambda^{\sigma}f$. In fact, all
$Nd$ eigenvalues of $\Delta^\sigma$ lie in the interval $[0,2D_G^{non}]$.

Let $\Sigma$ be the group generated by the elements of $\{\sigma_{xy}\mid (x,y)\in E^{or}\}$. We call $\Sigma$ the \emph{signature group} of the graph $(G,\sigma)$. If the action of $\Sigma$ on $\mathbb{K}^d$ is reducible, we have an orthogonal decomposition of $\mathbb{K}^d$, i.e.,
$$\mathbb{K}^d=U_1\oplus U_2\oplus\cdots\oplus U_r,\,\,\text{for some}\,\, r,$$
where the $U_i$'s are pairwise orthogonal w.r.t. the Hermitian inner product of $\mathbb{K}^d$ and each $U_i$ is an $\Sigma$-invariant subspace of $\mathbb{K}^d$ of dimension $d_i$ such that $\sum_{i=1}^rd_i=d$. Then there exist signatures $\sigma_i: E^{or}\to O(d_i)$ or $U(d_i),\,\,i=1,2,\ldots,r$, such that we can write
\begin{equation*}
\Delta^{\sigma}=\Delta^{\sigma_1}\oplus\Delta^{\sigma_2}\oplus\cdots\oplus \Delta^{\sigma_r},
\end{equation*}
by identifying each $U_i$ with the vector space $\mathbb{K}^{d_i}$.

\section{Curvature dimension inequalities with signatures}\label{section:CurDimIneSig}

In this section, we introduce the $CD^\sigma(K,n)$ inequality for $K\in \mathbb{R}$ and $n\in \mathbb{R}_+$ and discuss its basic properties. We will characterize the $CD^\sigma$ inequality in terms of linear matrix inequalities, and also in terms of heat semigroups for functions and vector fields.

\subsection{Definitions}
We start by discussing several basic properties of the operators $\Gamma^\sigma$ and $\Gamma^\sigma_2$ defined in Definition \ref{def:introGamma} (of course, we are using
the Laplacians in our current general setting). First, observe that they have the following Hermitian properties:
%
\begin{equation}\label{eq:gamma symmetric}
\Gamma^{\sigma}(f,g)=\overline{\Gamma^{\sigma}(g,f)}, \,\,\, \Gamma_2^{\sigma}(f,g)=\overline{\Gamma_2^{\sigma}(g,f)}, \,\,\,\forall f, g: V\to \mathbb{K}^d.
\end{equation}
Since the graph Laplacian $\Delta$ satisfies
\begin{equation}\label{eq:laplacian-integration}
\sum_{x\in V}\mu(x)\Delta(f^T\overline{g})(x)=0,
\end{equation}
the definition (\ref{eq:introGamma}) of $\Gamma^{\sigma}$ and the self-adjointness (\ref{eq:self-adjoint}) of $\Delta^{\sigma}$ lead to the following summation by part formula,
\begin{equation}\label{eq:summaBypart}
\sum_{x\in V}\mu(x)\Gamma^\sigma(f,g)(x)=-\langle f, \Delta^{\sigma}g\rangle_{\mu}=-\langle \Delta^{\sigma}f,g\rangle_{\mu}.
\end{equation}
Moreover, we have the following properties.
\begin{pro}\label{pro:Gammasigma}
For any  two functions $f,g: V\rightarrow \mathbb{K}^d$ and any $x\in V$, we have
\begin{enumerate}[(i)]
  \item $$\Gamma^\sigma(f,g)(x)=\frac{1}{2\mu(x)}\sum_{y,y\sim x}w_{xy}(\sigma_{xy}f(y)-f(x))^T(\overline{\sigma_{xy}g(y)-g(x)});$$
  \item $$\left|\Gamma^\sigma(f,g)(x)\right|\leq\sqrt{\Gamma^\sigma(f)(x)}\sqrt{\Gamma^\sigma(g)(x)}.$$
\end{enumerate}
\end{pro}
\begin{proof}
The formula (i) follows from a direct calculation. (ii) is a consequence of (i) by applying the Cauchy-Schwarz inequality.
\end{proof}

\begin{definition}[$CD^\sigma$ inequality]\label{defn:CDinequality}
  Let $K \in \mathbb {R}$ and $n \in  \R_+$. We say
  that $(G,\mu, \sigma)$ satisfies the $CD^\sigma$ inequality $CD^\sigma(K,n)$ if we have for any
  vector field $f: V\to \mathbb{K}^d$ and any vertex $x\in V$,
  \begin{equation}\label{eq:CDineq}
    \Gamma_2^\sigma(f)(x)\geq \frac{1}{n}\left| \Delta^\sigma f(x) \right|^2+K\Gamma^\sigma(f)(x).
  \end{equation}
  We call $K$ a lower curvature bound of $(G,\mu, \sigma)$ and $n$ a dimension parameter. We define the \emph{$n$-dimensional Ricci curvature} $K_n(G,\mu,\sigma;x)$ of $(G,\mu, \sigma)$ at the vertex $x\in V$ to be the largest $K$ that the inequality (\ref{eq:CDineq}) holds for a given dimension parameter $n$. We further define the \emph{precise $n$-dimensional Ricci curvature lower bound} $K_n(G,\mu,\sigma)$ of $(G,\mu, \sigma)$ as
\begin{equation}
K_n(G,\mu,\sigma):=\min_{x\in V} K_n(G,\mu,\sigma;x).
\end{equation}
We also simply write $K_n(\sigma;x)$ and $K_n(\sigma)$ when the setting $(G,\mu)$ is clear.
\end{definition}
Note that for given $K\in \mathbb{R}$, and $n_1, n_2\in\R_+$ with $n_1\leq n_2$, the inequality $CD^\sigma(K,n_1)$ implies $CD^\sigma(K,n_2)$. In other words, $CD^\sigma(K,n)$ provides a lower curvature bound $K$ and an upper dimension bound $n$ of the graph.

We also remark that rescalling the measure $\mu$ by a constant $c>0$ leads to
\begin{equation}
 K_n(G,c\mu,\sigma;x)=\frac{1}{c}K_n(G,\mu,\sigma;x).
\end{equation}
We will be particularly interested in graphs satisfying $CD^\sigma(K,\infty)$ in this paper.

The classical curvature-dimension inequality $CD(K,n)$ \`{a} la Bakry and \'{E}mery \cite{BaEm} on graphs is defined as follows: For any real-valued function $f:V\to \mathbb{R}$ and any vertex $x$, we have
 \begin{equation}\label{eq:CDineqOriginal}
    \Gamma_2(f)(x)\geq \frac{1}{n}\left| \Delta f(x) \right|^2+K\Gamma(f)(x).
  \end{equation}
Recall the definitions of $\Gamma$ and $\Gamma_2$ from (\ref{intr:Gamma}) and (\ref{intr:Gamma2}).

When $\sigma=\sigma_{\mathrm{triv}}: E^{or}\to id\in U(d)$ is the trivial signature, the graph $(G,\mu,\sigma)$ satisfies the inequality $CD^\sigma(K,n)$ if and only if $(G,\mu)$ satisfies the inequality $CD(K,n)$. In fact, this follows immediately from the following general result.
\begin{pro}\label{pro:CDdecomposition}
Assume that the action of the signature group $\Sigma$ of the graph $(G,\mu,\sigma)$ is decomposable, i.e., we have
\begin{equation*}
  \Delta^{\sigma}=\Delta^{\sigma_1}\oplus\Delta^{\sigma_2}\oplus\cdots\oplus \Delta^{\sigma_r},
\end{equation*}
where $\sigma_i:E^{or}\to U(d_i)\,\,\text{or}\,\,O(d_i), i=1,2,\ldots, r$.
Then the graph $(G, \mu, \sigma)$ satisfies the inequality $CD^{\sigma}(K,n)$ if and only if $(G, \mu, \sigma_i)$ satisfies $CD^{\sigma_i}(K,n)$ for each $i=1,2,\ldots,r$.
\end{pro}
\begin{proof}
By assumption, for any function $f: V\to \mathbb{K}^d$, there exist functions $f_i: V\to U_i\cong\mathbb{K}^{d_i}$, $i=1,2,\ldots,r$ such that $$f^T\overline{f}=f_1^T\overline{f_1}+f_2^T\overline{f_2}+\cdots+f_r^T\overline{f_r},$$ and,
$$  \Delta^{\sigma}f=\Delta^{\sigma_1}f_1\oplus\Delta^{\sigma_2}f_2\oplus\cdots\oplus \Delta^{\sigma_r}f_r.$$
Hence, for any $x\in V$, we obtain by Definition \ref{def:introGamma},
\begin{align*}
\Gamma^{\sigma}(f)(x)=\sum_{i=1}^r\Gamma^{\sigma_i}(f_i)(x),\,\,\text{ and }\,\,\Gamma_2^{\sigma}(f)(x)=\sum_{i=1}^r\Gamma_2^{\sigma_i}(f_i)(x).
\end{align*}
We also have $$\left| \Delta^{\sigma} f(x) \right|^2=\sum_{i=1}^r\left| \Delta^{\sigma_i} f_i(x) \right|^2.$$
Therefore, the inequality
$$\Gamma_2^\sigma(f)(x)\geq \frac{1}{n}\left| \Delta^\sigma f(x) \right|^2+K\Gamma^\sigma(f)(x),\,\,\forall x\in V,$$
is equivalent to the following inequality,
\begin{equation*}
\sum_{i=1}^r\Gamma_2^{\sigma_i}(f_i)(x)\geq \sum_{i=1}^r\left(\frac{1}{n}\left| \Delta^{\sigma_i} f_i(x) \right|^2+\Gamma^{\sigma_i}(f_i)(x)\right),
\end{equation*}
and the proposition follows immediately.
\end{proof}

Given a graph $(G,\mu,\sigma)$, where $\sigma:E^{or}\to O(d_1)$ or $U(d_1)$, we have a natural new signature
\begin{align*}
\sigma\otimes \mathrm{I}_{d_2}: E^{or}&\to O(d_1d_2) \,\,\text{ or }\,\,U(d_1d_2),\\
(x,y)&\mapsto \sigma_{xy}\otimes \mathrm{I}_{d_2},
\end{align*}
where $\mathrm{I}_{d_2}$ stands for the identity matrix of size $d_2\times d_2$. The following observation will be useful in our later discussion about the $CD^\sigma$ inequalities on Cartesian products of graphs in Appendix \ref{section:appendixCurCheeCar}.
\begin{coro}\label{cor:tensor}
A graph $(G, \mu, \sigma)$ satisfies $CD^{\sigma}(K, n)$ if and only if $(G, \mu, \sigma\otimes \mathrm{I}_{d_2})$ satisfies $CD^{\sigma\otimes \mathrm{I}_{d_2}}(K,n)$.
\end{coro}
\begin{proof}
We observe that the action of the signature group of $(G,\sigma\otimes \mathrm{I}_{d_2})$ on $\mathbb{K}^{d_1d_2}$ admits an orthogonal decomposition  and, therefore, we have
$$\Delta^{\sigma\otimes \mathrm{I}_{d_2}}=\underbrace{\Delta^{\sigma}\oplus\cdots\oplus\Delta^{\sigma}}_{d_2\,\, \text{times}}.$$
Corollary \ref{cor:tensor} is then a direct consequence of Proposition \ref{pro:CDdecomposition}.
\end{proof}


\subsection{Switching invariance}\label{section:switching inv}
The $CD^\sigma$ inequality is switching invariant.
\begin{pro}\label{pro:curvature-switching}
If $(G, \mu, \sigma)$ satisfies $CD^{\sigma}(K,n)$, then $(G, \mu, \sigma^{\tau})$ satisfies $CD^{\sigma^\tau}(K,n)$ for any switching function $\tau: V\to H$.
\end{pro}
\begin{proof}
Recalling (\ref{intro:unitary}), we check that we have for any $\tau: V\to H$ and $f, g: V\to \mathbb{K}^d$,
\begin{equation}\label{eq:gamma-switching}
\Gamma^{\sigma^{\tau}}(f,g)=\Gamma^{\sigma}(\tau^{-1}f, \tau^{-1}g) \,\,\text{ and }\,\, \Gamma_2^{\sigma^{\tau}}(f,g)=\Gamma_2^{\sigma}(\tau^{-1}f, \tau^{-1}g),
\end{equation}
using $\overline{\tau(x)^T}=\tau^{-1}(x)$.
The proposition then follows immediately from (\ref{intro:unitary}) and (\ref{eq:gamma-switching}).
\end{proof}
The arguments in the above proof show also that $K_n(G,\mu,\sigma;x)$, introduced in Definition \ref{defn:CDinequality}, is switching invariant for any given $n$.

We denote by $\mathrm{dist}$ the canonical graph distance and define the ball of radius $r$ centered at $x\in V$ by
\begin{equation*}
 B_r(x):=\{y\in V\mid \mathrm{dist}(x,y)\leq r\}.
\end{equation*}
\begin{pro}\label{pro:shortcycles}
 Let $(G,\mu,\sigma)$ be given. If the signature of every cycle of length $3$ or $4$ is equal to (the conjugate class of) $id\in H$, then $(G,\mu,\sigma)$ satisfies $CD^{\sigma}(K,n)$ if and only if $(G,\mu)$ satisfies $CD(K,n)$.
\end{pro}
\begin{proof} Let $x\in V$ be a vertex.
Since all cycles of $3$ or $4$ have trivial signature, we can switch all the signatures of edges in the subgraph induced by the ball $B_2(x)$ to be trivial. Note that the inequality (\ref{eq:CDineq}) only involves the vertices in the ball $B_2(x)$. Then the proposition follows from Propositions \ref{pro:curvature-switching} and \ref{pro:CDdecomposition}.
\end{proof}

\subsection{Coverings and a general lower curvature bound}\label{section:covering}
Let $(\tilde{G}, \tilde{\mu}, \tilde{\sigma})$ and $(G, \mu, \sigma)$ be two graphs. Let $\pi: (\tilde{G}, \tilde{\mu}, \tilde{\sigma})\to (G, \mu, \sigma)$ be a graph homomorphism, namely, $\pi: \tilde{V}\to V$ is surjective, and if $\{\tilde{x}, \tilde{y}\}\in \tilde{E}$, then $\{\pi(\tilde{x}), \pi(\tilde{y})\}\in E$. Moreover, we require
\begin{equation}
\tilde{\sigma}_{\tilde{x}\tilde{y}}=\sigma_{\pi(\tilde{x})\pi(\tilde{y})}, \,\,\, \tilde{w}_{\tilde{x}\tilde{y}}=w_{\pi(\tilde{x}) \pi(\tilde{y})}\,\,\,\text{ and }\,\,\,\tilde{\mu}(\tilde{x})=\mu(\pi(\tilde{x})).
\end{equation}
 Such a map $\pi$ is called a \emph{covering map} if, furthermore, the subgraph of $\tilde{G}$ induced by the ball $B_1(\tilde{x})$ centered at each vertex $\tilde{x}\in \tilde{V}$ is mapped bijectively to the subgraph of $G$ induced by the ball $B_1(x)$. If a covering map $\pi: (\tilde{G}, \tilde{\mu}, \tilde{\sigma})\to (G, \mu, \sigma)$ exists, we call the graph $(\tilde{G}, \tilde{\mu}, \tilde{\sigma})$ a \emph{covering graph} of $(G, \mu, \sigma)$.
\begin{thm}\label{thm:covering}
Let $(\tilde{G}, \tilde{\mu}, \tilde{\sigma})$ be a covering graph of $(G, \mu, \sigma)$. If $(\tilde{G}, \tilde{\mu}, \tilde{\sigma})$ satisfies $CD^{\tilde{\sigma}}(K, n)$, then $(G, \mu, \sigma)$ satisfies $CD^{\sigma}(K,n)$.
\end{thm}
\begin{proof}
For any function $f: V\to \mathbb{K}^d$, we can find a corresponding function $\tilde{f}: \tilde{V}\to \mathbb{K}^d$ such that
\begin{equation}
\tilde{f}(\tilde{x}):=f(\pi(\tilde{x}))\,\,\,\forall\,\, \tilde{x}\in \tilde{V},
\end{equation}
where $\pi$ is the covering map from $(\tilde{G}, \tilde{\mu}, \tilde{\sigma})$ to $(G, \mu, \sigma)$.

For any $x\in V$, and any $\tilde{x}\in \pi^{-1}(x)$, we can check by definition of a covering map that
\begin{equation}\label{eq:covering}
|\Delta^{\tilde{\sigma}}\tilde{f}(\tilde{x})|^2=\left|\Delta^{\sigma}f(x)\right|^2,\,\, \Gamma^{\tilde{\sigma}}(\tilde{f})(\tilde{x})=\Gamma^\sigma(f)(x),\,\, \Gamma_2^{\tilde{\sigma}}(\tilde{f})(\tilde{x})=\Gamma_2^\sigma(f)(x).
\end{equation}

Since $(\tilde{G}, \tilde{\mu}, \tilde{\sigma})$ satisfies $CD^{\tilde{\sigma}}(K,n)$, we obtain that for any $f: V\to \mathbb{K}^d$, and any vertex $x\in V$,
\begin{equation}
 \Gamma_2^{\tilde{\sigma}}(\tilde{f})(\tilde{x})\geq \frac{1}{n}|\Delta^{\tilde{\sigma}}\tilde{f}(\tilde{x})|^2+K\Gamma^{\tilde{\sigma}}(\tilde{f})(\tilde{x}).
\end{equation}
Combining this with (\ref{eq:covering}) completes the proof.
\end{proof}

\begin{coro}\label{cor:lower curvature bound}
Any graph $(G, \mu, \sigma)$ satisfies the inequality $$CD^{\sigma}\left(\frac{2}{D^{nor}_G}-D_G^{non},2\right).$$ In particular, any unweighted cycle graph with constant vertex measure $\mu\equiv\nu_0\cdot \mathbf{1}_V$ and any signature $\sigma: E^{or}\to H$ satisfies $$CD^{\sigma}(0,2).$$
\end{coro}
\begin{proof}
Let $(T_G, \tilde{\mu}, \tilde{\sigma})$ be the universal covering of $(G, \mu, \sigma)$, i.e., $T_G$ is a tree. It is shown in \cite[Theorem 1.2]{LinYau} (see also \cite[Theorem 8]{JostLiu14}) that $(T_G, \tilde{\mu})$ satisfies the $CD$ inequality
$$CD\left(\frac{2}{D^{nor}_{T_G}}-D^{non}_{T_G},2\right).$$
Due to Proposition \ref{pro:shortcycles}, we know that $(T_G, \tilde{\mu}, \tilde{\sigma})$ satisfies
$$CD^{\tilde{\sigma}}\left(\frac{2}{D^{nor}_{T_G}}-D^{non}_{T_G},2\right),$$ since a tree has no cycles.
By the definition of a covering graph, we have $D_G^{nor}=D_{T_G}^{nor}$ and $D_G^{non}=D_G^{non}$. Then the corollary follows directly from Theorem \ref{thm:covering}.
\end{proof}


\subsection{$CD^\sigma$ inequality as linear matrix inequalities}\label{section:LMI}
In this subsection, we present an equivalent formulation of the $CD^\sigma$ inequality via linear matrix inequalities. As a consequence, the problem of calculating the Ricci curvature of a graph is reduced to solving semidefinite programming problems. In this process, we explore the geometrical information captured by the $CD^\sigma$ inequality of a graph.

By Definition \ref{def:introGamma}, the operators $\Gamma^\sigma$ and $\Gamma^\sigma_2$ can be considered as two symmetric sesquilinear forms. Hence they can be represented by Hermitian matrices. For our purpose, we are interested in considering the two symmetric sesquilinear forms locally at every vertex $x\in V$. There exist two $(Nd)\times (Nd)$-Hermitian matrices $\Gamma^\sigma(x)$ and $\Gamma^\sigma_2(x)$ such that for any two functions $f,g:V\to \mathbb{K}^d$,
\begin{equation}
 \Gamma^\sigma(f,g)(x)=\overrightarrow{f}^T\Gamma^\sigma(x)\overline{\overrightarrow{g}}, \,\,\text{ and }\,\, \Gamma_2^\sigma(f,g)(x)=\overrightarrow{f}^T\Gamma_2^\sigma(x)\overline{\overrightarrow{g}}.
\end{equation}

Denote by $|B_r(x)|$ the cardinality of the set $B_r(x)$. Observe that the matrix $\Gamma^\sigma(x)$ only has a nontrivial block of size $|B_1(x)| \times |B_1(x)|$, while the matrix $\Gamma_2^\sigma$ only has a nontrivial block of size $|B_2(x)| \times |B_2(x)|$.

We denote by $\Delta^\sigma(x)$ the $(d\times Nd)$-matrix such that $\Delta^\sigma f(x)=\Delta^\sigma(x)\overrightarrow{f}$ for all functions $f: V\to \mathbb{K}^d$.
Given two Hermitian matrices $M_1$ and $M_2$, the inequality $M_1\geq M_2$ means that the matrix $M_1-M_2$ is positive semidefinite. Then we have the following equivalent definition of $CD^\sigma$ inequality.
\begin{definition}[$CD^\sigma$ inequality as linear matrix inequalities]
 Let $K\in \mathbb{R}$ and $n\in \mathbb{R}_+$. A graph $(G,\mu,\sigma)$ satisfies the $CD^\sigma$ inequality $CD^\sigma(K,n)$ if and only if, for any vertex $x\in V$, the following linear matrix inequality holds,
\begin{equation}
 \Gamma_2(x)\geq \frac{1}{n}\Delta^\sigma(x)^T\overline{\Delta^\sigma(x)}+K\Gamma^\sigma(x).
\end{equation}
\end{definition}
A direct consequence is the following proposition.
\begin{pro}[Semidefinite Programming]\label{pro:sdp}
The $n$-dimensional Ricci curvature $K_n(G,\mu,\sigma;x)$ of the the graph $(G,\mu,\sigma)$ at the vertex $x\in V$ is the solution of the following semidefinite programming,
\begin{align*}
 &\text{maximize}\,\,\, K\\
&\text{subject to}\,\,\,\Gamma_2^\sigma(x)-\frac{1}{n}\Delta^\sigma(x)^T\overline{\Delta^\sigma(x)}\geq K\Gamma^\sigma(x).
\end{align*}
\end{pro}
Semidefinite programming can be efficiently solved. There are several software packages available.

In the following, we describe the explicit structure of the matrices $\Delta^\sigma(x), \Gamma^\sigma(x)$ and $\Gamma_2^\sigma(x)$. For simplicity, we restrict to the setting
\begin{equation}\label{setting:1}
\mu=\mathbf{1}_V,\,\,\text{i.e.},\,\,\mu(x)=1 \,\,\forall\,\,x\in V,
\end{equation}
and
\begin{equation}\label{setting:2}
w_{xy}=1\,\,\forall\,\,\{x,y\}\in E.
\end{equation}

Given a vertex $x\in V$, let us denote its neighbors by $y_1,y_2,\ldots,y_{d_x}$. By abuse of notation, we still write $\Delta^\sigma(x)$ and $\Gamma^\sigma(x)$ for their nontrivial blocks corresponding to the vertices $x,y_1,\ldots,y_{d_x}$. Then it is easy to see that
\begin{equation}
\Delta^\sigma(x)=\left(
                   \begin{array}{cccc}
                     -d_xI_d & \sigma_{xy_1} & \cdots & \sigma_{xy_{d_x}} \\
                   \end{array}
                 \right),
\end{equation}
and
\begin{equation}\label{eq:GammaMatrix}
 2\Gamma^\sigma(x)=\left(
                    \begin{array}{cccc}
                      d_xI_d & -\overline{\sigma_{xy_1}} & \cdots & -\overline{\sigma_{xy_{d_x}}} \\
                      -\sigma_{xy_1}^T & I_d & \cdots & 0 \\
                      \vdots & \vdots & \ddots & \vdots \\
                      -\sigma_{xy_{d_x}}^T & 0 & \cdots & I_d \\
                    \end{array}
                  \right).
\end{equation}

For the matrix $\Gamma_2^\sigma(x)$, the structure of the subgraph induced by $B_2(x)$ is of relevance. We denote the sphere of radius $r$ centered at a vertex $x\in V$ by
$$\mathrm{S}_r(x):=\{y\in V\mid \mathrm{dist}(x,y)=r\}.$$
Then, the ball $B_2(x)$ has the decomposition $B_2(x)=\{x\}\sqcup \mathrm{S}_1(x)\sqcup\mathrm{S}_2(x)$.

We first introduce some natural geometric quantities before we present the entries of the matrix $\Gamma_2^\sigma(x)$. For any vertex $y\in \mathrm{S}_1(x)$, we have
\begin{equation}\label{eq:farneighbors}
 |\mathrm{S}_1(y)\cap\mathrm{S}_2(x)|:=\sum_{z,z\sim y,z\not\sim x,z\neq x}1,
\end{equation}
and
\begin{equation}\label{eq:triangles}
 \sharp_{\bigtriangleup}(x,y):=|\mathrm{S}_1(y)\cap\mathrm{S}_1(x)|:=\sum_{z,z\sim y,z\sim x}1.
\end{equation}
Note that (\ref{eq:triangles}) is the number of triangles (i.e., $3$-cycles) which contain the two neighbors $x$ and $y$. This justifies the notation $\sharp_{\bigtriangleup}(x,y)$.

For any vertex $z\in \mathrm{S}_2(x)$, we have
\begin{equation}\label{eq:nearneighbors}
|\mathrm{S}_1(z)\cap\mathrm{S}_1(x)|:=\sum_{y,y\sim x,y\sim z}1.
\end{equation}
Note that (\ref{eq:nearneighbors}) is related to the number of $4$-cycles which contain the two vertices $x$ and $z$.

\begin{remark}
The above three geometric quantities are all closely related to the growth rate of the cardinality of $B_r(x)$ (in other words, the volume of $B_r(x)$ w.r.t. the measure $\mu=\mathbf{1}_V$) when the radius $r$ increases.
\end{remark}

The quantity $\sharp_{\bigtriangleup}(x,y)$ counts the number of $3$-cycles regardless of their signatures. A signed version of this quantity is also important, and we define the following quantity describing the unbalancedness of the triangles containing the two neighbors $x$ and $y$:
\begin{equation}\label{eq:signedtriangles}
\sharp_{\bigtriangleup}^\sigma(x,y):=\sum_{z,z\sim y,z\sim x}\left(\mathrm{I}_d-\overline{\sigma_{xz}\sigma_{zy}\sigma_{yx}}\right).
\end{equation}
Note that the balanced triangles containing $x$ and $y$ do not contribute to the expression in (\ref{eq:signedtriangles}).
\begin{pro}\label{pro:gamma2Matrix}
 Under the setting of (\ref{setting:1}) and (\ref{setting:2}), the nontrivial block of $\Gamma_2^\sigma(x)$, which is Hermitian and of size $|B_2(x)|\times |B_2(x)|$, is given by the following blocks:
\begin{align}
&(4\Gamma_2^\sigma(x))_{xx}=(3d_x+d_x^2)\mathrm{I}_d;\\
&(4\Gamma_2^\sigma(x))_{xy}=-\left(3+d_x+|\mathrm{S}_1(y)\cap\mathrm{S}_2(x)|+\sharp_{\bigtriangleup}^\sigma(x,y)\right)\overline{\sigma_{xy}},\\
&\hspace{7cm}\text{ for any $y\in \mathrm{S}_1(x)$;}\notag\\
&(4\Gamma_2^\sigma(x))_{xz}=\sum_{y,y\sim x,y\sim z}\overline{\sigma_{xy}\sigma_{yz}}, \text{ for any $z\in \mathrm{S}_2(x)$;}\\
&(4\Gamma_2^\sigma(x))_{yy}=\left(5-d_x+3|\mathrm{S}_1(y)\cap\mathrm{S}_2(x)|+4\sharp_{\bigtriangleup}(x,y)\right)\mathrm{I}_d,\\
&\hspace{7cm}\text{ for any $y\in \mathrm{S}_1(x)$;}\notag\\
&(4\Gamma_2^\sigma(x))_{y_1y_2}=2\overline{\sigma_{y_1x}\sigma_{xy_2}}-4\overline{\sigma_{y_1y_2}},\\
&\hspace{1cm}\text{ for any $y_1,y_2\in \mathrm{S}_1(x)$, $y_1\neq y_2$, where we use $\sigma_{y_1y_2}=0$ if $\{y_1,y_2\}\not\in E$;}\notag\\
&(4\Gamma_2^\sigma(x))_{yz}=-2\overline{\sigma_{yz}},\\
&\hspace{1cm}\text{ for any $y\in \mathrm{S}_1(x)$ and $z\in\mathrm{S}_2(x)$, where we use $\sigma_{yz}=0$ if $\{y,z\}\not\in E$;}\notag\\
&(4\Gamma_2^\sigma(x))_{zz}=|\mathrm{S}_1(z)\cap\mathrm{S}_1(x)|\mathrm{I}_d, \text{ for any $z\in \mathrm{S_2}(x)$};\\
&(4\Gamma_2^\sigma(x))_{z_1z_2}=0, \text{ for any $z_1,z_2\in \mathrm{S}_2(x)$, $z_1\neq z_2$.}
\end{align}
\end{pro}
\begin{proof}
 This follows from a direct expansion of the identity
$$\Gamma_2^\sigma(f,g)(x)=\overrightarrow{f}^T\Gamma_2^\sigma(x)\overline{\overrightarrow{g}},\,\,\text{ for any $f,g:V\to \mathbb{K}^d$}.$$
We omit the details here.
\end{proof}
\begin{remark}\begin{enumerate}[(i)]
                \item The block $(4\Gamma_2^\sigma(x))_{xz}$ above is a signed version of the quantity $|\mathrm{S}_1(z)\cap\mathrm{S}_1(x)|$ in (\ref{eq:nearneighbors}).
                \item When $y_1,y_2\in \mathrm{S}_1(x)$ are neighbors, i.e. $\{y_1,y_2\}\in E$, we have a triangle containing $x,y_1$ and $y_2$. Then the block $(4\Gamma_2^\sigma(x))_{y_1y_2}$ can be rewritten as
$$-2\left(\mathrm{I}_d+\left(\mathrm{I}_d-\overline{\sigma_{y_1x}\sigma_{xy_2}\sigma_{y_2y_1}}\right)\right)\overline{\sigma_{y_1y_2}},$$
which describes the unbalancedness of this triangle.
              \end{enumerate}
\end{remark}

\subsection{Example of a signed triangle}\label{subsection:triangle}
We consider a particular example of a triangle graph $\mathcal{C}_{3}$, which consists of three vertices $x,y,$ and $z$, as shown in Figure \ref{F1}. We set
\begin{equation}\label{setting:3}
 \mu(x)=\mu(y)=\mu(z)=2,\,\,\,\text{ and }\,\,\, w_{xy}=w_{xz}=w_{yz}=1.
\end{equation}
Let $\sigma: E^{or}\to U(1):=\{z\in \mathbb{C}, |z|=1\}$ be a signature on $\mathcal{C}_{3}$. Assume that the signature of the cycle $\mathcal{C}_3$ is equal to (the conjugacy class of) $s\in U(1)$. Then $\sigma$ is switching equivalent to the signature given in Figure \ref{F1}, i.e.,
$$\sigma_{xy}=\sigma_{xz}=1 \,\,\,\text{ and }\,\,\, \sigma_{yz}=s.$$

\begin{figure}[h]
\begin{minipage}[t]{0.45\linewidth}
\centering
\includegraphics[width=0.6\textwidth]{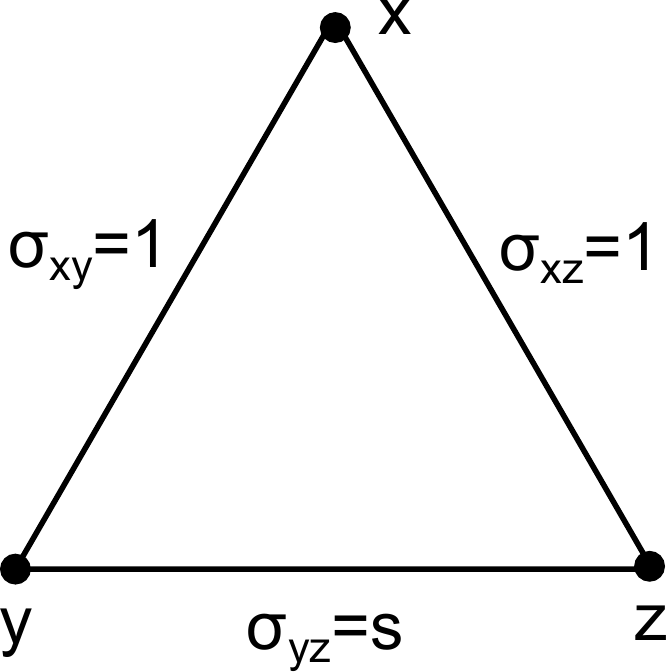}
\caption{A signed triangle\label{F1}}
\end{minipage}
\hfill
\end{figure}

%

\begin{pro}\label{pro:curvature triangle}
Let $(\mathcal{C}_3, \mu, \sigma)$ be as above and $s=Sgn(\mathcal{C}_3)$. Then it has constant $\infty$-dimensional Ricci curvature at every vertex. As a function of $s$, $K_\infty(s):=K_\infty(\mathcal{C}_3, \mu, \sigma)$ is given by
\begin{equation}\label{eq:triangle Curvature}
K_\infty(s)=\left\{
              \begin{aligned}
                &\frac{5}{4}, &&\hbox{if $s=1$;} \\
                &\frac{5-\sqrt{17+8\mathrm{Re}(s)}}{8}, &&\hbox{otherwise.}
              \end{aligned}
            \right.
\end{equation}
\end{pro}

\begin{remark}\label{remark:jumptriangle}
The curvature (\ref{eq:triangle Curvature}) is illustrated in Figure \ref{F2} as a function of the variable $\mathrm{Re}(s)$. The function $K_{\infty}(s)$ "jumps" at $s=1$. That is,
\begin{equation}
\lim_{s\to 1}K_\infty(s)=0, \,\,\,\text{ but }\,\,\, K_\infty(1)=\frac{5}{4}>0.
\end{equation}
We will show that such a "jump" appears in a more general setting in Section \ref{section:Lichnerowicz}.
\end{remark}
\begin{figure}[h]
\centering
\includegraphics[width=.5\textwidth]{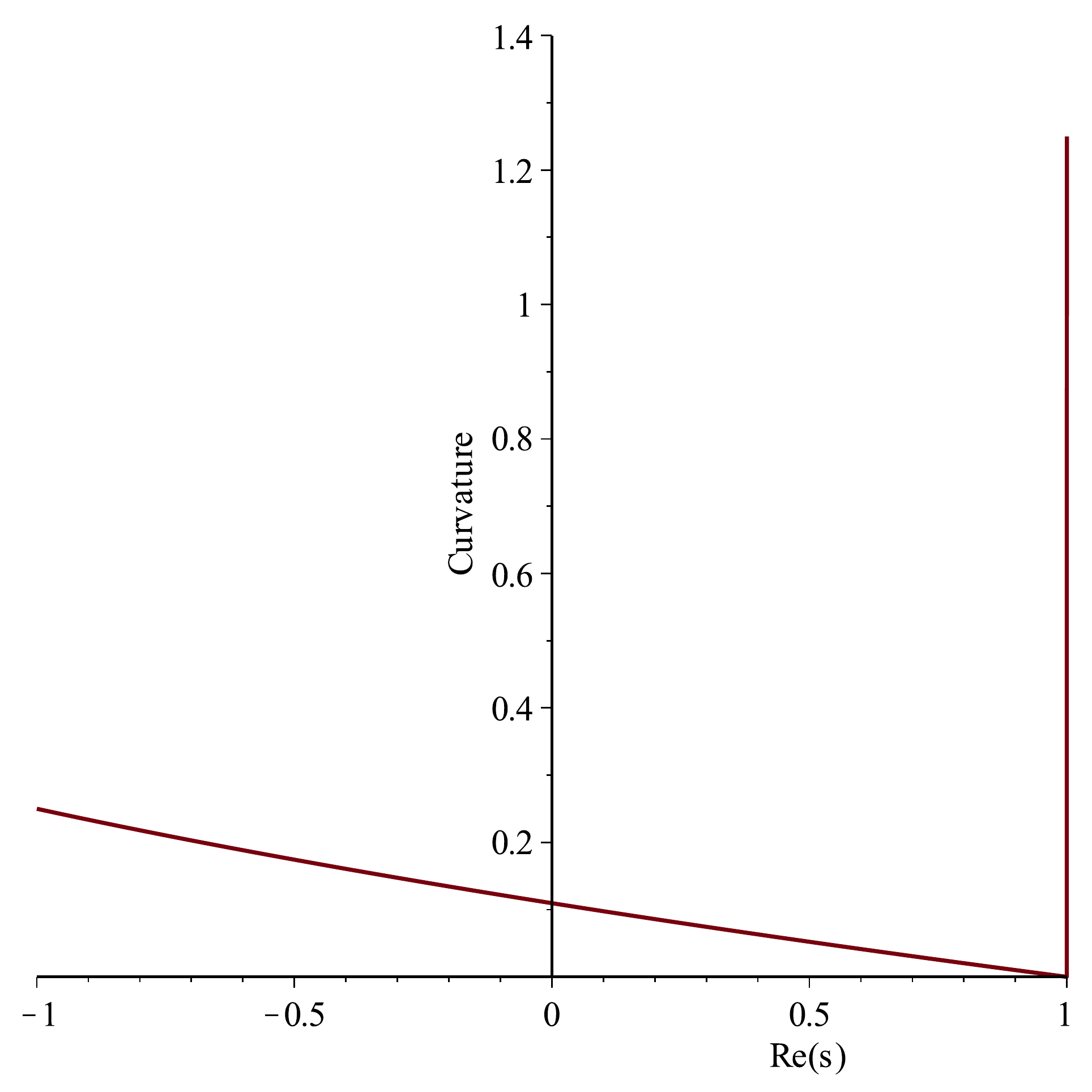}
\caption{$\infty$-dimensional Ricci curvature of a signed triangle}\label{F2}
\end{figure}
\begin{proof}[Proof of Proposition \ref{pro:curvature triangle}]
 Since the curvature is switching invariant, we can switch the signature $\sigma$ to be as shown in Figure \ref{F1} before calculating the curvature $K_\infty(\sigma;x)$ at $x$. In fact, one can do similar operations for calculating $K_\infty(y)$ and $K_\infty(z)$. So $(\mathcal{C}_3, \mu, \sigma)$ has constant curvature and we only need to calculate the curvature at $x$.

By the fact (\ref{eq:GammaMatrix}) and Proposition \ref{pro:gamma2Matrix}, we can obtain the corresponding matrices $\Gamma^\sigma(x)$ and $\Gamma_2^\sigma(x)$. Note in this example, we choose a different measure (\ref{setting:3}) from that in (\ref{setting:1}). Hence these matrices differ by a scaling of $1/2$ and $1/4$ respectively. Therefore, under the current setting (\ref{setting:3}), we have
\begin{equation*}
 2\Gamma^\sigma(x)=\frac{1}{2}\left(
                                \begin{array}{ccc}
                                  2 & -1 & -1 \\
                                  -1 & 1 & 0 \\
                                  -1 & 0 & 1 \\
                                \end{array}
                              \right)
\,\text{ and }\,
4\Gamma_2^\sigma(x)=\frac{1}{4}\left(
                                 \begin{array}{ccc}
                                   10 & -6+s & -6+\overline{s} \\
                                   -6+\overline{s} & 7 & 2-4\overline{s} \\
                                   -6+s & 2-4s & 7 \\
                                 \end{array}
                               \right).
\end{equation*}
By Proposition \ref{pro:sdp}, we need to solve the following semidefinite programming:
\begin{align}
 &\text{maximize}\,\,K\notag\\
&\text{subject to}\,\,\Gamma_2^\sigma(x)\geq K\Gamma^\sigma(x).\label{eq:TriLMI}
\end{align}
Inequality (\ref{eq:TriLMI}) is equivalent to positive semidefiniteness of the following matrix:
\begin{equation}
 16\Gamma_2^\sigma(x)-16K\Gamma^\sigma(x)=\left(
                                           \begin{array}{ccc}
                                             10-8K & -6+s+4K & -6+\overline{s}+4K \\
                                             -6+\overline{s}+4K & 7-4K & 2-4\overline{s} \\
                                             -6+s+4K & 2-4s & 7-4K \\
                                           \end{array}
                                         \right).
\end{equation}
By Sylvester's criterion, this is equivalent to nonnegativity of all principle minors of the above matrix. Calculating these principle minors, we translate the semidefinite programming to the following problem,
\begin{align*}
 \text{maximize} \,\,&K\notag\\
\text{subject to}\,\,&10-8K\geq 0,\,\,7-4K\geq 0\\
&16K^2-8(6+\mathrm{Re}(s))K+(33+12\mathrm{Re}(x))\geq 0,\\
&16K^2-56K+(16\mathrm{Re}(s)+29)\geq 0,\\
&8(1-\mathrm{Re}(s))K^2-10(1-\mathrm{Re}(s))K+(1-Re(s))^2\geq 0.
\end{align*}
Rewriting the above inequalities, we obtain
\begin{align*}
 \text{maximize} \,\,&K\notag\\
\text{subject to}\,\,&K\leq 5/4,\,\,K\leq 7/4,\\
&K\geq (5+\sqrt{17+8\mathrm{Re}(s)})/8\,\,\text{or}\,\,K\leq (5-\sqrt{17+8\mathrm{Re}(s)})/8,\\
&K\geq (7+2\sqrt{5-4\mathrm{Re}(s)})/4\,\,\text{or}\,\,K\leq (7-2\sqrt{5-4\mathrm{Re}(s)})/4,\\
&K\geq (6+\mathrm{Re}(s)+\sqrt{\mathrm{Re}(s)^2+3})/4\,\,\text{or}\,\,K\leq (6+\mathrm{Re}(s)-\sqrt{\mathrm{Re}(s)^2+3})/4.
\end{align*}
One can check directly that (\ref{eq:triangle Curvature}) is the solution of this optimization problem.
\end{proof}
Similarly, one can calculate the $\infty$-dimensional Ricci curvature of longer cycles $\mathcal{C}_N$ for $N\geq 4$.
\begin{pro}\label{pro:curLongCycles}
 Let $(\mathcal{C}_N, \mu, \sigma)$ be a cycle of length $N$ with the edge weights and measure $\mu$ given in (\ref{setting:3}) and $s=Sgn(\mathcal{C}_N)$. Then $(\mathcal{C}_N, \mu, \sigma)$ has constant $\infty$-dimensional Ricci curvature at every vertex.
Moreover, we have
\begin{equation}\label{eq:4cycle Curvature}
K_\infty(\mathcal{C}_4, \mu, \sigma)=\left\{
              \begin{array}{ll}
                1, & \hbox{if $s=1$;} \\
                0, & \hbox{otherwise,}
              \end{array}
            \right.
\end{equation}
and, for $N\geq 5$,
\begin{equation}\label{eq:5cycle Curvature}
K_\infty(\mathcal{C}_N, \mu, \sigma)=0.
\end{equation}
\end{pro}

We remark that new examples of graphs $(G,\sigma)$ satisfying the $CD^\sigma(0,\infty)$ inequality can be constructed by taking Cartesian products of known examples for various choices of the signature, edge weights, and vertex measure on the product graph. We refer to Appendix \ref{section:appendixCurCheeCar} for full details about the behavior of $CD^\sigma$ inequalities on Cartesian product graphs.

\subsection{Heat semigroup characterizations of $CD^{\sigma}$ inequalities}\label{section:HeatChar}

In this subsection, we derive characterizations of the $CD^{\sigma}$ inequality via the solution of the following associated continuous time heat equation,
\begin{equation}
\left\{\begin{aligned}
        &\frac{\partial u(x,t)}{\partial t}=\Delta^{\sigma}u(x,t),\\
        &u(x,0)=f(x),
        \end{aligned}
\right.
\end{equation}
where $f: V\to \mathbb{K}^d$ is an initial function. The solution $u: V\times [0, \infty)\to \mathbb{K}^d$ is given by $P_t^{\sigma}f:=e^{t\Delta^{\sigma}}f$, where $P_t^\sigma$ is a linear operator on the space $\ell^2(V, \mathbb{K}^d; \mu)$. Clearly, we have $P^{\sigma}_0f=f$. It is straightforward to check the following properties of $P_t^\sigma$.

\begin{pro}\label{pro:heatsemigroup}
  The operator $P^\sigma_t, t\geq 0$ satisfies the following properties:
  \begin{enumerate}[(i)]
  \item $P^\sigma_t$ is a self-adjoint operator on the space $\ell^2(V, \mathbb{K}^d; \mu)$;\label{pro:Ptsa}
  \item $P^\sigma_t$ commutes with $\Delta^\sigma$, i.e. $P^\sigma_t\Delta^\sigma=\Delta^\sigma P_t^\sigma$;\label{pro:Ptcomm}
  \item $P^\sigma_t P^\sigma_s=P^\sigma_{t+s}$ for any $t,s\geq 0$.\label{pro:Ptadd}
  \end{enumerate}
\end{pro}

The solution of the heat equation corresponding to the graph Laplacian $\Delta$ is simply denoted by $P_t:=e^{t\Delta}$. The matrix $P_t$ has the following additional properties besides the ones listed in Proposition \ref{pro:heatsemigroup}.

\begin{pro}\label{pro:heatsemigroupunsigned}
\begin{enumerate}[(i)]
  \item All matrix entries of $P_t$ are real and nonnegative;\label{pro:Ptnonnegative}
  \item For any constant function $c$ on $V$, we have $P_tc=c$.\label{pro:Ptprobability}
\end{enumerate}
\end{pro}
In particular, the above properties imply that for a function $f:V\to \mathbb{R}$ with $0\leq f(x)\leq c,\,\,\forall\,\,x\in V$, we have $0\leq P_tf(x)\leq c,\,\,\forall\,\,x\in V$.

\begin{proof}
 Recall that $\Delta$ can be written as the matrix $(\mathrm{diag}_\mu)^{-1}(A-\mathrm{diag}_D)$, where $\mathrm{diag}_D$ and $\mathrm{diag}_\mu$ are the diagonal matrices with $(\mathrm{diag}_D)_{xx}=d_x$ and $(\mathrm{diag}_\mu)_{xx}=\mu(x)$ for all $x\in V$, and $A$ is the weighted adjacency matrix.
 Now we exploit the fact that
 \begin{equation}\label{eq:Offdiag}
\text{all off-diagonal entries of }\,\,\Delta\,\,\text{ are nonnegative,}
\end{equation}
and, therefore, we can choose $C>0$ such that
$\Delta+C\cdot \mathrm{I}_N$ is entry-wise nonnegative.
Then $e^{\Delta+C\cdot \mathrm{I}_N}$ is also entry-wise nonnegative, which implies that the same holds for $P_t=e^{\Delta+C\cdot \mathrm{I}_N}\cdot e^{-C\cdot \mathrm{I}_N}$.

For the constant function $c$, we have
\begin{equation}\label{eq:heatNosign2}
\Delta c=0.
\end{equation}
Therefore, we have $\frac{\partial}{\partial t}P_tc=0$, which implies $P_tc=c$.
\end{proof}

\begin{remark}
Note that the two facts (\ref{eq:Offdiag}) and (\ref{eq:heatNosign2}) do not extend to general $P_t^\sigma$, even when $\sigma$ only takes values from $O(1)=\{\pm 1\}$. Therefore Proposition \ref{pro:heatsemigroupunsigned} does not hold for the more general operators $P_t^\sigma$.
\end{remark}

If $n=\infty$, the $CD^{\sigma}$ inequality is equivalent to the following local functional inequalities of $P^{\sigma}_tf$.
\begin{thm}\label{thm:curvature-characterization}
Let $(G, \mu, \sigma)$ be given. Then the following are equivalent:
\begin{enumerate}[(i)]
\item The inequality $CD^{\sigma}(K, \infty)$ holds, i.e., for any function $f:V\to \mathbb{K}^d$, we have
$$\Gamma_2^\sigma(f)\geq K\Gamma^\sigma(f);$$\label{eq:cdsigma}
\item For any function $f:V\to \mathbb{K}^d$ and $t\geq 0$, we have
$$\Gamma^{\sigma}(P_t^{\sigma}f)\leq e^{-2Kt}P_t(\Gamma^{\sigma}(f));$$\label{eq:BEgradient}
\item For any function $f:V\to \mathbb{K}^d$ and $t\geq 0$, we have
$$P_t(|f|^2)-|P_t^{\sigma}f|^2\geq \frac{1}{K}(e^{2Kt}-1)\Gamma^{\sigma}(P_t^{\sigma}f),$$
where we replace $(e^{2Kt}-1)/K$ by $2t$ in the case $K=0$.
\label{eq:BEgradient2}
\end{enumerate}
\end{thm}
\begin{remark} Theorem \ref{thm:curvature-characterization} is similar in spirit to \cite[Propostion 3.3]{Bakry}. Note that the Proposition \ref{pro:heatsemigroupunsigned} (\ref{pro:Ptnonnegative}), which is crucial for the proof of \cite[Propostion 3.3]{Bakry}, is not true for $P_t^\sigma$ in general. However, with our definitions of the operators $\Gamma^\sigma$ and $\Gamma_2^\sigma$, we avoid this difficulty.

\end{remark}

\begin{proof}
(\ref{eq:cdsigma}) $\Rightarrow$ (\ref{eq:BEgradient}): For any $0\leq s\leq t$, we consider
\begin{equation}
F(s):=e^{-2Ks}P_s(\Gamma^{\sigma}(P_{t-s}^{\sigma}f)).
\end{equation}
Since $F(0)=\Gamma^{\sigma}(P_t^{\sigma}f)$ and $F(t)=e^{-2Kt}P_t(\Gamma^{\sigma}(f))$, it is enough to prove $\frac{d}{ds}F(s)\geq~0$.
We calculate
  \begin{equation*}
    e^{2Ks}\frac{d}{ds}F(s)=-2KP_s(\Gamma^\sigma(P^\sigma_{t-s}f))+\Delta P_s(\Gamma^\sigma(P^\sigma_{t-s}f))+P_s(\frac{d}{ds}\Gamma^\sigma(P^\sigma_{t-s}f)),
  \end{equation*}
and
\begin{equation*}
\frac{d}{ds}\Gamma^\sigma(P^\sigma_{t-s}f)=-\Gamma^\sigma(\Delta^\sigma P_{t-s}^\sigma f, P_{t-s}^\sigma f)-\Gamma^\sigma(P_{t-s}^\sigma f, \Delta^\sigma P_{t-s}^\sigma f).
\end{equation*}
Therefore, $\Delta P_s=P_s\Delta$ and the inequality $CD^{\sigma}(K, \infty)$ imply
 \begin{equation*}
    \frac{d}{ds}F(s)=2e^{-2Ks}P_s \big[ \Gamma^\sigma_2(P^\sigma_{t-s}f)-K\Gamma^\sigma(P^\sigma_{t-s}f) \big]\geq 0,
  \end{equation*}
where we used Proposition \ref{pro:heatsemigroupunsigned} (\ref{pro:Ptnonnegative}). This proves (\ref{eq:BEgradient}).

(\ref{eq:BEgradient}) $\Rightarrow$ (\ref{eq:BEgradient2}):  For $0\leq s\leq t$, we consider
\begin{equation}
G(s):=P_s(\left|P^\sigma_{t-s}f\right|^2).
\end{equation}
Note that $G(0)=\left| P^\sigma_tf \right|^2$ and $G(t)=P_t(|f|^2)$.  Using the estimate (\ref{eq:BEgradient}) and Proposition \ref{pro:heatsemigroup}, we have
  \begin{align*}
    \frac{d}{ds}G(s)&=\Delta P_s(\left| P^\sigma_{t-s}f \right|^2)+P_s\left[-(P^\sigma_{t-s}f)^T(\overline{\Delta^\sigma P^\sigma_{t-s}f})-(\Delta^\sigma P^\sigma_{t-s}f)^T\overline{P^\sigma_{t-s}f}\right]\\
    &=2P_s(\Gamma^\sigma(P^\sigma_{t-s}f))\geq 2e^{2Ks}\Gamma^\sigma(P^\sigma_tf).
  \end{align*}
Therefore, we obtain
  \begin{equation*}
    G(t)-G(0)=\int_0^t \frac{d}{ds}G(s) ds\geq 2\Gamma^\sigma(P^\sigma_tf)\int_0^te^{2Ks}ds= \frac{e^{2Kt}-1}{K}\Gamma^\sigma(P^\sigma_tf).
  \end{equation*}
  This proves (\ref{eq:BEgradient2}).

(\ref{eq:BEgradient2}) $\Rightarrow$ (\ref{eq:cdsigma}): Here, we consider the inequality (\ref{eq:BEgradient2}) at $t=0$ and use the expansion
\begin{equation*}
P_t^\sigma=\mathrm{Id}+t\Delta^\sigma+\frac{t^2}{2}(\Delta^\sigma)^2+o(t^2).
\end{equation*}
Dividing (\ref{eq:BEgradient2}) by $2t^2$ and letting $t$ tend to zero, we obtain
\begin{align*}
&\frac{1}{4}\Delta^2(|f|^2)-\frac{1}{4}f^T\left(\overline{(\Delta^\sigma)^2f}\right)-\frac{1}{4}\left((\Delta^\sigma)^2f\right)^T\overline{f}-\frac{1}{2}(\Delta^\sigma f)^T(\overline{\Delta^\sigma f})\\
\geq & K\Gamma^{\sigma}(f)+\Gamma^{\sigma}(f, \Delta^\sigma f)+\Gamma^{\sigma}(\Delta^\sigma f, f).
\end{align*}
Using Definition \ref{def:introGamma}, the above inequality simplifies to
\begin{equation*}
\Gamma^\sigma_2(f)\geq K\Gamma^{\sigma}(f),
\end{equation*}
which shows (\ref{eq:cdsigma}).
\end{proof}

\section{Multi-way Cheeger constants with signatures}\label{section:MultCheeSig}

In this section, we introduce multi-way Cheeger constants with signatures for graphs $(G,\mu, \sigma)$.

\subsection{Cheeger constants with signatures}\label{section:CheegerSig}
Following the ideas of \cite{LLPP15}, we introduce a Cheeger type constant of $(G, \mu, \sigma)$ as a mixture of a frustration index and the expansion rate.
For any nonempty subset $S\subseteq V$, the frustration index $\iota^\sigma(S)$ is a measure of the unbalancedness of the signature $\sigma$ on the induced subgraph of $S$. For that purpose, we need to choose a norm on $H$, to measure the distance between different elements in $H$.
\begin{definition}
Given a $(d\times d)$-matrix $A=(a_{ij})$, we define the \emph{average $(2,1)$-norm} $|A|_{2,1}$ of $A$ as
\begin{equation}\label{defn:matrixNorm}
|A|_{2,1}:=\frac 1 d \sum_{i=1}^d\left(\sum_{j=1}^d|a_{ij}|^2\right)^{\frac{1}{2}}.
\end{equation}
If we denote the vector of the $i$-th column of $A$ by $A^i$, this norm can be rewritten as $|A|_{2,1}:=\frac 1 d \sum_{i=1}^d |A^i|$. Recall that $|A^i|^2:=(A^i)^T\overline{A^i}$.
\end{definition}
\begin{remark}
\begin{enumerate}[(i)]
\item The average $(2,1)$-norm is smaller or equal to the Frobenius norm (alternatively, called Hilbert-Schmidt norm), i.e., we have for any $(d\times d)$-matrix $A=(a_{ij})$,
    \begin{equation}\label{eq:normwithFrobenius}
    |A|_{2,1}\leq \frac{1}{\sqrt{d}}|A|_F,
    \end{equation}
    where $|A|_F:=\left(\sum_{i,j=1}^d|a_{ij}|^2\right)^{\frac{1}{2}}$ denotes the Frobenius norm of $A$. This is a straightforward consequence of the Cauchy-Schwarz inequality directly.
\item The average $(2,1)$-norm is not sub-multiplicative in general, i.e. $|AB|_{2,1}\leq |A|_{2,1}|B|_{2,1}$ is not necessarily true for any $(d\times d)$-matrices $A$ and $B$. However, if $B\in O(d)$ or $U(d)$, we have
\begin{equation}\label{eq:normrotationinvariant}
|BA|_{2,1}=|A|_{2,1}.
\end{equation}
Note that in this case, $|B|_{2,1}=1$.
\end{enumerate}
\end{remark}


\begin{definition}[Frustration index]\label{defn:frustration index}
Let $(G,\mu,\sigma)$ be given. We define the frustration index $\iota^\sigma(S)$ for $\emptyset\neq S \subseteq V$ as
\begin{align*}
\iota^\sigma(S) :=& \min_{\tau: S \to  H}\sum_{\{x,y\}\in E_S}
w_{xy}|\sigma_{xy}\tau(y)-\tau(x)|_{2,1}\\
=&\min_{\tau: S \to  H}\sum_{\{x,y\}\in E_S}
w_{xy}|\sigma^\tau_{xy}-id|_{2,1},
\end{align*}
where $E_S$ is the edge set of the induced subgraph of $S$ in $G$.
\end{definition}
\begin{remark}\begin{enumerate}[(i)]
                \item By (\ref{eq:normrotationinvariant}), the quantity $$|\sigma_{xy}\tau(y)-\tau(x)|_{2,1}=|\sigma_{yx}\tau(x)-\tau(y)|_{2,1}$$ is independent of the orientation of the edge $\{x,y\}\in E$. Hence, the summation $\sum_{\{x,y\}\in E_S}
w_{xy}|\sigma_{xy}\tau(y)-\tau(x)|_{2,1}$ is well defined.
                \item In the definition of the frustration index, we are taking the infimum over all possible switching functions. Hence,
the frustration index is a switching invariant quantity.
                \item The average $(2,1)$-norm is only one possible choice which can be used in the definition of the frustration index. A more canonical norm to be used is the Frobenius norm. However, having the aim to present the strongest Buser type inequalities (\ref{eq:Buser}) in Section \ref{section:Buser}, we choose the average $(2,1)$-norm here (recall (\ref{eq:normwithFrobenius})).
              \end{enumerate}
\end{remark}

We denote by $|E(S, V\setminus S)|$ the boundary measure of $S\subseteq V$, which is given by
$$
|E(S, V\setminus S)|:=\sum_{x\in S}\sum_{y\not\in S}w_{xy}.
$$
In the above, we use the convention that $w_{xy}=0$ if $x\not\sim y$. The $\mu$-volume of $S$ is given by $$\mu(S):=\sum_{x\in S}\mu(x).$$

\begin{definition}\label{def:subpartition}
We call $k$ subsets $\{S_i\}_{i=1}^k$ of $V$ a \emph{nontrivial $k$-subpartition} of $V$, if all $S_i$ are nonempty and pairwise disjoint.
\end{definition}
Now we are prepared to define the Cheeger constant.
\begin{definition}[Cheeger constant]\label{defn:Cheeger} Let $(G, \mu, \sigma)$ be given.
The \emph{$k$-way Cheeger constant} $h_k^\sigma$ is defined as
\[
h_k^\sigma := \min_{\{S_i\}_{i=1}^k} \max_{1 \leq i \leq k} \phi^\sigma(S_i),
\]
where the minimum is taken over all possible nontrivial  $k$-subpartitions $\{S_i\}_{i=1}^k$ of $V$ and
\[
\phi^\sigma(S):=\frac{\iota^\sigma(S) + |E(S,V\setminus S)|}{\mu(S)}.
\]
\end{definition}
Note that the multi-way Cheeger constants defined above are switching invariant. Definition~\ref{defn:Cheeger} is a natural extension of the Cheeger constants developed in \cite{AtayLiu14,LLPP15}, and is related to the constants discussed in \cite{BSS13}.
\begin{remark}[Relations with Bandeira, Singer and Spielman's constants]\label{remark:BSS}
In \cite{BSS13}, Bandeira, Singer and Spielman introduced the so-called \emph{$O(d)$ frustration $\ell_1$ constant} $\nu_{G,1}$ as follows,
\begin{equation*}
\nu_{G,1}:=\min_{\tau: V\to O(d)}\frac{1}{\sqrt{d}\mu(V)}\sum_{x,y\in V}w_{xy}|\sigma_{xy}\tau(y)-\tau(x)|_F,
\end{equation*}
where $|\cdot|_F$ denotes the Frobenius norm of a matrix.
Modifying $\nu_{G,1}$ by also allowing zero matrices in the image of $\tau$, we obtain
\begin{equation*}
\nu^*_{G,1}:=\min_{\tau: V\to H\cup\{0\}}\frac{\sum_{x,y\in V}w_{xy}|\sigma_{xy}\tau(y)-\tau(x)|_F}{\sum_{x\in V}\mu(x)|\tau(x)|_F},
\end{equation*}
where we denote the $(d\times d)$-zero matrix by $0$. Note that $|\tau(x)|_F=\sqrt{d}$, for $\tau(x)\in H$.

We observe that the following relation between our Cheeger constant $h_1^\sigma$ and the constant $\nu^*_{G,1}$:
\begin{equation*}
h_1^\sigma\leq \frac{1}{2}\nu^*_{G,1},
\end{equation*}which is verified by the following calculation:
\begin{align*}
h_1^\sigma=&\min_{\emptyset\neq S\subseteq V}\frac{\iota^\sigma(S)+|E(S,V\setminus S)|}{\mu(S)}\\
=&\min_{\tau:V\to H\cup\{0\}}\frac{\sum_{\{x,y\}\in E}w_{xy}|\sigma_{xy}\tau(y)-\tau(x)|_{2,1}}{\sum_{x\in V}\mu(x)|\tau(x)|_{2,1}}\\
\leq&\min_{\tau:V\to H\cup\{0\}}\frac{\sum_{\{x,y\}\in E}w_{xy}|\sigma_{xy}\tau(y)-\tau(x)|_{F}}{\sqrt{d}\sum_{x\in V}\mu(x)|\tau(x)|_{2,1}}\\
=&\frac{1}{2}\nu^*_{G,1}.
\end{align*}
In the inequality above, we used (\ref{eq:normwithFrobenius}).
\end{remark}

For convenience, we call $\{S_i\}_{i=1}^k$ a \emph{connected}, nontrivial $k$-subpartition of $V$, if all sets $S_i\subseteq V$ are nonempty and pairwise disjoint, and if every subgraph induced by $S_i$ is connected. Then the Cheeger constants introduced in Definition \ref{defn:Cheeger} do not change if we restrict our considerations to connected, nontrivial $k$-subpartitions:

\begin{lemma}\label{lemma:connectedCheeger} Let $(G, \mu, \sigma)$ be given. Then we have
\[
h_k^\sigma = \min_{\{S_i\}_{i=1}^k} \max_{1 \leq i \leq k} \phi^\sigma(S_i),
\]
where the minimum is taken over all possible connected, nontrivial  $k$-subpartitions $\{S_i\}_{i=1}^k$ of $V$.
\end{lemma}

\begin{proof}
Let $\{S_i\}_{i=1}^k$ be a possibly nonconnected, nontrivial  $k$-subpartition achieving $h_k^\sigma$.
Suppose $S_i$ has the connected components $W^1_i, \ldots, W^{n(i)}_i$. Then,
\[
\phi^\sigma(S_i) \mu(S_i) = \sum_{j=1}^{n(i)} \phi^\sigma(W^j_i)  \mu(W^j_i)
\]
and $\mu(S_i) = \sum_{j=1}^n \mu(W^j_i)$. Hence, there exists $j(i)\in\{1,2,\ldots, n(i)\}$ such that $\phi^\sigma(W^{j(i)}_i) \leq \phi^\sigma(S_i)$.
Consequently,
\[
\max_{1 \leq i \leq k} \phi^\sigma(W^{j(i)}_i) \leq \max_{1 \leq i \leq k} \phi^\sigma(S_i) = h_k^\sigma,
\]
and thus, $\{W^{j(i)}_i\}_{i=1}^k$ is a connected, nontrivial  $k$-subpartition of $V$ with
\[
h_k^\sigma = \max_{1 \leq i \leq k} \phi^\sigma(W^{j(i)}_i)
\]
This implies the lemma.
\end{proof}

\subsection{Frustration index via spanning trees}\label{subsection:spanning tree} This subsection is motivated by the following question:
 Is there any easier way to calculate the frustration index $\iota^\sigma(S)$ of a subset $S\subseteq V$?
We will provide an affirmative answer in the case $H=U(1)$.

Note that the average $(2,1)$-norm reduces to the absolute value of a complex number, and  the frustration index $\iota^\sigma(S)$ for $S \subseteq V$ simplifies to
\[
\iota^\sigma(S) := \min_{\tau: S \to  U(1)}\sum_{\{x,y\}\in E_S}
w_{xy}|\sigma_{xy}\tau(y)-\tau(x)|,
\]
where $E_S$ is the edge set of the induced subgraph of $S$.
Here our aim is to make it easier to calculate $\iota^{\sigma}(S)$ by considering all spanning trees of the induced subgraph of $S$ and taking the minmum over so-called \emph{constant functions on these trees with respect to the signature}. This reduces the original minimization problem to a finite combinatorial problem. We will show in Appendix \ref{section:appendixSpanningTree} via a counterexample that this reduction is no longer possible in the case of higher dimensional signature groups.
\begin{definition}
Let $(G,\sigma)$ be a finite, connected graph the signature $\sigma: E^{or}\to H$. A function $\tau: V\to H$ is \emph{constant on $G$ with respect to $\sigma$} if, for all $(x,y) \in E^{or}$, we have
\[
\sigma_{xy}\tau(y) = \tau(x).
\]
In other words, $\tau$ is a switching function such that $\sigma^\tau$ is trivial, i.e., $\sigma^\tau_{xy}=id\in H$ for all $(x,y)\in E^{or}$.
\end{definition}

Let $T=(S, E_T)$, $E_T\subseteq E_S$, be a spanning tree of the induced subgraph of $S$. We write $C_T(S):=\{\tau:S \to U(1) : \tau \mbox{ is constant on } T
\mbox{ with respect to } \sigma\}$. Moreover, we define $\mathbb{T}_S$ as the set of all of all spanning trees of the induced subgraph of $S$.


Since $T$ is a tree, the set $C_T(S)$ is not empty. Since $T$ is a spanning tree, we have $C_T(S)= \tau U(1) :=\{\tau z: S\to U(1)\mid z\in U(1)\}$ for any $\tau \in C_T(S)$.

\begin{thm}\label{thm:spanning tree}
Let $S \subseteq V$ be a nonempty subset of $V$ which induces a connected subgraph. Then,
\begin{equation}\label{eq:spanning tree}
\iota^\sigma(S) = \min_{T \in \mathbb{T}_S}\sum_{\{x,y\}\in E_S}
w_{xy}|\sigma_{xy}\tau_T(y)-\tau_T(x)|,
\end{equation}
where $\tau_T$ denotes an arbitrary representative of $C_T(S)$.

Moreover, if a function $\tau:S\to U(1)$ satisfies $\sum_{\{x,y\}\in E_S}  w_{xy}|\sigma_{xy}\tau(y)-\tau(x)| = \iota^\sigma(S)$, then there is a spanning tree $T=(S, E_T)$ such that $\tau$ is constant on $T$ with respect to $\sigma$.
\end{thm}
We remark that in (\ref{eq:spanning tree}) we are taking the minimum over a finite set. Moreover, given a spanning tree $T\in \mathbb{T}_S$, only terms associated to edges of $E_S$ not belonging to the spanning tree contribute to the sum.

Theorem \ref{thm:spanning tree} can be considered an an extension of \cite[Theorem 2]{HararyKabell80}, where Harary
and Kabell derived this result on unweighted graphs for the case $H=O(1)=\{\pm 1\}$. Their proof depends in an essential way on the fact that the frustration index in their setting (which they called \emph{line index of balance}) only assumes integer values. Therefore, their proof cannot be extended to the current general setting.

We first prove a basic lemma.

\begin{lemma}\label{lemma:1dimgeometry}
Let $Z:= \{z_1,\ldots,z_n\} \subset U(1)$ and $w_1,\ldots,w_n > 0$. Then we have
\begin{equation}\label{eq:minsimp}
\min_{z\in U(1)} \sum_{k=1}^n w_k |z-z_k|   =   \min_{z\in Z} \sum_{k=1}^n w_k |z-z_k|.
\end{equation}
Moreover, if  $z \in U(1) \setminus Z$, then
\[
 \sum_{k=1}^n w_k |z-z_k|   >   \min_{z\in Z} \sum_{k=1}^n w_k |z-z_k|.
\]
\end{lemma}
\begin{proof}
The minimum over on the left hand side of (\ref{eq:minsimp}) exists, since $U(1)$ is compact and $\sum_{k=1}^n w_k |z-z_k|$ is continuous in $z$. Suppose that the minimum is assumed in $z_0= e^{it_0}$ with $z_0 \notin Z$.
That is, the function $\phi: \R \to \R, t \mapsto \sum_{k=1}^n w_k |e^{it}-z_k|$ assumes a local minimum in $t_0$. Since $z_0 \notin Z$, the second derivative $\phi''$ exists at $t_0$ and is not negative due to the minimum property.
But for all $k \in \{1,\ldots,n\}$, we can set $z_k=e^{i t_k}$ and compute
\[
\frac {d^2} {dt^2}  |e^{it}-e^{i t_k}| (t_0) =  2 \frac {d^2} {dt^2}  \left| \sin \frac{t-t_k} 2 \right| (t_0) < 0.
\]
This is a contradiction and, hence, $z_0 \in Z$. This finishes the proof of the lemma.
\end{proof}

Now, we prove the theorem with the help of the lemma.
\begin{proof}[Proof of Theorem \ref{thm:spanning tree}]
First, we notice that the expression $w_{xy}|\sigma_{xy}\tau_T(y)-\tau_T(x)|$ does not depend on the choice of $\tau_T \in C_T(S)$ since $C_T(S)= \tau_TU(1)$.
Hence, the restriction to the representative of $C_T(S)$ make sense.

Let $\tau_0:S \to U(1)$ be a minimizer of
$\sum_{\{x,y\}\in E_S}
w_{xy}|\sigma_{xy}\tau(y)-\tau(x)|.$
Denote by
\[
E_0:= \{\{x,y\}\in E_S : \sigma_{xy}\tau_0(y) = \tau_0(x)  \}
\]
the set of edges where $\tau_0$ is constant with respect to $\sigma$.
It is sufficient to show that $G_0 = (S,E_0)$ is connected since then there is a spanning tree $T_0$ of $G_0$ such that $\tau_0$ is constant on $T_0$ with respect to $\sigma$.

Suppose $G_0$ is not connected. Then there is a connected component $W \subsetneq S$.
Denote $\partial_S W := \{(x,y)\in E^{or}_S : x \in W, y \in S \setminus W\}$. We have $\partial_S W \neq \emptyset$ since $S$ is connected. Moreover, we have $\sigma_{xy}\tau_0(y)  \neq \tau_0(x)$ for all $(x,y) \in \partial_S W$, since $W$ is a connected component and, otherwise, $y$ would also belong to $W$, contradicting to $(x,y) \in \partial_S W$.

The previous lemma states that
\[
\min_{z\in U(1)}\sum_{(x,y)\in \partial_S W }   w_{xy}|\sigma_{xy}\tau_0(y)- z \tau_0(x)| = \min_{z\in U(1)}\sum_{(x,y)\in \partial_S W }   w_{xy}|\sigma_{xy}\tau_0(y)\overline{\tau_0(x)} - z |
\]
achieves the minimum only in elements of the set
$\{  \sigma_{xy}\tau_0(y)\overline{\tau_0(x)} : (x,y) \in \partial_S W   \}$.
But $1 \notin \{  \sigma_{xy}\tau_0(y)\overline{\tau_0(x)} : (x,y) \in \partial_S W   \}$, since $\sigma_{xy}\tau_0(y)  \neq \tau_0(x)$ for all $(x,y) \in \partial_S W$.
Hence, there exists $z_0\in U(1)$ such that
\begin{equation}
    \sum_{(x,y)\in \partial_S W }   w_{xy}|\sigma_{xy}\tau_0(y)- z_0 \tau_0(x)|
<   \sum_{(x,y)\in \partial_S W }   w_{xy}|\sigma_{xy}\tau_0(y)-  \tau_0(x)|.    \label{tau1 tau0}
\end{equation}
We define $\tau_1: S\to U(1)$,
\[
\tau_1(x) := \begin{cases}
							z_0 \tau_0(x),  & \mbox{ if } x \in W; \\
							\tau_0(x),    & \mbox{ if }x \in S\setminus	W.
						 \end{cases}			
\]
Consequently,
\begin{eqnarray*}
   &&  \sum_{\{x,y\}\in E_S}  w_{xy}|\sigma_{xy}\tau_1(y)-\tau_1(x)| \\
&=&  \sum_{\{x,y\}\in E_W \cup E_{S\setminus W}}  w_{xy}|\sigma_{xy}\tau_1(y)-\tau_1(x)|  +  \sum_{(x,y)\in \partial_S W}  w_{xy}|\sigma_{xy}\tau_1(y)-\tau_1(x)|   \\
&=&  \sum_{\{x,y\}\in E_W \cup E_{S\setminus W}}  w_{xy}|\sigma_{xy}\tau_0(y)-\tau_0(x)| 	+  \sum_{(x,y)\in \partial_S W}  w_{xy}|\sigma_{xy}\tau_0(y)- z_0 \tau_0(x)|	\\
&\stackrel{(\ref{tau1 tau0})}{<}&  \sum_{\{x,y\}\in E_W \cup E_{S\setminus W}}  w_{xy}|\sigma_{xy}\tau_0(y)-\tau_0(x)| 	+  \sum_{(x,y)\in \partial_S W}  w_{xy}|\sigma_{xy}\tau_0(y)-  \tau_0(x)|	\\
&=&  \sum_{\{x,y\}\in E_S}  w_{xy}|\sigma_{xy}\tau_0(y)-\tau_0(x)|.
\end{eqnarray*}

This is a contradiction to the fact that $\tau_0$ is a minimizer of $\sum_{\{x,y\}\in E_S}
w_{xy}|\sigma_{xy}\tau(y)-~\tau(x)|$. Thus, $G_0$ has to be connected. This finishes the proof.
\end{proof}

Recall from Lemma \ref{lemma:connectedCheeger} that the Cheeger constant $h_k^\sigma$ is the minimum of $\max_{1\leq i\leq k}\phi^{\sigma}(S_i)$ over
all possible \emph{connected}, nontrivial $k$-subpartitions $\{S_i\}_{i=1}^k$. Therefore, Theorem \ref{thm:spanning tree} implies that the calculation of $h_k^\sigma$ reduces to a
finite combinatorial minimization problem if $H=U(1)$.

\section{Buser inequalities}\label{section:Buser}\label{section:BusIne}
%
In this section, we prove our main theorem, namely, higher order Buser type inequalities for nonnegatively curved graphs (cf. Theorem \ref{thm:introMain} in the Introduction).
\begin{thm}[Main Theorem]\label{thm:Buser}
  Let $(G,\mu, \sigma)$ satisfy $CD^\sigma(0,\infty)$. Then for all $1\leq k\leq N$, we
  have
  \begin{equation}\label{eq:Buser}
       \sqrt{\lambda^\sigma_{kd}}\leq 4 \sqrt{D_G^{nor}}  \left(kd \sqrt{\log(2kd)}\right)   h_k^\sigma.
  \end{equation}
\end{thm}

Before we present the proof, we first discuss the following two lemmata.
We will use the following notation for the $\ell^{p}(V,\mathbb{K}^d;\mu)$ norm of functions, $1\leq p\leq \infty$,
$$\Vert f\Vert_{p,\mu}:=\left(\sum_{x\in V}\mu(x)|f(x)|^p\right)^{\frac{1}{p}}.$$
For simplicity, we omit the subscript $\mu$ in the following arguments.
\begin{lemma}\label{lemma:lonenorm}
Assume that $(G,\mu, \sigma)$ satisfies $CD^\sigma(0,\infty)$. Then for any function $f: V\to \mathbb{K}^d$ and $t\geq 0$,
we have
  \begin{equation}\label{eq:lonenorm}
    \Vert f-P^\sigma_tf\Vert_1\leq \sqrt{2t}\Vert\sqrt{\Gamma^\sigma(f)}\Vert_1.
  \end{equation}
\end{lemma}
\begin{proof}
First, the equivalent formulation of the $CD^{\sigma}(0,\infty)$ inequality in Theorem \ref{thm:curvature-characterization} (\ref{eq:BEgradient2}) implies that
\begin{equation}\label{eq:linfty}
\Vert \sqrt{P_t(|f|^2)}\Vert_{\infty}\geq \sqrt{2t}\Vert \sqrt{\Gamma^\sigma(P^\sigma_tf)}\Vert_{\infty}.
\end{equation}
The inequality (\ref{eq:lonenorm}) is actually a dual version of the above one.
We set $$g(x):=\left\{
              \begin{array}{ll}
                0, & \hbox{if $f(x)-P_t^\sigma f(x)=0$;} \\
                (f(x)-P_t^\sigma f(x))/|f(x)-P_t^\sigma f(x)|, & \hbox{otherwise,}
              \end{array}
            \right.
$$
and calculate
\begin{align*}
 \Vert f-P^\sigma_tf\Vert_1&=\langle f-P_t^\sigma f, g\rangle_{\mu}=\langle -\int_0^t\frac{\partial}{\partial s}P_s^\sigma fds, g\rangle_{\mu}\\
 &=-\int_0^t\langle \Delta^\sigma f, P_s^{\sigma}g\rangle_{\mu}ds\\
 &=\int_0^t\sum_{x\in V}\mu(x)\Gamma^\sigma(f, P_s^\sigma g)(x)ds,
\end{align*}
where we used the self-adjointness of $P_t^\sigma$ and the summation by part formula (\ref{eq:summaBypart}). We further apply Proposition \ref{pro:Gammasigma} and the estimate (\ref{eq:linfty}) to derive
\begin{align*}
 \Vert f-P^\sigma_tf\Vert_1&\leq \int_0^t\sum_{x\in V}\mu(x)\sqrt{\Gamma^\sigma(f)(x)}\sqrt{\Gamma^\sigma(P_s^\sigma g)(x)}ds\\
 &\leq \int_0^t\Vert\sqrt{\Gamma^\sigma(f)}\Vert_1\Vert\sqrt{\Gamma^\sigma(P_s^\sigma g)}\Vert_{\infty}ds\\
 &\leq \Vert\sqrt{\Gamma^\sigma(f)}\Vert_1\int_0^t\frac{1}{\sqrt{2s}}\Vert \sqrt{P_s(|g|^2)}\Vert_{\infty}ds\\
 &\leq \sqrt{2t}\Vert\sqrt{\Gamma^\sigma(f)}\Vert_1.
\end{align*}
In the last inequality we used the fact $P_s(|g|^2)\leq \Vert|g|^2\Vert_{\infty}=1$, which follows from Proposition \ref{pro:heatsemigroupunsigned}.
\end{proof}

We still need the following technical lemma.
\begin{lemma}\label{lemma:sqrtGamma}
For any function $f: V\to \mathbb{K}^d$, we have
\begin{equation}\label{eq:sqrtGamma}
\Vert\sqrt{\Gamma^\sigma(f)}\Vert_1\leq \sqrt{2D_G^{nor}}\sum_{\{x,y\}\in E}w_{xy}|\sigma_{xy}f(y)-f(x)|
\end{equation}
\end{lemma}
\begin{proof}
It is straightforward to calculate
  \begin{align*}
    \Vert\sqrt{\Gamma^\sigma(f)}\Vert_1&=\sum_{x\in V}\mu(x)\sqrt{\frac{1}{2\mu(x)}\sum_{y,y\sim x}w_{xy}\left|\sigma_{xy}f(y)-f(x)\right|^2}\\
    &\leq \sum_{x\in V}\sqrt{\frac{\mu(x)}{2}}\sum_{y,y\sim x}\sqrt{w_{xy}}\left|\sigma_{xy}f(y)-f(x)\right|\\
    &\leq \sqrt{\frac{D_G^{nor}}{2}}\sum_{x\in V}\sum_{y,y\sim x}w_{xy}\left|\sigma_{xy}f(y)-f(x)\right|.
  \end{align*}
This simplifies to (\ref{eq:sqrtGamma}), since the summands above are symmetric w.r.t. $x$ and $y$.
\end{proof}

Now, we have all ingredients for the proof of the Buser type inequality (\ref{eq:Buser}).

\begin{proof}[Proof of Theorem~\ref{thm:Buser}]
Let $\{S_i\}_{i=1}^k$ be an arbitrary nontrivial $k$-subpartition of $V$. For each $S_i$, let $\tau_i: S_i\to H$ be the function achieving the values $\iota^\sigma(S_i)$ introduced in Definition \ref{defn:frustration index}. We extend each $\tau_i$ trivially to a function on $V$, by assigning zero matrices to the vertices in $V\setminus S_i$. By abuse of notation, we denote this extension, again, by $\tau_i: V\to H$. Each $\tau_i$ gives rise to $d$ pairwise orthogonal functions in $\ell^2(V, \mathbb{K}^d;\mu)$:
\begin{equation}
\tau_i^l: V\to \mathbb{K}^{d},\,\,\,x\mapsto \left(\tau_i(x)\right)^l, \,\,l=1,2,\ldots, d,
\end{equation}
where $\left(\tau_i(x)\right)^l$ denotes for the $l$-th column vector of the matrix $\tau_i(x)$. Note that for $x\in S$, we have $|\tau_i^l(x)|=1$.

For every $1\leq i\leq k$, we apply Lemma \ref{lemma:sqrtGamma} to obtain
\begin{align}
&\frac 1 d \sum_{l=1}^d
\Vert\sqrt{\Gamma^\sigma(\tau_i^l)}\Vert_1  \notag\\
\leq & \frac 1 d \sum_{l=1}^d   \sqrt{2D_G^{nor}}\left(\sum_{\{x,y\}\in E_{S_i}}w_{xy}|\sigma_{xy}\tau_i^l(y)-\tau_i^l(x)|+|E(S_i,V\setminus S_i)|\right)  \notag\\
\leq &\sqrt{2D_G^{nor}}(\iota^\sigma(S_i)+|E(S_i,V\setminus S_i)|).\label{eq:B-tobecollone}
\end{align}
On the other hand, we have by Lemma \ref{lemma:lonenorm},
\begin{align}
\sqrt{2t}\Vert\sqrt{\Gamma^\sigma(\tau_i^l)}\Vert_1\geq & \sum_{x\in V}\mu(x)\left|\tau_i^l(x)-P_t^\sigma\tau_i^l(x)\right|\notag\\
\geq &\sum_{x\in V}\mu(x)\left|\tau_i^l(x)-P_t^\sigma\tau_i^l(x)\right|\cdot|\tau_i^l(x)|\notag\\
\geq & \sum_{x\in V}\mu(x)\mathrm{Re}\left((\tau_i^l(x))^T(\overline{\tau_i^l(x)-P_t^\sigma\tau_i^l(x)})\right),\notag
\end{align}
where $\mathrm{Re}(\cdot)$ denotes the real part of a complex number, and we used the Cauchy-Schwarz inequality in the last inequality. By Proposition \ref{pro:heatsemigroup}, we continue to calculate
\begin{align}\label{eq:B-tobecolltwo}
\sqrt{2t}\Vert\sqrt{\Gamma^\sigma(\tau_i^l)}\Vert_1\geq  \mathrm{Re}\left( \langle \tau_i^l, \tau_i^l-P_t^\sigma\tau_i^l \rangle_{\mu}  \right)=\Vert\tau_i^l\Vert_2^2-\Vert P_{t/2}^\sigma\tau_i^l\Vert_2^2.
\end{align}
Let $\{\psi_n\}_{n=1}^{Nd}$ be an orthonormal basis of $\ell^2(V, \mathbb{K}^d; \mu)$ consisting of the eigenfunctions corresponding to $\{\lambda^\sigma_n\}_{n=1}^{Nd}$, respectively. Setting
$$\alpha_{i,n}^l:=\langle \tau_i^l, \psi_n\rangle_{\mu},$$
we have
\begin{equation}\label{eq:B-tobecollthree}
\sum_ {n=1}^{Nd}  \left| \alpha_{i,n}^l \right|^2 =\Vert\tau_i^l\Vert^2_2= \mu(S_i),
\end{equation}
and
\begin{equation}\label{eq:B-tobecollfour}
\Vert P_{t/2}^\sigma\tau_i^l\Vert_2^2=\sum_ {n=1}^{Nd}e^{-t\lambda^\sigma_n}\left| \alpha_{i,n}^l \right|^2.
\end{equation}
Now (\ref{eq:B-tobecollone}), (\ref{eq:B-tobecolltwo}), (\ref{eq:B-tobecollthree}), and (\ref{eq:B-tobecollfour}) together imply, for each $1\leq i\leq k$,
\begin{align}\label{eq:phi-lambda}
                                2\sqrt{D_G^{nor}t}\phi^{\sigma}(S_i)
\geq&   \frac 1 d \sum_{l=1}^d   \left(    1-\sum_ {n=1}^{Nd}e^{-t\lambda^\sigma_n}\frac{\left| \alpha_{i,n}^l \right|^2}{\mu(S_i)}   \right) \notag  \\
\geq&\frac 1 d \sum_{l=1}^d   \left(    1-\sum_ {n=1}^{kd-1}\frac{\left| \alpha_{i,n}^l \right|^2}{\mu(S_i)}-e^{-t\lambda_{kd}^\sigma}\sum_{n=kd}^{Nd}\frac{\left| \alpha_{i,n}^l \right|^2}{\mu(S_i)} \right) \notag  \\
\geq&    1-  \frac 1 d \sum_{l=1}^d \sum_{n=1}^{kd-1}\frac{\left| \alpha_{i,n}^l \right|^2}{\mu(S_i)}   -    e^{-t\lambda_{kd}^\sigma}.
\end{align}
By (\ref{eq:B-tobecollthree}), we know
\begin{equation*}
1-  \frac 1 d \sum_{l=1}^d \sum_{n=1}^{kd-1}\frac{\left| \alpha_{i,n}^l \right|^2}{\mu(S_i)}\geq 0,
\end{equation*}
but our aim is to show that for some $i_0\in \{1,2,\ldots,k\}$ this expression is strictly positive. We rewrite the summands as follows,
\begin{equation}
\frac{\left| \alpha_{i,n}^l \right|^2}{\mu(S_i)}=\left|\left\langle\frac{\tau_i^l}{\sqrt{\mu(S_i)}}, \psi_n\right\rangle\right|^2.
\end{equation}
Since the functions $\tau_i^l/\sqrt{\mu(S_i)}, i=1,2,\ldots, k, l=1,2,\ldots,d,$ are orthonormal in the space $\ell^2(V, \mathbb{K}^d;\mu)$, we obtain
\begin{equation}
\sum_{i=1}^k\sum_{l=1}^d\frac{\left| \alpha_{i,n}^l \right|^2}{\mu(S_i)}\leq \Vert\psi_n\Vert_2^2=1.
\end{equation}
Summation over $n$ yields
\begin{equation}
\sum_{i=1}^k\sum_{l=1}^d\sum_{n=1}^{kd-1}\frac{\left| \alpha_{i,n}^l \right|^2}{\mu(S_i)}\leq kd-1.
\end{equation}
Consequently, there exists an $i_0\in \{1,2,\ldots, k\}$ such that
\begin{equation}
\frac 1 d \sum_{l=1}^d  \sum_{n=1}^{kd-1}\frac{\left| \alpha_{i_0,n}^{l} \right|^2}{\mu(S_i)}\leq 1-\frac{1}{kd}.
\end{equation}
We insert this estimate into inequality (\ref{eq:phi-lambda}) to obtain
\begin{equation}
2\sqrt{D_G^{nor}t}\max_{1\leq i\leq k}\phi^{\sigma}(S_{i})\geq \frac{1}{kd}-e^{-t\lambda_{kd}^\sigma}.
\end{equation}
Since the $k$-subpartition was chosen arbitrarily, we have
\begin{equation}
2  \sqrt{D_G^{nor}t} \cdot h_k^\sigma  \geq \frac{1}{kd} - e^{-t\lambda_{kd}^\sigma}.
\end{equation}
For $\lambda_{kd}\neq 0$, we choose $t=\log(2dk)/\lambda_{kd}^\sigma$ to obtain
\begin{equation}
4 \sqrt{D_G^{nor}}kd \sqrt{\log(2dk)}   h_k^\sigma   \geq \sqrt{\lambda_{kd}^\sigma}.
\end{equation}
This completes the proof.
\end{proof}
Recall from Corollary \ref{cor:lower curvature bound} that any graph $(G,\mu,\sigma)$ has a specific finite lower curvature bound. In case of a negative lower curvature bound, we have the following result.
For a subset $S\subseteq V$, we define the following constant, which is no greater than $\iota^\sigma(S)$,
\begin{equation}
\widetilde{\iota^\sigma}(S) := \min_{\substack{f: S \to  \mathbb{K}^{d}\\ |f(x)|=1,\,\forall x\in S}}\sum_{\{x,y\}\in E_S}
w_{xy}|\sigma_{xy}f(y)-f(x)|.
\end{equation}
Using this constant, we have the following isoperimetric type inequality.

\begin{thm}\label{thm:isoperimetry}
 Let $(G,\mu, \sigma)$ satisfy $CD^\sigma(-K,\infty)$, for $K\geq 0$. Then for any subset $\emptyset\neq S\subseteq V$,
 \begin{equation}
 \widetilde{\iota^\sigma}(S)+|E(S, V\setminus S)|\geq\frac{1}{2\sqrt{2D_G^{nor}}}\min\left\{(1-e^{-1})\sqrt{\lambda_1^\sigma}, \frac{\lambda_1^\sigma}{2\sqrt{2K}}\right\}\mu(S).
 \end{equation}
\end{thm}
\begin{proof}
Modifying the proof of Lemma \ref{lemma:lonenorm} for $K\geq 0$, we derive from the inequality $CD^\sigma(-K, \infty)$ that for any function $f: V\to \mathbb{K}^d$,
\begin{equation}
\Vert f-P^\sigma_tf\Vert_1\leq \int_0^t\sqrt{\frac{K}{1-e^{-2Ks}}}ds\Vert\sqrt{\Gamma^\sigma(f)}\Vert_1.
\end{equation}
Using $1-e^{-u}\geq u/2$ for $0\leq u\leq 1$, we have for $0\leq t\leq 1/(2K)$,
\begin{equation}
\Vert f-P^\sigma_tf\Vert_1\leq 2\sqrt{t}\Vert\sqrt{\Gamma^\sigma(f)}\Vert_1.
\end{equation}
Let $S$ be an arbitrary nonempty subset of $V$. Let $f^0: S\to \mathbb{K}^{d}$ with $|f^0(x)|=1$ for all $x\in S$ be the function achieving the value of $\widetilde{\iota^\sigma}(S)$. By similar reasoning as in the proof of Theorem \ref{thm:Buser}, we obtain for $0\leq t\leq 1/(2K)$,
\begin{equation}
2\sqrt{2D_G^{nor}t}\left(\widetilde{\iota^\sigma}(S)+|E(S, V\setminus S)|\right)\geq \mu(S)(1-e^{-t\lambda_1^\sigma}).
\end{equation}
If $\lambda_1^\sigma\geq 2K$, we set $t=1/\lambda^\sigma_1$ and obtain
\begin{equation}
2\sqrt{2D_G^{nor}}\left(\widetilde{\iota^\sigma}(S)+|E(S, V\setminus S)|\right)\geq \mu(S)(1-e^{-1})\sqrt{\lambda_1^\sigma}.
\end{equation}
If $\lambda_1^\sigma< 2K$, we set $t=1/(2K)$ and obtain
\begin{equation}
2\sqrt{\frac{D_G^{nor}}{K}}\left(\widetilde{\iota^\sigma}(S)+|E(S, V\setminus S)|\right)\geq \mu(S)(1-e^{-\frac{\lambda_1^\sigma}{2K}})\geq \mu(S)\frac{\lambda_1^\sigma}{4K}.
\end{equation}
Combining both cases completes the proof.
\end{proof}
We now define the following Cheeger type constant $\widetilde{h_1^\sigma}$ corresponding to $\widetilde{\iota^\sigma}(S)$.
\begin{definition}
 Let $(G,\mu,\sigma)$ be given. The constant $\widetilde{h_1^\sigma}$ is defined as
\begin{equation*}
 \widetilde{h_1^\sigma}=\min_{\emptyset\neq S\subseteq V}\frac{\widetilde{\iota^\sigma}(S)+|E(S,V\setminus S)|}{\mu(S)}.
\end{equation*}
\end{definition}
By definition, we observe that $\widetilde{h_1^\sigma}\leq h_1^\sigma$. Theorem \ref{thm:isoperimetry} implies the following estimate immediately.
\begin{coro}\label{thm:Buser-K}
  Let $(G,\mu, \sigma)$ satisfy $CD^\sigma(-K,\infty), K\geq 0$. Then we
  have
  \begin{equation}\label{eq:Buser-K}
       \lambda^\sigma_{1}\leq 8\max\{(e/(e-1))^2 D_G^{nor}(\widetilde{h_1^\sigma})^2, \sqrt{D_G^{nor}K}\widetilde{h_1^\sigma}\}.
  \end{equation}
\end{coro}
Note that, for the constant $\widetilde{h_1^\sigma}$, the following Cheeger type inequality is proved in \cite[Theorem 4.1]{BSS13} (see also \cite[Theorem 4.6 and Remark 4.9]{LLPP15}).
\begin{thm}[\cite{BSS13}]\label{thm:Cheeger}
 Let $(G,\mu,\sigma)$ be given. Then we have
\begin{equation}
 \frac{2}{5D_G^{non}}\widetilde{h_1^\sigma}^2\leq \lambda_1^\sigma\leq 2\widetilde{h_1^\sigma}.
\end{equation}
\end{thm}
\begin{example}[Signed Triangle] We revisit the example of a signed triangle discussed in Section \ref{subsection:triangle} (see Figure \ref{F1}). In this case we have $H=U(1)$ and, therefore, $h_1^\sigma=\widetilde{h_1^\sigma}$. Using Theorem \ref{thm:spanning tree}, we can check
$$h_1^\sigma=\frac{|s-1|}{6}=\frac{\sqrt{2(1-\mathrm{Re}(s))}}{6}.$$
The Buser type inequality Theorem \ref{thm:Buser} tells us
\begin{equation}\label{eq:triangle Buser}
 \lambda_1^\sigma\leq 32\log 2(h_1^\sigma)^2,
\end{equation}
while the Cheeger type inequality Theorem \ref{thm:Cheeger} gives
\begin{equation}\label{eq:triangle Cheeger}
 \frac{2}{5}(h_1^\sigma)^2\leq \lambda_1^\sigma\leq 2h_1^\sigma.
\end{equation}
The comparison of the estimates (\ref{eq:triangle Buser}) and (\ref{eq:triangle Cheeger}) is shown in Figure \ref{FtriangleComp}, where we treat the quantities $\lambda_1^\sigma$ and $h_1^\sigma$ as functions of the variable $\mathrm{Re}(s)$.
\end{example}
\begin{figure}[h]
\centering
\includegraphics[width=0.5\textwidth]{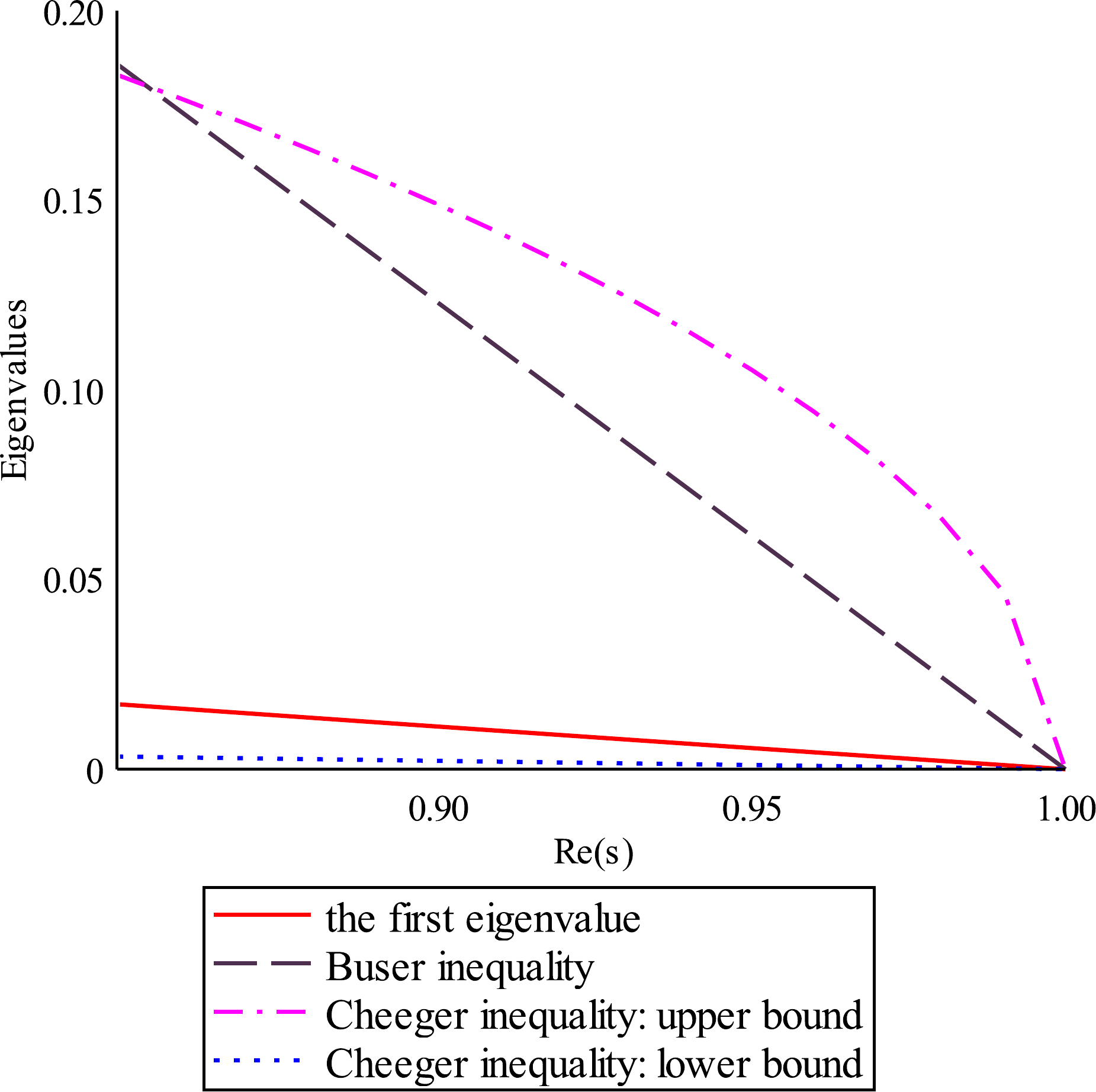}
\caption{Comparison of Cheeger and Buser estimates for a signed triangle\label{FtriangleComp}}
\end{figure}

\section{Lichnerowicz estimate and applications}\label{section:Lichnerowicz}

We have the following Lichnerowicz type eigenvalue estimate (cf. Theorem \ref{thm:introLic} in the Introduction).
\begin{thm}[Lichnerowicz inequality]\label{thm:lichnerowicz}
Assume that $(G, \mu, \sigma)$ satisfies $CD^{\sigma}(K, n)$ for $K\in \mathbb{R}$ and $n\in \mathbb{R}_+$. Then we have for any non-zero eigenvalue $\lambda^{\sigma}$ of $\Delta^{\sigma}$,
\begin{equation}\label{eq:lichnerowicz}
\frac{n-1}{n}\lambda^{\sigma}\geq K,
\end{equation}
where we use the convention $(n-1)/n=1$ in the case $n=\infty$.
\end{thm}
\begin{proof}
Let $\psi: V\to \mathbb{K}^d$ be the corresponding eigenfunction of $\lambda^{\sigma}$ with unit $\ell^2(V,\mathbb{K}^d;\mu)$ norm. Integrating the inequality $CD^{\sigma}(K, n)$ over the measure $\mu$, we obtain
\begin{equation}\label{eq:cdineq-integration}
\sum_{x\in V}\mu(x)\Gamma_2^\sigma(\psi)(x)\geq \frac{1}{n}\sum_{x\in V}\mu(x)\left| \Delta^\sigma \psi(x) \right|^2+K\sum_{x\in V}\mu(x)\Gamma^\sigma(\psi)(x).
\end{equation}
By (\ref{eq:laplacian-integration}), we have
\begin{equation*}
\sum_{x\in V}\mu(x)\Gamma_2^\sigma(\psi)(x)=-\sum_{x\in V}\mu(x)\mathrm{Re}(\Gamma^\sigma(\psi, \Delta^{\sigma}\psi)(x))=\lambda^{\sigma}\sum_{x\in V}\mu(x)\Gamma^\sigma(\psi)(x).
\end{equation*}
Recalling the summation by part formula (\ref{eq:summaBypart}), we have
\begin{equation*}
\sum_{x\in V}\mu(x)\Gamma^\sigma(\psi)(x)=-\langle \psi, \Delta^{\sigma}\psi\rangle_{\mu}=\lambda^{\sigma}.
\end{equation*}
Therefore, (\ref{eq:cdineq-integration}) tells us that
\begin{equation}
(\lambda^{\sigma})^2\geq \frac{1}{n}(\lambda^{\sigma})^2+\lambda^{\sigma}K.
\end{equation}
This implies (\ref{eq:lichnerowicz}) in the case $\lambda^{\sigma}\neq 0$.
\end{proof}

Consequently, we have the following estimates about the lower curvature bound of a graph.
\begin{coro}\label{cor:CurUpperCheeger}
Let $(G,\mu, \sigma)$ satisfy $CD^\sigma(K,n)$, for some $K\in \mathbb{R}$ and $n\in \mathbb{R}_+$. Then we have the following facts:
\begin{enumerate}[(i)]
 \item If $n=1$, we have $$K\leq 0;$$\label{eq:n=1}
  \item If $0< n<1$, we have $K<0$. If, furthermore, $\sigma$ is not balanced, we have $$K\leq -\frac{2(1-n)}{5nD_G^{non}}(\widetilde{h_1^\sigma})^2;$$\label{eq:n<1}
  \item If $1<n\leq \infty$ and $\sigma$ is not balanced, we have $$K\leq \frac{2(n-1)}{n}\widetilde{h_1^\sigma}\leq\frac{2(n-1)}{n}h_1^\sigma.$$\label{eq:n>1}
\end{enumerate}
\end{coro}
\begin{proof}
The estimate (\ref{eq:n=1}) follows directly from Theorem \ref{thm:lichnerowicz}.
Note that $\lambda_1^\sigma$ is positive when $\sigma$ is not balanced. Hence we can
combine Theorem \ref{thm:lichnerowicz} and Theorem \ref{thm:Cheeger} to conclude estimates (\ref{eq:n<1}) and (\ref{eq:n>1}).
\end{proof}

If the graph has a nonnegative lower curvature bound, we can improve the estimate Corollary \ref{cor:CurUpperCheeger} (\ref{eq:n>1}) by applying Corollary \ref{thm:Buser-K}.
\begin{coro}\label{cor:jump}
 Let $(G,\mu, \sigma)$ satisfy $CD^\sigma(K,n)$, for some $K\geq 0$ and $1<n\leq \infty$. If the signature $\sigma$ is not balanced, then
  we have
  \begin{equation}\label{eq:curvature-Cheeger}
      K\leq \frac{n-1}{n}\min\left\{2\widetilde{h_1^\sigma},\frac{8e^2}{(e-1)^2}D_G^{nor}(\widetilde{h_1^\sigma})^2\right\}.
  \end{equation}
\end{coro}
\begin{proof}
Recall that $CD^\sigma(K,n)$ implies $CD^\sigma(K,\infty)$. Hence, Corollary \ref{thm:Buser-K} is applicable here.
\end{proof}
\begin{remark}[Jump of the curvature around a balanced signature]\label{remark:jump}
Suppose that a graph $(G, \mu)$ with a balanced signature has positive $n$-dimensional Ricci curvature, i.e., $K_{n}(\sigma_{\mathrm{triv}})>0$. (Recall that every balanced signature is switching equivalent to $\sigma_{\mathrm{triv}}$. Note that by Corollary \ref{cor:CurUpperCheeger},  $K_{n}(\sigma_{\mathrm{triv}})>0$ is only possible when $1<n\leq\infty$.) Then by Corollary \ref{cor:CurUpperCheeger}, we observe that the curvature $K_{n}(\sigma)$ of
$(G,\mu, \sigma)$, as a function of the signature $\sigma$, has the following "jump" phenomenon: For unbalanced signatures $\sigma$, when they are close to the balanced signature $\sigma_{\mathrm{triv}}$,
\begin{equation}
 \limsup_{\iota^\sigma(V)\to 0}K_n(\sigma)\leq 0,\,\,\,\text{but}\,\,K_n(\sigma_{\mathrm{triv}})>0.
\end{equation}
In the above expression, we use $\iota^\sigma(V)$ as a measure for the difference between $\sigma$ and $\sigma_{\mathrm{triv}}$.

\end{remark}
The jump of the curvature is closely relate to the jump phenomenon of the first non-zero eigenvalue of $\Delta^\sigma$. When the signature $\sigma$ of a connected graph becomes balanced, the first non-zero eigenvalue jumps from $\lambda_1^\sigma$ to $\lambda_2^\sigma$.
\begin{example}[Signed Triangle]
We consider the example of a signed triangle again. Recall that we have observed the jump phenomenon of the curvature of a signed triangle in Remark \ref{remark:jumptriangle} (see Figure \ref{F2}). In Figure \ref{FtriangleEigenJump} of the Introduction, the jumps of the $\infty$-dimensional Ricci curvature and the first non-zero eigenvalue of a signed triangle are illustrated in the same diagram.
\end{example}

We conclude this section by an interesting application of the jump phenomenon of the curvature.
\begin{thm}\label{thm:ApplOFJump}
Suppose that a graph $G$ has at least one cycle, but no cycles of length $3$ or $4$. Then, for any signature $\sigma$, any edge weights and any vertex measure $\mu$, we have
$$K_n(G,\mu,\sigma)\leq 0,\,\,\text{ for any }\,\,n\in \mathbb{R}_+.$$
\end{thm}
\begin{proof}
 Since $G$ contains at least one cycle, there exist unbalanced signatures on $G$. On the other hand, we have \begin{equation}\label{eq:contrToJump}K_n(G,\mu,\sigma)=K_n(G,\mu,\sigma_{\mathrm{triv}})\end{equation} by Proposition \ref{pro:shortcycles}, as $G$ has no cycles of length $3$ or $4$. Therefore, if $K_n(G,\mu,\sigma_{\mathrm{triv}})>0$, the equality (\ref{eq:contrToJump}) leads to a contradiction to the jump of the curvature observed in Remark \ref{remark:jump}. Hence we must have $K_n(G,\mu,\sigma)\leq 0$.
\end{proof}
 Note that the conditions on the graph in Theorem \ref{thm:ApplOFJump} are purely combinatorial, whereas the curvature estimate holds for any edge weights and vertex measures.

Combining Theorem \ref{thm:ApplOFJump} and Corollary \ref{cor:lower curvature bound}, we obtain an indirect verification of (\ref{eq:5cycle Curvature}). Actually, we obtain the following more general result.
\begin{coro}\label{cor:5cycle}
Let $N\geq 5$ and $(\mathcal{C}_N,\mu,\sigma)$ be an unweighted cycle with constant vertex measure $\mu=\nu_0\cdot \mathbf{1}_V$. Then we have
$$K_n(\mathcal{C}_N,\mu,\sigma)=0 \,\,\text{for any}\,\,n\geq 2.$$
\end{coro}

\section{Eigenvalue ratios of graphs with $O(1)$ signatures}\label{section:O(1)}

In this section, we restrict our considerations to the setting of a graph $(G,\mu)$ with a signature
\begin{equation*}
 \sigma: E^{or}\to O(1)=\{\pm 1\}.
\end{equation*}
We show that Theorem \ref{thm:Buser} can be applied to derive an upper bound for the ratio of the $k$-th eigenvalue $\lambda_k^\sigma$ to the first eigenvalue $\lambda_1^\sigma$ when $(G,\mu,\sigma)$ satisfies $CD^\sigma(0,\infty)$.

Note that the connection Laplacian reduces to an operator on $\ell^2(V,\mathbb{R};\mu)$. That is, for any real function $f:V\rightarrow \mathbb{R}$ and any vertex $x\in V$, we have
\begin{equation}\label{eq:signed Laplacian}
\Delta^\sigma f(x):=\frac{1}{\mu(x)}\sum_{y,y\sim x}w_{xy}(\sigma_{xy}f(y)-f(x))\in \mathbb{R}.
\end{equation}
The eigenvalues of $\Delta^\sigma$ can be listed as
\begin{equation*}
 0\leq \lambda_1^\sigma\leq\cdots\leq \lambda_k^\sigma\cdots\leq \lambda_N^\sigma\leq 2D_G^{non}.
\end{equation*}

In \cite[Theorem 3]{AtayLiu14}, Atay and Liu prove the following estimate.
\begin{thm}[\cite{AtayLiu14}]\label{thm:AtayLiu}
For any graph $(G,\mu,\sigma)$ with $\sigma:E^{or}\to O(1)$ and any natural number $1\leq k\leq N$, we have
\begin{equation}\label{eq:AtayLiu}
 h_1^\sigma\leq 16\sqrt{2D_G^{non}}k\frac{\lambda_1^\sigma}{\sqrt{\lambda_k^\sigma}}.
\end{equation}
\end{thm}
This result is an extension of the so-called \emph{improved Cheeger inequality} due to Kwok~et~al.~\cite{KLLGT2013} for the graph Laplacian $\Delta$. We also mention that in the current case of $O(1)$ signatures, the multi-way Cheeger constants, given in Definition \ref{defn:Cheeger}, have more explicit combinatorial expressions. We refer to \cite{AtayLiu14} for more details.

As an application of Theorem \ref{thm:Buser}, we prove the following eigenvalue ratio estimates.
\begin{thm}\label{thm:EigenRatio}
For any graph $(G,\mu,\sigma)$ with $\sigma:E^{or}\to O(1)$ satisfying $CD^\sigma(0,\infty)$ and any natural number $1\leq k\leq N$, there exists an absolute constant $C$ such that
\begin{equation}\label{eq:eigenRatio}
 \lambda_k^\sigma\leq CD_G^{nor}D_G^{non}k^2\lambda_1^\sigma.
\end{equation}
\end{thm}
\begin{proof}
Since the inequality $CD^\sigma(0,\infty)$ is satisfied, we have by Theorem \ref{thm:Buser},
\begin{equation*}
 \sqrt{\lambda_1^\sigma}\leq 4\sqrt{(\log 2)D_G^{nor}}h_1^\sigma.
\end{equation*}
Combining this with Theorem \ref{thm:AtayLiu}, we obtain
\begin{equation}\label{eq:o1Buser}
 \sqrt{\lambda_1^\sigma}\leq 64\sqrt{2\log 2}\sqrt{D_G^{nor}D_G^{non}}k\frac{\lambda_1^\sigma}{\sqrt{\lambda_k^\sigma}}.
\end{equation}
This implies (\ref{eq:eigenRatio}) immediately.
\end{proof}
A direct corollary is the following Buser type inequality.
\begin{coro}\label{cor:BuserO1}
 Let $(G,\mu,\sigma)$ with $\sigma:E^{or}\to O(1)$ satisfy $CD^\sigma(0,\infty)$. Then for all $1\leq k\leq N$, there exists an absolute constant $C$ such that
\begin{equation}\label{eq:O1coro}
 \sqrt{\lambda_k^\sigma}\leq C\sqrt{D_G^{non}}D_G^{nor}kh_1^\sigma.
\end{equation}
\end{coro}
\begin{proof}
 Combining (\ref{eq:o1Buser}) with Theorem \ref{thm:EigenRatio} leads to this result immediately.
\end{proof}
\begin{remark}
Comparing this result with Theorem \ref{thm:Buser}, the advantage of the estimate (\ref{eq:O1coro}) lies in the fact that $h_1^\sigma\leq h_k^\sigma$ and that the order of $k$ in (\ref{eq:O1coro}) is lower. However, in the estimate (\ref{eq:O1coro}), the orders of the degrees $D_G^{nor}$ and $D_G^{non}$ are higher than in Theorem \ref{thm:Buser}.
\end{remark}
Finally, we observe that Theorem \ref{thm:AtayLiu} and hence the estimate in Corollary \ref{cor:BuserO1} can not be true for general signatures $\sigma: E^{or}\to H$, even in the $1$-dimensional case $H=U(1)$. To explain the reason, let us revisit the example of a signed triangle.

\begin{example}[Signed Triangle]\label{example:section 7} The example of a signed triangle, discussed in Section \ref{subsection:triangle}, carries a $U(1)$ signature $\sigma$ is assigned (see Figure \ref{F1}). If $\mathrm{Re}(s)$ tends to $1$, i.e., if the signature on the triangle tends to be balanced, we observe that $\lambda_2^\sigma$ has a positive lower bound (see Figure \ref{FtriangleEigenJump}), while both $\lambda_1^\sigma$ and $h_1^\sigma$ tend to zero, but at a different rate (see Figure \ref{FtriangleComp}). In fact, by Theorem \ref{thm:Buser}, we have
\begin{equation}\label{eq:triangleBuser2}
 \lambda_1^\sigma\leq 32\log 2(h_1^\sigma)^2.
\end{equation}
 Assume that Theorem \ref{thm:AtayLiu} holds in this case for $k=2$. Combining this with (\ref{eq:triangleBuser2}), we obtain $1\leq Ch_1^\sigma$ for some absolute constant $C>0$. This is a contradiction. Hence, Theorem \ref{thm:AtayLiu} cannot hold for more general signatures.
\end{example}

\appendix

\section{\\Curvature and Cheeger constants on Cartesian products}\label{section:appendixCurCheeCar}
In this section, we discuss the $CD^\sigma$ inequality and the Cheeger constants on the Cartesian product of two graphs.
For two graphs, $G_1=(V_1, E_1)$ and $G_2=(V_2, E_2)$, their Cartesian product $G_1\times G_2=(V_1\times V_2, E_{12})$ is defined as follows. Any two vertices $(x_1,y_1), (x_2,y_2)\in V_1\times V_2$ are connected by an edge in $E_{12}$ if and only if
\begin{equation*}
\text{either}\,\,\{x_1,x_2\}\in E_1, y_1=y_2 \,\,\,\text{ or }\,\,\,x_1=x_2, \{y_1, y_2\}\in E_2.
\end{equation*}

\subsection{Curvature on Cartesian Products}
We first discuss the simpler case of graphs with constant vertex measures.

\subsubsection{Graphs with constant vertex measures}
Given two graphs $G_1=(V_1, E_1,w_1)$ and $G_2=(V_2,E_2,w_2)$, we assign the following edge weights to the Cartesian product $G_1\times G_2=(V_1\times V_2, E_{12})$:
\begin{align}
w_{12,(x_1,y)(x_2,y)}:=w_{1,x_1x_2}, \,\,\text{ for any }\,\,\{x_1,x_2\}\in E_1, y\in V_2;\notag\\
w_{12,(x,y_1)(x,y_2)}:=w_{2,y_1y_2}, \,\,\text{ for any }\,\,\{y_1,y_2\}\in E_2, x\in V_1.\label{eq:weightCartesianI}
\end{align}

Let $\sigma_i: E_i^{or}\to H_i, i=1,2$ be the signatures of $G_i, i=1,2$, respectively. We need to find a proper construction of the signature on the Cartesian product graph $G_1\times G_2$.

We first consider the case that
$$H_1=H_2:=H=O(d)\,\,\text{ or }\,\, U(d), \,\,\text{ for some }\,\, d\in \mathbb{Z}_{>0}.$$
In this case, we define the signature $\sigma_{12}: E_{12}^{or}\to H$ as follows
\begin{align}
\sigma_{12, (x_1,y)(x_2,y)}&:=\sigma_{1,x_1x_2}, \,\,\text{ for any }\,\,(x_1,x_2)\in E_1^{or}, y\in V_2;\notag\\
\sigma_{12, (x,y_1)(x,y_2)}&:=\sigma_{2,y_1y_2}, \,\,\text{ for any }\,\,(y_1,y_2)\in E_2^{or}, x\in V_1.\label{eq:signCartesianI}
\end{align}
Let $\Sigma_i$ be the \emph{signature group} of the graph $G_i$ with $\sigma_i$ (recall the definition of the signature group at the end of Section \ref{subsection:Connection Laplacian}). We say the two subgroups $\Sigma_1$ and $\Sigma_2$ of $H$ \emph{commute}, if for any $s_1\in \Sigma_1$, $s_2\in \Sigma_2$, we have $s_1s_2=s_2s_1$.

\begin{thm}\label{thm:CartesianI}
Let $(G_1, 1_{\mathbf{V_1}}, \sigma_1)$ and $(G_2, 1_{\mathbf{V_2}}, \sigma_2)$ be two graphs with $\sigma_i: E_i^{or}\to~H,$ $i=1,2$. Assume that they satisfy
$$CD^{\sigma_1}(K_1,n_1)\,\,\text{ and }\,\,CD^{\sigma_2}(K_2, n_2),$$ respectively. If their signature groups $\Sigma_1$ and $\Sigma_2$ commute, then their Cartesian product graph $(G_1\times G_2, 1_{\mathbf{V_1\times V_2}}, \sigma_{12})$, with the edge weight $w_{12}$ given in (\ref{eq:weightCartesianI}), satisfies
$$CD^{\sigma_{12}}(K_1\wedge K_2, n_1+n_2),$$ where $K_1\wedge K_2:=\min\{K_1, K_2\}$.
\end{thm}
Note that the commutativity restriction of $\Sigma_1$ and $\Sigma_2$ is a very natural condition. The intuition behind the above result is that the "mixed structure" in the Cartesian product is "flat". To be precise, we want for two balanced signatures $\sigma_1, \sigma_2$ on $G_1, G_2$ that $\sigma_{12}$ on $G_1\times G_2$ is also balanced. In Figure \ref{F3}, we show a typical new cycle created in the process of taking the Cartesian product, where $\{x_i,x\}\in E_1$ and $\{y_k,y\}\in E_2$.
Since $\Sigma_1$ and $\Sigma_2$ commute, the signature of this cycle, given by
$\sigma_{1,xx_i}\sigma_{2,yy_k}\sigma_{1,xx_i}^{-1}\sigma_{2,yy_k}^{-1}$, is trivial.
\begin{figure}[h]
\centering
\includegraphics[width=0.35\textwidth]{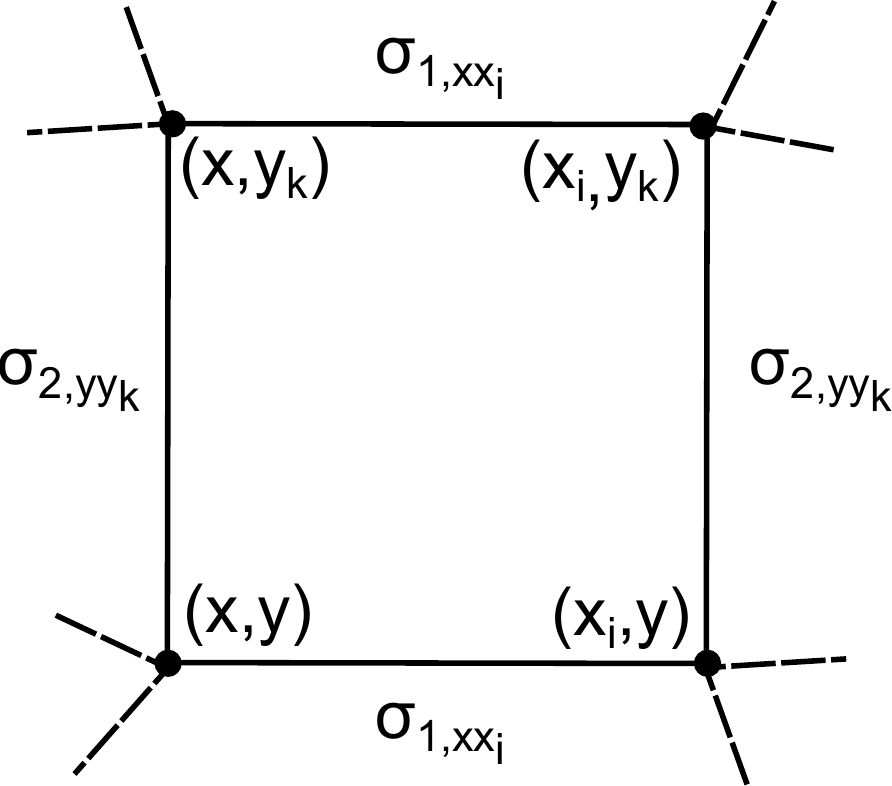}
\caption{A typical cycle in the Cartesian product graph\label{F3}}
\end{figure}
\begin{proof}[Proof of Theorem \ref{thm:CartesianI}] Let $f: V_1\times V_2\to \mathbb{K}^d$ be any $\mathbb{K}^d$ valued function on the product graph. For fixed $y\in V_2$, we define $f_y(\cdot):=f(\cdot,y)$ to be a $\mathbb{K}^d$ valued function on $G_1$. Similarly, we define the function $f^x(\cdot):=f(x,\cdot): V_2\to \mathbb{K}^d$.

By the construction of $\sigma_{12}$, it is straightforward to check that
\begin{align}
\Delta^{\sigma_{12}}f(x,y)&=\Delta^{\sigma_1}f_y(x)+\Delta^{\sigma_2}f^x(y),\label{eq:splitLaplacian}\\
\Gamma^{\sigma_{12}}(f)(x,y)&=\Gamma^{\sigma_1}(f_y)(x)+\Gamma^{\sigma_2}(f^x)(y).\label{eq:splitGamma}
\end{align}
For the operator $\Gamma_2^{\sigma_{12}}$, we claim that
\begin{equation}\label{eq:splitGammat2}
\Gamma_2^{\sigma_{12}}(f)(x,y)\geq \Gamma_2^{\sigma_1}(f_y)(x)+\Gamma_2^{\sigma_2}(f^x)(y).
\end{equation}
Once inequality (\ref{eq:splitGammat2}) is verified, we apply the $CD^{\sigma}$ inequalities on $G_1$ and $G_2$ to estimate
\begin{align}
&\Gamma_2^{\sigma_{12}}(f)(x,y)\notag\\
\geq &\frac{1}{n_1}|\Delta^{\sigma_1}f_y(x)|^2+K_1\Gamma^{\sigma_1}(f_y)(x)+\frac{1}{n_2}|\Delta^{\sigma_2}f^x(y)|^2+K_2\Gamma^{\sigma_2}(f^x)(y)\notag\\
\geq &\frac{1}{n_1+n_2}\left(|\Delta^{\sigma_1}f_y(x)|+|\Delta^{\sigma_2}f^x(y)|\right)^2+( K_1\wedge K_2)\left(\Gamma^{\sigma_1}(f_y)(x)+\Gamma^{\sigma_2}(f^x)(y)\right)\notag\\
\geq&\frac{1}{n_1+n_2}|\Delta^{\sigma_{12}}f(x,y)|^2+(K_1\wedge K_2)\Gamma^{\sigma_{12}}(f)(x,y).\label{eq:tightI}
\end{align}
In the third inequality above, we used the triangle inequality and the equalities (\ref{eq:splitLaplacian}) and (\ref{eq:splitGamma}).

Hence, it only remains to prove the claim (\ref{eq:splitGammat2}). Recall that
\begin{equation}\label{eq:gamma2Cartesian}
2\Gamma_2^{\sigma_{12}}(f)(x,y)=\Delta\Gamma^{\sigma_{12}}(f)(x,y)-\Gamma^{\sigma}(f, \Delta^{\sigma_{12}}f)(x,y)-\Gamma^{\sigma}(\Delta^{\sigma_{12}}f,f)(x,y).
\end{equation}
For simplicity, we denote the neighbors of $x$ in $V_1$ by $x_i$ and the neighbors of $y$ in $V_2$ by $y_k$. We will then write, for short,
\begin{equation*}
w_{1,i}:=w_{1,xx_i}, w_{2,k}:=w_{2,yy_k}, \,\,\text{and}\,\,\sigma_{1,i}:=\sigma_{1,xx_i}, \sigma_{2,k}:=\sigma_{2,yy_k},
\end{equation*}
and $\sum_{x_i}$ ($\sum_{y_k}$, resp.) the summation over all neighbors of $x\in V_1$ ($y\in V_2$, resp.).
We first calculate
\begin{align*}
&\Delta\Gamma^{\sigma_{12}}(f)(x,y)\notag=\\
&\underbrace{\sum_{x_i}w_{1,i}\left(\Gamma^{\sigma_{12}}(f)(x_i,y)-\Gamma^{\sigma_{12}}(f)(x,y)\right)}_{=:L_1}
+\underbrace{\sum_{y_k}w_{2,k}\left(\Gamma^{\sigma_{12}}(f)(x,y_k)-\Gamma^{\sigma_{12}}(f)(x,y)\right)}_{=:L_2}.
\end{align*}
Applying (\ref{eq:splitGamma}), we obtain
\begin{align}
L_1=&\sum_{x_i}w_{1,i}\left(\Gamma^{\sigma_{1}}(f_y)(x_i)+\Gamma^{\sigma_2}(f^{x_i})(y)-\Gamma^{\sigma_{1}}(f_y)(x)-\Gamma^{\sigma_2}(f^x)(y)\right)\notag\\
=&\Delta\Gamma^{\sigma_1}(f_y)(x)+\frac{1}{2}\sum_{x_i, y_k}w_{1,i}w_{2,k}(|\sigma_{2,k}f(x_i,y_k)-f(x_i,y)|^2\notag\\
&\hspace{6cm}-|\sigma_{2,k}f(x,y_k)-f(x,y)|^2).\label{eq:L1}
\end{align}
Similarly, we have
\begin{align}
L_2
=&\Delta\Gamma^{\sigma_2}(f^x)(y)+\frac{1}{2}\sum_{x_i, y_k}w_{1,i}w_{2,k}(|\sigma_{1,i}f(x_i,y_k)-f(x,y_k)|^2\notag\\
&\hspace{6cm}-|\sigma_{i,i}f(x_i,y)-f(x,y)|^2).\label{eq:L2}
\end{align}
Now we calculate the remaining terms in (\ref{eq:gamma2Cartesian}).
\begin{align}
&\Gamma^{\sigma_{12}}(f, \Delta^{\sigma_{12}}f)(x,y)\notag\\
=&\frac{1}{2}\sum_{x_i}w_{1,i}(\sigma_{1,i}f(x_i,y)-f(x,y))^T(\overline{\sigma_{1,i}\Delta^{\sigma_{12}}f(x_i,y)-\Delta^{\sigma_{12}}f(x,y)})\notag\\
&+\frac{1}{2}\sum_{y_k}w_{2,k}(\sigma_{2,k}f(x,y_k)-f(x,y))^T(\overline{\sigma_{2,k}\Delta^{\sigma_{12}}f(x,y_k)-\Delta^{\sigma_{12}}f(x,y)})\notag\\
=:&M_1+M_2.
\end{align}
Applying the equality (\ref{eq:splitLaplacian}), we obtain
\begin{align}
M_1=&\Gamma^{\sigma_1}(f_y, \Delta^{\sigma_1}f_y)(x)\notag\\
&\hspace{.5cm}+\frac{1}{2}\sum_{x_i}w_{1,i}(\sigma_{1,i}f(x_i,y)-f(x,y))^T(\overline{\sigma_{1,i}\Delta^{\sigma_2}f^{x_i}(y)-\Delta^{\sigma_2}f^{x}(y)}).
\end{align}
Hence, we get
\begin{align}
&M_1-\Gamma^{\sigma_1}(f_y, \Delta^{\sigma_1}f_y)(x)\notag\\
=&\frac{1}{2}\sum_{x_i, y_k}w_{1,i}w_{2,k}[(\sigma_{1,i}f(x_i,y)-f(x,y))^T(\overline{\sigma_{1,i}\sigma_{2,k}f(x_i,y_k)-\sigma_{2,k}f(x,y_k)})\notag\\
&\hspace{6.5cm}-|\sigma_{1,i}f(x_i,y)-f(x,y)|^2].\label{eq:M1}
\end{align}
Similarly, we have
\begin{align}
&M_2-\Gamma^{\sigma_2}(f^x, \Delta^{\sigma_2}f^x)(y)\notag\\
=&\frac{1}{2}\sum_{x_i, y_k}w_{1,i}w_{2,k}[(\sigma_{2,k}f(x,y_k)-f(x,y))^T(\overline{\sigma_{2,k}\sigma_{1,i}f(x_i,y_k)-\sigma_{1,i}f(x_i,y)})\notag\\
&\hspace{6.5cm}-|\sigma_{2,k}f(x,y_k)-f(x,y)|^2].\label{eq:M2}
\end{align}
Combining (\ref{eq:L1}) and (\ref{eq:M2}), we arrive at
\begin{align}
&(L_1-\Delta\Gamma^{\sigma_1}(f_y)(x))-(M_2-\Gamma^{\sigma_2}(f^x, \Delta^{\sigma_2}f^x)(y))-(\overline{M_2-\Gamma^{\sigma_2}(f^x, \Delta^{\sigma_2}f^x)(y)})\notag\\
=&\frac{1}{2}\sum_{x_i, y_k}w_{1,i}w_{2,k}[|\sigma_{2,k}f(x_i,y_k)-f(x_i,y)|^2+|\sigma_{2,k}f(x,y_k)-f(x,y)|^2\notag\\
&\hspace{2cm}+(\sigma_{2,k}f(x,y_k)-f(x,y))^T(\overline{\sigma_{2,k}\sigma_{1,i}f(x_i,y_k)-\sigma_{1,i}f(x_i,y)})\notag\\
&\hspace{2cm}+(\overline{\sigma_{2,k}f(x,y_k)-f(x,y)})^T(\sigma_{2,k}\sigma_{1,i}f(x_i,y_k)-\sigma_{1,i}f(x_i,y))].
\end{align}
Since $\Sigma_1$ and $\Sigma_2$ commute, we have
\begin{equation*}
\sigma_{2,k}\sigma_{1,i}=\sigma_{1,i}\sigma_{2,k}.
\end{equation*}
Therefore, we obtain
\begin{align}
&(L_1-\Delta\Gamma^{\sigma_1}(f_y)(x))-(M_2-\Gamma^{\sigma_2}(f^x, \Delta^{\sigma_2}f^x)(y))-(\overline{M_2-\Gamma^{\sigma_2}(f^x, \Delta^{\sigma_2}f^x)(y)})\notag\\
=&\frac{1}{2}\sum_{x_i, y_k}w_{1,i}w_{2,k}|\sigma_{1,i}\sigma_{2,k}f(x_i,y_k)-\sigma_{1,i}f(x_i,y)-\sigma_{2,k}f(x,y_k)+f(x,y)|^2.\label{eq:L1M2}
\end{align}
Similarly, by combining (\ref{eq:L2}) and (\ref{eq:M1}), we obtain
\begin{align}
&(L_2-\Delta\Gamma^{\sigma_2}(f^x)(y))-(M_1-\Gamma^{\sigma_1}(f_y, \Delta^{\sigma_1}f_y)(x))-(\overline{M_1-\Gamma^{\sigma_1}(f_y, \Delta^{\sigma_1}f_y)(x)})\notag\\
=&\frac{1}{2}\sum_{x_i, y_k}w_{1,i}w_{2,k}|\sigma_{2,k}\sigma_{1,i}f(x_i,y_k)-\sigma_{2,k}f(x,y_k)-\sigma_{1,i}f(x_i,y)+f(x,y)|^2.\label{eq:L2M1}
\end{align}
Adding (\ref{eq:L1M2}) and $(\ref{eq:L2M1})$, and using (\ref{eq:gamma symmetric}), we get
\begin{align}
&2\Gamma_2^{\sigma_{12}}(f)(x,y)-2\Gamma^{\sigma_1}_2(f_y)(x)-2\Gamma_2^{\sigma_2}(f^x)(y)\notag\\
=&\sum_{x_i, y_k}w_{1,i}w_{2,k}|\sigma_{2,k}\sigma_{1,i}f(x_i,y_k)-\sigma_{2,k}f(x,y_k)-\sigma_{1,i}f(x_i,y)+f(x,y)|^2\geq 0.\label{eq:gamma2estimate}
\end{align}
This proves (\ref{eq:splitGammat2}).
\end{proof}

\begin{remark}[Tightness of Theorem \ref{thm:CartesianI}]
The estimate in Theorem \ref{thm:CartesianI} is tight at least in the case of taking the Cartesian product of $(G, \mathbf{1}_{V}, \sigma)$ with itself, assuming that its signature group $\Sigma$ is abelian. That is, for any given $n\in \mathbb{R}_+$, the precise lower curvature bounds satisfy
$$K_{2n}(G\times G, \mathbf{1}_{V\times V}, \sigma_{12})=K_n(G, \mathbf{1}_{V}, \sigma).$$
Note that the tightness of Theorem \ref{thm:CartesianI} lies in the tightness of (\ref{eq:tightI}) and (\ref{eq:gamma2estimate}). By assumption, there exists a function $f: V\to \mathbb{K}^d$ and a vertex $x\in V$ such that
$$\Gamma_2^\sigma(f)(x)=\frac{1}{n}|\Delta^\sigma f(x)|^2+K\Gamma^\sigma(f)(x)\,\,\text{ and }\,\,\Gamma^\sigma{f}(x)\neq 0.$$
Then we can choose a function $F:V\times V\to \mathbb{K}^d$ satisfying, locally, around the vertex $(x,x)\in V\times V$,
\begin{enumerate}[(i)]
  \item $F(x,x):=f(x)$;
  \item $F(x_i,x):=f(x_i)$, for all $x_i\sim x$;
  \item $F(x,x_k):=f(x_k)$, for all $x_k\sim x$;
  \item $F(x_i,x_k):=\sigma_{xx_i}^{-1}\sigma_{xx_k}^{-1}(\sigma_{xx_k}f(x_k)+\sigma_{xx_i}^{-1}f(x_i)-f(x))$, for all $x_i\sim x$ and $x_k\sim x$.
\end{enumerate}
Note that (i-iii) implies $\Delta^\sigma F_x(x)=\Delta^\sigma F^x(x)$ and, hence, (\ref{eq:tightI}) holds with equalities. Property (iv) ensures that (\ref{eq:gamma2estimate}) holds also with equality. This shows the tightness of the result.
\end{remark}

Next, we discuss the situation when the two groups $H_1$ and $H_2$ are different. We assume that
$$H_1=O(d_1), H_2=O(d_2),\,\,\text{ or }\,\,H_1=U(d_1), H_2=U(d_2), \,\,\text{for some}\,\, d_1,d_2\in \mathbb{Z}_{>0},$$
where $d_1, d_2$ can be different integers. In such a general situation, we construct the signature $\widehat{\sigma}_{12}: E_{12}^{or}\to H_1\otimes H_2$ on the Cartesian product graph in the following way:
\begin{align}
\widehat{\sigma}_{12, (x_1,y)(x_2,y)}&:=\sigma_{1,x_1x_2}\otimes \mathrm{I}_{d_2}, \,\,\text{ for any }\,\,(x_1,x_2)\in E_1^{or}, y\in V_2;\notag\\
\widehat{\sigma}_{12, (x,y_1)(x,y_2)}&:=\mathrm{I}_{d_1}\otimes\sigma_{2,y_1y_2}, \,\,\text{ for any }\,\,(y_1,y_2)\in E_2^{or}, x\in V_1,\label{eq:signCartesianII}
\end{align}
where $\mathrm{I}_{d_i}$ is the identity matrix of size $d_i\times d_i$, $i=1,2$.
\begin{thm}\label{thm:CartesianII}
Let $(G_1, 1_{\mathbf{V_1}}, \sigma_1)$ and $(G_2, 1_{\mathbf{V_2}}, \sigma_2)$ be two graphs with $\sigma_i: E_i^{or}\to~H_i,$ $i=1,2$. Assume that they satisfy
$$CD^{\sigma_1}(K_1,n_1)\,\,\text{ and }\,\,CD^{\sigma_2}(K_2, n_2),$$
respectively. Then their Cartesian product graph $(G_1\times G_2, 1_{\mathbf{V_1\times V_2}}, \widehat{\sigma}_{12})$ with the edge weight $w_{12}$ given in (\ref{eq:weightCartesianI}) satisfies
$$CD^{\widehat{\sigma}_{12}}(K_1\wedge K_2, n_1+n_2).$$
\end{thm}
This is an immediate consequence of Theorem \ref{thm:CartesianI} and Corollary \ref{cor:tensor}.
\begin{proof}[Proof of Theorem \ref{thm:CartesianII}]
By Corollary \ref{cor:tensor}, we know that $(G_1, 1_{\mathbf{V_1}}, \sigma_1\otimes \mathrm{I}_{d_2})$ and $(G_2, 1_{\mathbf{V_2}}, \mathrm{I}_{d_1}\otimes\sigma_2)$ satisfy $CD^{\sigma_1\otimes \mathrm{I}_{d_2}}(K_1,n_1)$ and $CD^{\mathrm{I}_{d_1}\otimes\sigma_2}(K_2, n_2)$, respectively. Note that for any $(x, x_i)\in E_1^{or}$ and $(y, y_k)\in E_2^{or}$, we have
\begin{equation}
(\sigma_{1,xx_i}\otimes \mathrm{I}_{d_2})(\mathrm{I}_{d_1}\otimes\sigma_{2,yy_k})=(\mathrm{I}_{d_1}\otimes\sigma_{2,yy_k})(\sigma_{1,xx_i}\otimes \mathrm{I}_{d_2}).
\end{equation}
That is, the corresponding signature groups of $(G_1, 1_{\mathbf{V_1}}, \sigma_1\otimes \mathrm{I}_{d_2})$ and $(G_2, 1_{\mathbf{V_2}}, \mathrm{I}_{d_1}\otimes\sigma_2)$ commute. Hence, we can apply Theorem \ref{thm:CartesianI} and finish the proof.
\end{proof}

\begin{remark}[Vertex measure]\label{remark:vertex measure}
In Theorems \ref{thm:CartesianI} and \ref{thm:CartesianII}, we use the particular vertex measure $\mu(x)=1$ for all vertices $x$. In fact, we have more flexibility about those measures. Assume that the vertex measures of $G_1, G_2, G_1\times G_2$ take constant values $\nu_1, \nu_2, \nu_{12}\in \mathbb{R}$, respectively. Then under the assumptions of Theorem \ref{thm:CartesianI} (replacing $\mathbf{1}_{V_i}$ by $\nu_i\cdot \mathbf{1}_{V_i}$), we have that the graph $(G_1\times G_2, \nu_{12}\cdot \mathbf{1}_{V_1\times V_2}, \sigma_{12})$ satisfies
$$CD^{\sigma_{12}}\left(\frac{1}{\nu_{12}}(\nu_1K_1\wedge\nu_2K_2), n_1+n_2\right).$$
The result in Theorem \ref{thm:CartesianII} can be generalized similarly.
\end{remark}

\subsubsection{Graphs with nonconstant vertex measures}
For two graphs $(G_1,\mu_1)$ and $(G_2, \mu_2)$ whose vertex measures are not necessarily constant, we modify the definition of the edge weights of their Cartesian product. In \cite{ChungTetali98}, Chung and Tetali introduced the following
edge weight for the Cartesian product $G_1\times G_2=(V_1\times V_2, E_{12})$,
\begin{align}
w_{12,(x_1,y)(x_2,y)}^{\Box}:=w_{1,x_1x_2}\mu_2(y), \,\,\text{ for any }\,\,\{x_1,x_2\}\in E_1, y\in V_2;\notag\\
w_{12,(x,y_1)(x,y_2)}^{\Box}:=w_{2,y_1y_2}\mu_1(x), \,\,\text{ for any }\,\,\{y_1,y_2\}\in E_2, x\in V_1.\label{eq:weightCartesianII}
\end{align}
and the specific vertex measure
\begin{equation}\label{eq:ChungTetali2}
2\mu_1\mu_2:  V_1\times V_1\ni(x,y)\mapsto 2\mu_1(x)\mu_2(y)\in \mathbb{R}.
\end{equation}
Observe that, in the case $\mu_i=\mathbf{d}_{V_i},\,i=1,2$, we have
\begin{equation*}
2\mu_{1}(x)\mu_2(y)=\sum_{x_i}w_{12,(x_i,y)(x,y)}^{\Box}+\sum_{y_k}w_{12,(x,y_k)(x,y)}^{\Box}.
\end{equation*}
The definitions (\ref{eq:weightCartesianII}) and (\ref{eq:ChungTetali2}) lead to a Laplacian associated  to a natural random walk on the product graph $G_1\times G_2$.

\begin{thm}\label{thm:CartesianIII}
Let $(G_1, \mu_1, \sigma_1)$ and $(G_2, \mu_2, \sigma_2)$ be two graphs with $\sigma_i:~E_i^{or}~\to~H,$ $i=1,2$. Assume that they satisfy
$$CD^{\sigma_1}(K_1,n_1)\,\,\text{ and }\,\,CD^{\sigma_2}(K_2, n_2),$$ respectively. If their signature groups $\Sigma_1$ and $\Sigma_2$ commute, then their Cartesian product graph $(G_1\times G_2, 2\mu_{1}\mu_2, \sigma_{12})$, with the edge weight $w_{12}^{\Box}$ given in (\ref{eq:weightCartesianII}), satisfies
$$CD^{\sigma_{12}}\left(\frac{1}{2}(K_1\wedge K_2), n_1+n_2\right),$$ where $K_1\wedge K_2:=\min\{K_1, K_2\}$.
\end{thm}
\begin{proof}
The proof is analogous to the proof of Theorem \ref{thm:CartesianI}. We only mention here that, this time, we have
\begin{align*}
\Delta^{\sigma_{12}}f(x,y)&=\frac{1}{2}\Delta^{\sigma_1}f_y(x)+\frac{1}{2}\Delta^{\sigma_2}f^x(y),\\
\Gamma^{\sigma_{12}}(f)(x,y)&=\frac{1}{2}\Gamma^{\sigma_1}(f_y)(x)+\frac{1}{2}\Gamma^{\sigma_2}(f^x)(y),
\end{align*}
and
\begin{equation*}
\Gamma_2^{\sigma_{12}}(f)(x,y)\geq \frac{1}{4}\Gamma_2^{\sigma_1}(f_y)(x)+\frac{1}{4}\Gamma_2^{\sigma_2}(f^x)(y).
\end{equation*}
The last inequality above is derived from
\begin{align}
&2\mu(x)\mu(y)\left(4\Gamma_2^{\sigma_{12}}(f)(x,y)-\Gamma^{\sigma_1}_2(f_y)(x)-\Gamma_2^{\sigma_2}(f^x)(y)\right)\notag\\
=&\sum_{x_i, y_k}w_{1,i}w_{2,k}|\sigma_{2,k}\sigma_{1,i}f(x_i,y_k)-\sigma_{2,k}f(x,y_k)-\sigma_{1,i}f(x_i,y)+f(x,y)|^2\geq 0.\notag
\end{align}
\end{proof}
\begin{remark}
Let us assign a general vertex measure $\mu_{12}$ to the Cartesian product graph. Then, under the assumption of Theorem \ref{thm:CartesianIII}, a proof analogous to the proof of \ref{thm:CartesianI} yields that $(G_1\times G_2,\mu_{12}, \sigma_{12})$, with the edge weight $w_{12}^{\Box}$ given in (\ref{eq:weightCartesianII}), satisfies
$$CD^{\sigma_{12}}\left(\min_{(x,y)\in V_1\times V_2}\frac{\mu_1(x)\mu_2(y)}{\mu_{12}(x,y)}(K_1\wedge K_2), n_1+n_2\right).$$
Note that this general result also includes Theorem \ref{thm:CartesianI} as a particular case.
\end{remark}

A result similar to Theorem \ref{thm:CartesianII} follows immediately from Theorem \ref{thm:CartesianIII}.

\begin{thm}\label{thm:CartesianIV}
Let $(G_1, \mu_1, \sigma_1)$ and $(G_2, \mu_2, \sigma_2)$ be two graphs with $\sigma_i: E_i^{or}\to H_i, i=1,2$. Assume that they satisfy
$$CD^{\sigma_1}(K_1,n_1)\,\,\text{ and }\,\,CD^{\sigma_2}(K_2, n_2),$$
respectively. Then their Cartesian product graph $(G_1\times G_2, \mu_{12}, \widehat{\sigma}_{12})$, with the edge weight $w^{\Box}_{12}$ given in (\ref{eq:weightCartesianII}), satisfies
$$CD^{\widehat{\sigma}_{12}}\left(\min_{(x,y)\in V_1\times V_2}\frac{\mu_1(x)\mu_2(y)}{\mu_{12}(x,y)}(K_1\wedge K_2), n_1+n_2\right).$$
\end{thm}

\subsection{Cheeger constants on Cartesian products} In this subsection, we discuss relations between the Cheeger constants on two graphs and on their Cartesian products.

Recall (\ref{eq:normrotationinvariant}), i.e., for any $(d\times d)$-matrix $A$ and any $B\in O(d)$ or $U(d)$, their average $(2,1)$-norm satisfies
\begin{equation}\label{eq:normtensorization}
|B A|_{2,1}=|A|_{2,1}.
\end{equation}
This ensures the following relation between the Cheeger constants.

\begin{thm}\label{thm:CartesianCheeger}
	Let $(G_1, 1_{\mathbf{V_1}}, \sigma_1)$ and $(G_2, 1_{\mathbf{V_2}}, \sigma_2)$ be two graphs with $\sigma_i: E_i^{or}\to~H_i,$ $i=1,2$. Assume that they satisfy
	$$CD^{\sigma_1}(K_1,n_1)\,\,\text{ and }\,\,CD^{\sigma_2}(K_2, n_2),$$ respectively. Suppose that $H_1$ and $H_2$ are embedded in a group $H_{12}$ such that $H_1$ and $H_2$ commute. Define $\sigma_{12}$ as in (\ref{eq:signCartesianI}). Then the $kl$-way Cheeger constant $h_{kl}^{{\sigma}_{12}}$ of their Cartesian product graph $(G_1\times G_2, 1_{\mathbf{V_1\times V_2}}, {\sigma}_{12})$, with the edge weight $w_{12}$ given in (\ref{eq:weightCartesianI}), satisfies
	$$h_{kl}^{{\sigma}_{12}}\leq h_k^{\sigma_1}+h_l^{\sigma_2}.$$
	
\end{thm}


We first show the following lemma.
\begin{lemma}\label{lemma:cartesianFrustration}
For any subsets $S_i\subseteq V_i,\,i=1,2$, we have
\begin{equation}
\iota^{\sigma_{12}}(S_1\times S_2)\leq |S_1|\iota^{\sigma_2}(S_2)+|S_2|\iota^{\sigma_1}(S_1).
\end{equation}
\end{lemma}
\begin{proof}
Let $\tau_i: S_i\to H_i$ be the function that achieves the frustration index $\iota^{\sigma_i}(S_i)$. Set $\tau:=\tau_1 \tau_2: S_1\times S_2\to H_{12}$. Then, by definition, we calculate
\begin{align*}
\iota^{{\sigma}_{12}}&(S_1\times S_2)\\
\leq &\sum_{y\in S_2}\sum_{\{x,x'\}\in E_{S_1}}w_{1,xx'}|\sigma_{1,xx'}\tau_1(x')\tau_2(y)-\tau_1(x)\tau_2(y)|_{2,1}\\
&+\sum_{x\in S_1}\sum_{\{y,y'\}\in E_{S_2}}w_{2,yy'}| \sigma_{2,yy'}\tau_1(x)\tau_2(y')-\tau_1(x)\tau_2(y)|_{2,1}\\
=&|S_2| \sum_{\{x,x'\}\in E_{S_1}}w_{1,xx'}|\sigma_{1,xx'}\tau_1(x')-\tau_1(x)|_{2,1}\\
&+|S_1| \sum_{\{y,y'\}\in E_{S_2}}w_{2,yy'}|\sigma_{2,yy'}\tau_2(y')-\tau_2(y)|_{2,1}.
\end{align*}
In the last equality, we used that $H_1$ and $H_2$ commute and (\ref{eq:normtensorization}). This implies the lemma immediately.
\end{proof}
\begin{proof}[Proof of Theorem \ref{thm:CartesianCheeger}]
For any two subsets $S_i\subseteq V_i, \, i=1,2$, it is straightforward to check that
\begin{equation*}
|E(S_1\times S_2, V_1\times V_2\setminus S_1\times S_2)|\leq |S_2||E(S_1, V_1\setminus S_1)|+|S_1||E(S_2, V_2\setminus S_2)|.
\end{equation*}
Combining this with Lemma \ref{lemma:cartesianFrustration}, and using the fact $|S_1\times S_2|=|S_1||S_2|$, we obtain
\begin{equation*}
\phi^{{\sigma}_{12}}(S_1\times S_2)\leq \phi^{\sigma_1}(S_1)+\phi^{\sigma_2}(S_2).
\end{equation*}
Then the theorem follows immediately, by definition, since every nontrivial $k$-subpartition $\left\{S_1^{(i)}\right\}_{i=1}^k$ of $V_1$ and every nontrivial $l$-subpartition $\left\{S_2^{(j)}\right\}_{j=1}^l$ of $V_2$ induce a nontrivial $kl$-subpartition $\left\{S_1^{(i)} \times S_2^{(j)}\right\}_{i=1,j=1}^{k,l}$ of $V_1 \times V_2$.
\end{proof}

Note that Theorem \ref{thm:CartesianCheeger} can be applied in the following particular case: Given two signatures $\sigma_i: E_i^{or}\to~H_i,$ $i=1,2$, we can embed them into $H_1\otimes H_2$ by identifying $H_1$ with $H_1\otimes I_{d_2}$ and $H_2$ with $I_{d_1}\otimes H_2$, respectively. Then the signature $\sigma_{12}$ on $G_1\times G_2$ coincides with the signature $\widehat{\sigma}_{12}$ given in (\ref{eq:signCartesianII}).

For graphs with nonconstant vertex measures, we can extend the above proof to obtain the following result.
\begin{thm}\label{thm:CartesianCheegerII}
Let $(G_1, \mu_1, \sigma_1)$ and $(G_2, \mu_2, \sigma_2)$ be two graphs with $\sigma_i: E_i^{or}\to~H_i,$ $i=1,2$. Then the $kl$-way Cheeger constant $h_{kl}^{\widehat{\sigma}_{12}}$ of their Cartesian product graph $(G_1\times G_2, 2\mu_{1}\mu_2, \widehat{\sigma}_{12})$, with the edge weight $w^{\Box}_{12}$ given in (\ref{eq:weightCartesianII}), satisfies
$$h_{kl}^{\widehat{\sigma}_{12}}\leq \frac{1}{2}(h_k^{\sigma_1}+h_l^{\sigma_2}).$$
\end{thm}

\section{\\Frustration index and spanning trees}\label{section:appendixSpanningTree}
In Section \ref{subsection:spanning tree}, we showed that, in the case $H=U(1)$, there is an easier way to calculate the frustration index of a subset $S\subseteq V$. Recall that the frustration index $\iota^\sigma(S)$ is defined as the minimum of $\sum_{\{x,y\}\in S}w_{xy}|\sigma_{xy}\tau(y)-\tau(x)|_{2,1}$ over all possible switching functions $\tau$ on $S$. Theorem \ref{thm:spanning tree} tells that it is enough to take the minimum of
\begin{equation}\label{eq:appendixFrustration}
\sum_{\{x,y\}\in S}w_{xy}|\sigma_{xy}\tau_T(y)-\tau_T(x)|_{2,1}
\end{equation}
over all spanning trees $T$ of $S$, where $\tau_T$ is an arbitrary representative of the set
$$C_T(S)=\{\tau: S\to U(1): \tau \,\,\text{ is constant on }\,\,T\,\,\text{w.r.t.}\,\,\sigma\}.$$
Recall that (\ref{eq:appendixFrustration}) is well defined because for any two $\tau_1,\tau_2\in C_T(S)$, there exists $z\in U(1)$, such that $\tau_1=\tau_2z$, and hence
\begin{equation}\label{eq:appendixnorm}
 |\sigma_{xy}\tau_1(y)-\tau_1(x)|_{2,1}=|(\sigma_{xy}\tau_2(y)-\tau_2(x))z|_{2,1}=|\sigma_{xy}\tau_2(y)-\tau_2(x)|_{2,1}.
\end{equation}
That is, the quantity $|\sigma_{xy}\tau_T(x)-\tau_T(y)|_{2,1}$ does not depend on the choice of $\tau_T\in C_T(S)$.

 It is natural to ask whether Theorem \ref{thm:spanning tree} can be generalized to higher dimensional signatures, i.e., $H=U(d)$, for $d\geq 2$. We first observe that, for the signature $\sigma: E^{or}\to U(d)$, $d\geq 2$, the quantity (\ref{eq:appendixFrustration}) is not well defined since it depends on the representatives! Note that, for any $(d\times d)$-matrices $A$ and $B\in U(d)$, we do not always have $|AB|_{2,1}=|A|_{2,1}$ (recall that $|BA|_{2,1}=|A|_{2,1}$). But this is needed in the verification of (\ref{eq:appendixnorm}).

However, if we use the Frobenius norm $|\cdot|_F$ instead, (\ref{eq:appendixFrustration}) is still well defined. In this section, we present a counterexample to show that Theorem \ref{thm:spanning tree} does not hold for higher dimensional signatures, even if we use the Frobenius norm in the definition of the frustration index.

Recall that Lemma \ref{lemma:1dimgeometry}, which is a statement about the metric space $S^1=U(1)$, plays a crucial role in the proof of Theorem \ref{thm:spanning tree}. This lemma does not generalize to higher dimensional spheres. Already in $S^2$, we have the following counterexample: For three equidistributed points $P_1, P_2, P_3$ on a meridian close to the north pole $N$, we have
$$d(P_1, P_2)+d(P_1, P_3)> d(N, P_1)+d(N, P_2)+d(N, P_3),$$
where $d$ denotes the intrinsic distance in $S^2$. Lifting this example into $U(2)$ by using the Hopf fibration $S^1\to SU(2)\cong S^3\to S^2$, we obtain the following  matrices in $U(2)$:
\begin{equation}
 A_0=\left(
       \begin{array}{cc}
         0 & -1 \\
         1 & 0 \\
       \end{array}
     \right),\,\,
A_j=\left(
      \begin{array}{cc}
        re^{i\alpha_k} & -\sqrt{1-r^2} \\
        \sqrt{1-r^2} & re^{-i\alpha_k} \\
      \end{array}
    \right), \,\,k=1,2,3,
\end{equation}
where $r\in [0,1],\,\alpha_k=2(k-1)\pi/3$. We check that
\begin{equation*}
 |A_k-A_l|_{F}=\sqrt{6}r,\,\,\forall\,\,1\leq k\neq l\leq 3,
\end{equation*}
and
\begin{equation*}
 |A_0-A_k|_{F}=2\sqrt{1-\sqrt{1-r^2}},\,\,\forall\,\,k=1,2,3.
\end{equation*}
Therefore, for small $r$ (e.g., when $r\leq 0.85$),
\begin{equation}\label{eq:counterLemma3.11}
\sum_{k=1}^3|A_0-A_k|_{F}<\min_{A\in \{A_1, A_2, A_3\}}\sum_{k=1}^3|A-A_k|_{F}.
\end{equation}
(\ref{eq:counterLemma3.11}) implies that the generalization of Lemma \ref{lemma:1dimgeometry} does not hold in $U(2)$.

For later purposes, we transform one of the matrices $\{A_1,A_2, A_3\}$ to be the identity matrix $I_2$. Set
\begin{equation*}
 B_0=A_3^{-1}A_0, B_k=A_3^{-1}A_k, \,\,k=1,2,3.
\end{equation*}
Then we have $B_3=I_2$. Using the definition of the Frobenius norm, we obtain
\begin{equation}\label{eq:equilateral}
  |B_1-B_2|_{F}=|B_2-B_3|_{F}=|B_3-B_1|_{F}=\sqrt{6}r,
\end{equation}
and, for small $r$,
\begin{equation}\label{eq:counterLemma3.11B}
\sum_{k=1}^3|B_0-B_k|_{F}<\min_{B\in \{B_1, B_2, B_3\}}\sum_{k=1}^3|B-B_k|_{F}\leq |B_1-I_2|_{F}+|B_2-I_2|_{F}.
\end{equation}

Let us consider the graph shown in Figure \ref{F4}. This is a graph with vertex set $V=\{x,y,z,w\}$, edge set $E=\{\{x,y\},\{y,z\}, \{z,w\}, \{w,x\}, \{y,w\}\}$ and a signature $\sigma: E^{or}\to U(2)$ as shown in the figure.

\begin{figure}[h]
\centering
\includegraphics[width=0.35\textwidth]{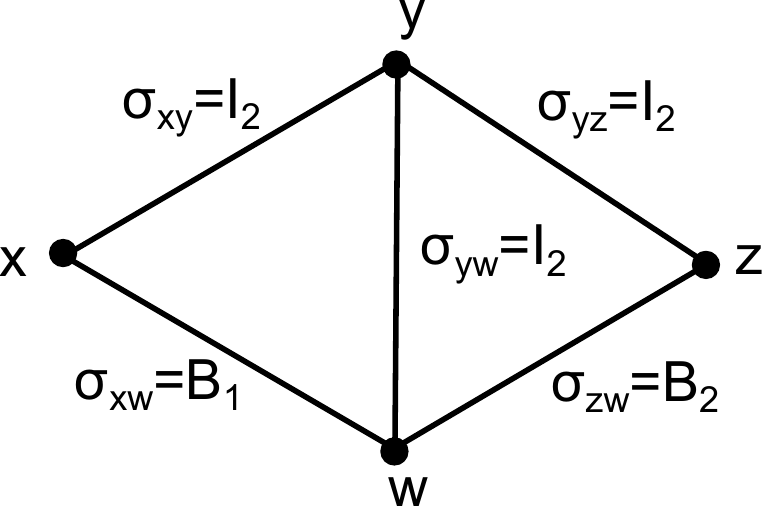}
\caption{A counterexample\label{F4}}
\end{figure}

\begin{pro}
For the graph as shown in Figure \ref{F4}, we have
\begin{equation*}
 \iota_F^\sigma(V):=\min_{\tau:V\to U(2)}\sum_{\{x,y\}\in E_S}w_{xy}|\sigma_{xy}\tau(y)-\tau(x)|_F<\min_{T\in \mathbb{T}_S}\sum_{\{x,y\}\in E}|\sigma_{xy}\tau_T(y)-\tau_T(x)|_{F},
\end{equation*}
where $\tau_T$ is a representative of the set $$C_T(V)=\{\tau: V\to U(2): \tau \,\,\text{ is constant on }\,\,T\,\,\text{w.r.t.}\,\,\sigma\}.$$
\end{pro}

\begin{proof}
Observe that $T_1=(V, \{\{x,y\}, \{y,w\}, \{y,z\}\})$ is a spanning tree and that the function $\tau_{T_1}\equiv I_2$ is a constant function on $T_1$ with respect to $\sigma$. We calculate
\begin{align*}
 \sum_{\{x,y\}\in E}|\sigma_{xy}\tau_{T_1}(y)-\tau_{T_1}(x)|_{F}&=|B_1-I_2|_{2,1}+|B_2-I_2|_{F}\\
&\stackrel{(\ref{eq:counterLemma3.11B})}{>}\sum_{k=1}^3|B_0-B_k|_{F}\\
&=\sum_{\{x,y\}\in E}|\sigma_{xy}\tau_0(y)-\tau_0(x)|_{F},
\end{align*}
where the switching function $\tau_0$ is defined via $\tau_0(x)=\tau_0(y)=\tau_0(z)=B_0,\,\,\tau_0(w)=I_2$. Therefore, by definition, we have
\begin{equation}\label{eq:iotaVT1}
 \iota_F^\sigma(V)<\sum_{\{x,y\}\in E}|\sigma_{xy}\tau_{T_1}(y)-\tau_{T_1}(x)|_{F}.
\end{equation}
The graph in Figure \ref{F4} has $8$ spanning trees, which we denote by $T_i, \,\, i=1,2,\ldots, 8$. We claim that
\begin{equation}\label{eq:T1T8}
 \sum_{\{x,y\}\in E}|\sigma_{xy}\tau_{T_i}(y)-\tau_{T_i}(x)|_{F}=|B_1-I_2|_{F}+|B_2-I_2|_{F} \,\,\forall\,\, i=1,2,\ldots,8.
\end{equation}
The proposition then follows immediately from (\ref{eq:iotaVT1}) and (\ref{eq:T1T8}).

Claim (\ref{eq:T1T8}) can be checked directly with the help of (\ref{eq:equilateral}) for all choices of spanning trees.
\end{proof}

\section*{Acknowledgements}
This work was supported by the EPSRC Grant EP/K016687/1 "Topology, Geometry
and Laplacians of Simplicial Complexes". FM acknowledges the hospitality of the Department of Mathematical Sciences of Durham University during his visit in February 2015. Parts of this work were completed while SL and FM were attending the \emph{Spring School: Discrete Ricci curvature} in Institut Henri Poincar\'{e} (I.H.P.), Paris, during May 18-22, 2015, and while SL and NP were visiting the Max-Planck-Institut f\"{u}r Mathematik, Bonn, during July 1-31, 2015. We thank these institutes for their hospitality.


\begin{thebibliography}{99}
\bibitem{Alon1986} N. Alon, Eigenvalues and expanders, Combinatorica 6
  (1986), no. 2, 83-96.
\bibitem{AM1985} N. Alon and V. Milman, $\lambda_1$, isoperimetric
  inequalities for graphs, and superconcentrators, J. Combin. Theory
  Ser. B 38 (1985), no. 1, 73-88.

\bibitem{AtayLiu14} F. M. Atay, S. Liu, Cheeger constants, structural balance, and spectral clustering analysis for signed graphs, arXiv:1411.3530, 2014.
\bibitem{Bakry} D. Bakry,
Functional inequalities for Markov semigroups, Probability measures on groups: recent directions and trends, 91-147, Tata Inst. Fund. Res., Mumbai, 2006.

\bibitem{BaEm} D. Bakry and M. \'Emery, Diffusions hypercontractives
  (French) [Hypercontractive diffusions], S\'eminaire de
  probabilit\'es, XIX, 1983/84, Lecture Notes in Math. 1123,
  J. Az\'{e}ma and M. Yor (Editors), Springer, Berlin, 1985,
  pp. 177-206.
\bibitem{BSS13} A. S. Bandeira, A. Singer, D. A. Spielman, A Cheeger inequality for the graph connection Laplacian,
SIAM J. Matrix Anal. Appl. 34 (2013), no. 4, 1611-1630.
\bibitem{BHLLMY13} F. Bauer, P. Horn, Y. Lin, G. Lippner, D. Mangoubi, S.-T. Yau, Li-Yau inequality on graphs, J. Differential Geom. 99 (2015), no. 3, 359-405.
\bibitem{BJ} F. Bauer, J. Jost, Bipartite and neighborhood graphs and the spectrum of the normalized graph Laplacian, Comm. Anal. Geom. 21 (2013), no. 4, 787-845.
\bibitem{BJL12} F. Bauer, J. Jost, S. Liu, Ollivier-Ricci curvature and the spectrum of the normalized graph Laplace operator,
Math. Res. Lett. 19 (2012), no. 6, 1185-1205.
\bibitem{BL06} Y. Bilu, N. Linial, Lifts, discrepancy and nearly optimal spectral gap, Combinatorica 26 (2006), no. 5, 495-519.
\bibitem{Buser82} P. Buser, A note on the isoperimetric constant, Ann. Sci. \'{E}cole Norm. Sup. (4) 15 (1982), no. 2, 213-230.
\bibitem{ChungLinYau14} F. R. K. Chung, Y. Lin, S.-T. Yau, Harnack
  inequalities for graphs with non-negative Ricci curvature,
  J. Math. Anal. Appl. 415 (2014), 25-32.\bibitem{ChungTetali98} F. R. K. Chung, P. Tetali, Isoperimetric inequalities for Cartesian products of graphs,
Combin. Probab. Comput. 7 (1998), no. 2, 141-148.
\bibitem{DS09} D. Cvetkovi\'{c},  S. K. Simi\'{c},
Towards a spectral theory of graphs based on the signless Laplacian, I,
Publ. Inst. Math. (Beograd) (N.S.) 85 (99) (2009), 19-33.
\bibitem{Dodziuk1984} J. Dodziuk, Difference equations, isoperimetric inequality and transience of certain random walks, Trans. Amer. Math. Soc. 284 (1984), no. 2, 787-794.
\bibitem{Funano2013} K. Funano, Eigenvalues of Laplacian and multi-way
  isoperimetric constants on weighted Riemannian manifolds,
  arXiv:1307.3919v1, July 2013.
\bibitem{Gross74} J. L. Gross, Voltage graphs, Discrete Math. 9 (1974), 239-246.
\bibitem{GrossTucker74} J. L. Gross, T. W. Tucker, Quotients of complete graphs: revisiting the Heawood map-coloring problem, Pacific J. Math. 55 (1974), 391-402.
\bibitem{Harary53} F. Harary, On the notion of balance of a signed graph, Michigan Math. J. 2 (1953), no. 2,
143-146.
\bibitem{HararyKabell80} F. Harary, J. A. Kabell, A simple algorithm to detect balance in signed graphs, Math. Social Sci. 1 (1980/81), no. 1, 131-136.
\bibitem{Horn14} P. Horn, Y. Lin, S. Liu, S.-T. Yau, Volume doubling, Poincar\'e inequality and Gaussian heat kernel estimate for nonnegative curvature graphs, arXiv:1411.5087v3, 2014
\bibitem{HuaLin15} B. Hua, Y. Lin, Stochastic completeness for graphs with curvature dimension conditions, arXiv:1504.00080, 2015.
\bibitem{JostLiu14} J. Jost, S. Liu, Ollivier's Ricci curvature,
  local clustering and curvature-dimension inequalities on graphs,
 Discrete Comput. Geom. 51 (2014), no. 2, 300-322.
\bibitem{KKRT15} B. Klartarg, G. Kozma, P. Ralli and P. Tetali, Discrete curvature and abelian groups, Canad. J. Math., online first.
\bibitem{KLLGT2013} T.-C. Kwok, L.-C. Lau, Y.-T. Lee, S. Oveis Gharan, L. Trevisan, Improved Cheeger's inequality: Analysis of spectral
partitioning algorithms through higher order spectral gap,
STOC'13-Proceedings of the 2013 ACM Symposium on Theory of
Computing, 11-20, ACM, New York, 2013.
\bibitem{LLPP15} C. Lange, S. Liu, N. Peyerimhoff, O. Post, Frustration index and Cheeger inequalities for discrete and continuous magnetic Laplacians,
	Calc. Var. Partial Differential Equations 54 (2015), no. 4, 4165-4196.
\bibitem{Ledoux04} M. Ledoux, Spectral gap, logarithmic Sobolev
  constant, and geometric bounds, Surveys in differential geometry,
  Vol. IX, 219-240, Surv. Differ. Geom., IX, Int. Press, Somerville,
  MA, 2004.
\bibitem{LinYau} Y. Lin, S.-T. Yau, Ricci curvature and eigenvalue
 estimate on locally finite graphs, Math. Res. Lett. 17 (2010),
  no. 2, 343-356.
\bibitem{Liu13} S. Liu, Multi-way dual Cheeger constants and spectral
  bounds of graphs, Adv. Math. 268 (2015), 306-338.
\bibitem{Liu14} S. Liu, An optimal dimension-free upper bound for
  eigenvalue ratios, arXiv:1405.2213, 2014.
\bibitem{LiuPeyerimhoff14} S. Liu, N. Peyerimhoff, Eigenvalue ratios of nonnegatively curved graphs, arXiv:1406.6617, 2014.
\bibitem{LPV14} S. Liu, N. Peyerimhoff, A. Vdovina, Signatures, lifts and eigenvalues of graphs, arXiv:1412.6841, 2014.
\bibitem{MSS} A. W. Marcus, D. A. Spielman, N. Srivastava, Interlacing families I: bipartite Ramanujan graphs of all degrees, Ann. of Math. 182 (2015), 307-325.
\bibitem{FM14} F. M\"{u}nch, Li-Yau inequality on finite graphs via non-linear curvature dimension conditions, arXiv:1412.3340, 2014.
\bibitem{FM15} F. M\"{u}nch, Remarks on curvature dimension conditions on graphs, arXiv:1501.05839, 2015.
\bibitem{Ollivier09} Y. Ollivier, Ricci curvature of Markov chains on metric spaces,
J. Funct. Anal. 256 (2009), no. 3, 810-864.
\bibitem{Schmuckenschlager98} M. Schmuckenschl\"{a}ger, Curvature of nonlocal Markov generators, Convex geometric analysis (Berkeley, CA, 1996), 189-197, Math. Sci. Res. Inst. Publ., 34, Cambridge Univ. Press, Cambridge, 1999.
\bibitem{Shubin94} M. A. Shubin, Discrete magnetic Laplacian, Comm. Math. Phys. 164 (1994), no. 2, 259-275.
\bibitem{SingerWu} A. Singer, H.-T. Wu, Vector diffusion maps and the connection Laplacian,
Comm. Pure Appl. Math. 65 (2012), no. 8, 1067-1144.
\bibitem{Sunada93} T. Sunada, A discrete analogue of periodic magnetic Schr\"{o}dinger operators, Geometry of the spectrum (Seattle, WA, 1993), 283-299, Contemp. Math., 173, Amer. Math. Soc., Providence, RI, 1994.
\bibitem{Trevisan2012} L. Trevisan, Max cut and the smallest eigenvalue, STOC'09-Proceedings of the 2009 ACM International Symposium on Theory of Computing, 263-271, ACM, New York, 2009;
SIAM J. Comput. 41 (2012), no. 6, 1769-1786.
\bibitem{Zaslavsky82} T. Zaslavsky, Signed graphs, Discrete Appl. Math. 4 (1982), no. 1, 47-74.
\bibitem{ZaslavskyMatrices} T. Zaslavsky, Matrices in the theory of signed simple graphs, Advances in discrete mathematics and applications (Mysore, 2008), 207-229, Ramanujan Math. Soc. Lect. Notes Ser.~13, Ramanujan Math. Soc., Mysore, 2010.
\end{thebibliography}
\end{document}